\documentclass[reqno,11pt]{amsart}
\title{Models For Knot Spaces And Atiyah Duality }

\author{Syunji Moriya}
\thanks{The  author is partially supported by JSPS KAKENHI Grant Number 26800037.}
\address{Department of Mathematics and Information Sciences, Osaka Prefecture University, Sakai, 599-8531, Japan}
\email{moriyasy@gmail.com}
\usepackage{setspace}
\usepackage{amsthm}
\usepackage{amscd}
\usepackage{amsmath,amsthm,amssymb,amscd,mathrsfs,picture}
\usepackage[arrow,matrix]{xy}
\input xy
\xyoption{all}
\usepackage{enumerate}

\theoremstyle{definition}
\newtheorem{defi}{Definition}[section]

\newtheorem{exa}[defi]{Example}
\newtheorem{rem}[defi]{Remark}
\theoremstyle{plain}
\newtheorem{prop}[defi]{Proposition}

\newtheorem{lem}[defi]{Lemma}

\newtheorem{thm}[defi]{Theorem}

\newtheorem{cor}[defi]{Corollary}

\newcommand{\eps}{\epsilon}

\newcommand{\CECH}{\check{\mathrm{T}}}
\newcommand{\CECHC}{\check{\mathrm{C}}}
\newcommand{\CC}{\mathcal{C}}
\newcommand{\CPTM}{\langle [M]\rangle}
\newcommand{\CSINHA}{\mathcal{C}_{\langle[M]\rangle}}
\newcommand{\CCM}{\mathcal{TH}_M}
\newcommand{\TCCM}{\widetilde{\mathcal{TH}}_M}
\newcommand{\MM}{\mathcal{M}}
\newcommand{\CH}{\mathcal{CH}}
\newcommand{\FCW}{f\mathcal{CW}}

\newcommand{\TM}{\widehat{M}}

\newcommand{\MT}{\widehat{M}^{-\tau}}
\newcommand{\TNT}{\widetilde{N}^{-\tau}}
\newcommand{\BNT}{\bar{N}^{-\tau}}
\newcommand{\DeltaT}{\mathcal{TH}}
\newcommand{\tw}{\widehat{w}}
\newcommand{\td}{\tilde{d}}
\newcommand{\fat}{\mathrm{fat}}
\newcommand{\FPsi}{\mathsf{F}_{\Psi}}
\newcommand{\THOM}{\mathsf{T}\mathrm{h}^M}
\newcommand{\C}{\mathcal{C}}
\newcommand{\Sph}{\mathbb{S}}

\newcommand{\TPsi}{\widetilde{\Psi}}

\newcommand{\TB}{\widetilde{B}}

\newcommand{\NT}{N^{-\tau}}
\newcommand{\Map}{\mathrm{Map}}
\newcommand{\oper}{\mathcal{O}}
\newcommand{\ASS}{\mathcal{A}}
\newcommand{\KK}{\mathcal{K}}
\newcommand{\DD}{\mathcal{D}}
\newcommand{\configc}{\mathfrak{c}}
\newcommand{\configd}{\mathfrak{d}}

\newcommand{\E}{\mathcal{E}}
\newcommand{\Ho}{\mathbf{Ho}}
\newcommand{\CAT}{\mathit{Cat}}
\newcommand{\Cat}{\mathcal{C}}
\newcommand{\SP}{\mathcal{SP}^\Sigma}
\newcommand{\Hoch}{\mathrm{CH}}

\newcommand{\hotimes}{\, \hat{\otimes}\, }

\newcommand{\CG}{\mathcal{CG}}

\newcommand{\tot}{\operatorname{Tot}}

\newcommand{\GG}{\mathsf{G}}

\newcommand{\Emb}{\mathrm{Emb}}
\newcommand{\Imm}{\mathrm{Imm}}

\newcommand{\RR}{\mathbb{R}}
\newcommand{\HH}{\mathcal{H}}
\newcommand{\SHH}{S\mathcal{H}}
\newcommand{\Sphere}{\mathbb{S}}

\newcommand{\FF}{\mathbb{F}}
\newcommand{\ZZ}{\mathbb{Z}}
\newcommand{\QQ}{\mathbb{Q}}
\newcommand{\pp}{\mathfrak{p}}

\newcommand{\EE}{\mathbb{E}}
\newcommand{\colim}{\mathrm{colim}}
\newcommand{\fcolim}{\underset{\omega}{\mathrm{fcolim}}}
\newcommand{\hocolim}{\mathrm{hocolim}}
\newcommand{\holim}{\mathrm{holim}}
\newcommand{\fc}{\mathrm{fc}}

\newcommand{\FUN}{\mathcal{F}un}
\newcommand{\LL}{\mathbb{L}}

\newcommand{\CECHF}{\check{\mathsf{C}}}

\newcommand{\kk}{\boldsymbol{k}}
\newcommand{\dd}{\boldsymbol{d}}
\newcommand{\BDelta}{\boldsymbol{\Delta}}

\newcommand{\BGSSS}{\check{\mathbb{E}}}
\newcommand{\SINHASS}{\mathbb{E}}

\allowdisplaybreaks[1]

\begin{document}
\maketitle
\begin{abstract}
Let  $\Emb(S^1,M)$ be the space of smooth embeddings from the circle to a closed manifold $M$ of dimension $\geq 4$. We study a cosimplicial model of $\Emb(S^1,M)$ in stable categories, using  a spectral version of Poincar\'e-Lefschetz duality called  Atiyah duality. We actually deal with a notion of a comodule instead of the cosimplicial model, and prove a comodule version of  the duality  as in Theorem \ref{Tmaincomodule}. As an application, we introduce a new spectral sequence converging to $H^*(\Emb(S^1,M))$ for simply connected $M$ and for major coefficient rings as in Theorem \ref{Tmainss}. Using this, we compute $H^*(\Emb(S^1, S^k\times S^l))$ in low degrees with some conditions on $k$, $l$. We also prove the inclusion $\Emb(S^1,M)\to \Imm(S^1,M)$ to the immersions induces an isomorphism on $\pi_1$ for some simply connected $4$-manifolds, related to   a question posed by Arone and Szymik.  We also prove an equivalence of singular cochain complex of $\Emb(S^1,M)$ and a homotopy colimit of chain complexes of a Thom spectrum of a bundle over a comprehensible space as in Theorem \ref{Tmainhomotopycolimit}.  Our key ingredient is a structured version of the  duality due to R. Cohen. 
\end{abstract}
\tableofcontents

\section{Introduction}
In \cite{sinha,sinha1} Sinha constructed  cosimplicial  models  of  spaces of knots in a manifold of  dimension $\geq 4$, based on the Goodwillie-Klein-Weiss embedding calculus \cite{GK, GW}.  Sinha's model was crucially used in the affirmative solution to  Vassiliev's conjecture for  a spectral sequence for the space of long knots in $\RR^d$ ($d\geq 4$) for rational coefficient in \cite{LTV} (see \cite{BH} for other coefficients).  
In the present paper, we study a version of  Sinha's model in stable categories.  \\
\indent Precisely speaking, our situation  is different from  \cite{sinha, sinha1} in that we consider the space $\Emb (S^1,M)$ of smooth embeddings from the circle $S^1$ to a closed manifold $M$ without any base point condition while in \cite{sinha,sinha1}  spaces of   embeddings from $[0,1]$ to a manifold with boundary with some endpoints conditions are dealt with.  This is for technical simplicity and a similar theory will be valid for Sinha's  spaces. The space $\Emb (S^1,M)$ is studied in \cite{aroneszymik,budney}  and  study of embedding spaces including the knot space is  a  motivation of   \cite{willwacher,idrissi}\\
\indent In the rest of the paper, $M$ denotes a connected closed smooth manifold of dimension $\dd\geq 4$.  We endow the space $\Emb(S^1,M)$ with the $C^\infty$-topology. For our situation, we can construct a cosimplicial model similar to Sinha's, which is  called {\em Sinha's cosimplicial model}  and  denoted by $\CC^\bullet\CPTM$. This  model relates the knot space and the configuration space as follows.
\begin{itemize}
\item There exists a weak homotopy equivalence $\underset{\BDelta}{\holim}\ \CC^\bullet\CPTM\simeq \Emb (S^1,M)$, where  $\underset{\BDelta}{\holim}$ denotes the homotopy limit over the category  $\BDelta$ of standard simplices  (see subsection \ref{SSnt}).
\item $\CC^n\CPTM$ is homotopy equivalent to the configuration space $C'_{n+1}(M)$ of $n+1$ points with a tangent vector  defined as the pullback of the following diagram.
\[
\TM^{\times n+1}\longrightarrow M^{\times n+1}\longleftarrow C_{n+1}(M).
\] 
Here $\TM$ denotes the  tangent sphere bundle of $M$ and $C_{n+1}(M)$ denotes the ordered configuration space of $n+1$-points in $M$.  The left arrow is the product of the projection and the right arrow is the inclusion. Actually, $\CC^n\CPTM$ is a version of the Fulton-MacPherson compactification.
\end{itemize} 
\indent  To state our main theorems, we need some notations. Fix an embedding $e_0:\TM\to\RR^K$ and a tubular neighborhood $\nu$ of the image $e_0(\TM )$ in $\RR^K$. We regard the product $\nu^{\times n}$ as a disk bundle over $\TM^{\times n}$ and let $\nu^{\times n}|_{X}$ denote the restriction of the base to a subspace $X\subset \TM^{\times n}$.   For a positive integer $n$, let $\GG(n)$ be the set of graphs $G$ with set of vertices $V(G)=\underline{n}=\{1,\dots, n\}$ and set of edges $E(G)\subset \{(i,j)\mid i,j\in \underline{n},\ i<j\}$. Let $\pi_0(G)$ be the set of connected components of $G$. A  map $M^{\times \pi_0(G)}\to M^{\times n}$ is induced by the quotient map $\underline{n}\to \pi_0(G)$, considering  $M^{\times A}$ as the set of maps $A\to M$.   Let $\Delta_G$ be  the pullback of the following diagram 
\[
\TM^{\times n}\stackrel{\text{projection}}{\longrightarrow} M^{\times n}\longleftarrow M^{\times \pi_0(G)}.
\]
$\Delta_G$ is naturally regarded as a subspace of $\TM^{\times n}$ via the projection of the pullback. $\Delta_G$ is a rather comprehensible space, comparing to the space  $\CC^{n-1}\CPTM$ (or $C'_{n}(M)$). For example, its cohomology ring is computed in Lemmas \ref{Lahg}, \ref{Lbhg} under some assumption. Define a subspace $\Delta_\fat(n)\subset \TM^{\times n}$ as the union $\cup_{G\in\GG(n)}\Delta_G$.   We use a notion of right modules over an operad, which is similar to the notion in Loday-Vallette \cite{LV}. Let  $\DD_1$ be a version of the little intervals operad (see Definition \ref{Dlittleinterval}).  We define a right $\DD_1$-module $\CSINHA$ which shares a large part of structure with $\CC^\bullet \CPTM$ and is quite analogous to the modules considered in \cite{AT,BW} (see Definitions \ref{Dcontra-module},  \ref{Dconfigurationmodule}). For example,  $\CSINHA(n)=\CC^{n-1}\CPTM$ by definition. $\CSINHA$ also has information enough to reconstruct $\Emb(S^1,M)$. We work with the category of symmetric spectra $\SP$. $(\CSINHA)^\vee$ denotes a left $\DD_1$-comodule in $\SP$ given by $(\CSINHA)^\vee (n)=(\CSINHA(n))^\vee$, where $(-)^\vee$ denotes the Spanier-Whitehead dual. See subsection \ref{SSnt} and Definition \ref{Dsemistablemodule} for  other notations in the following theorem.\newpage
\begin{thm}[Theorem \ref{TAtiyahdual}]\label{Tmaincomodule}
There exists a left $\DD_1$-comodule $\CCM$ of non-unital commutative symmetric ring spectra   as follows.
\begin{enumerate}
\item The object $\CCM(n)$ at arity $n$ is a natural model of the Thom spectrum 
\[
\Sigma^{-nK}Th(\nu^{\times n})/Th(\nu^{\times n}|_{X})\quad \text{with}\quad X=\Delta_{\fat}(n) ,
\] where $\Sigma$ denotes the suspension equivalence and $Th(-)$ the associated Thom space.  Concretely speaking,  $\CCM(n)$ is a relative version of R. Cohen's non-unital model in \cite{cohen}.
\item There exists 
a zigzag  of $\pi_*$-isomorphisms of  left $\DD_1$-comodules 
\[
(\CSINHA )^{\vee}\simeq \CCM\, .
\]
\end{enumerate}
\end{thm}
 Theorem \ref{Tmaincomodule} is a structured version of the Poincar\'e-Lefschetz duality
\[
H^*(C'_n(M))\cong H_*(\TM^{\times n},\Delta_{\fat}(n) )\cdots\cdots (*), 
\]
deduced from the equality $C_n'(M)=\TM^{\times n}-\Delta_{\fat}(n)$. (We are loose on degrees.)
If we do not consider  the (non-unital) commutative multiplications, an analogue to Theorem \ref{Tmaincomodule} holds in the category of prespectra  (in the sense of  \cite{MMSS}), a more naive, non-symmetric monoidal category of spectra  and it  is enough to prove Theorem \ref{Tmainss} below, but the multiplications may be useful for future study and our constructions hardly become easier for prespectra.\\   
\indent Throughout this paper, we fix a coefficient ring $\kk$ and suppose $\kk$ is either of a subring of the rationals $\QQ$ or the field $\FF_{\pp}$ of $\pp$ elements for a prime $\pp$. All normalized singular (co)chains $C^*,\ C_*$ and singular (co)homology $H^*, H_*$ are supposed to have coefficients in $\kk$, unless otherwise stated. As an application of Theorem \ref{Tmaincomodule}, we introduce a new spectral sequence converging to $H^*(\Emb(S^1,M))$. 
\begin{thm}[Theorems \ref{Tam}, \ref{Tconvergence} and \ref{Talgebraicss}]\label{Tmainss}
Suppose $M$ is  simply connected. There exists a second-quadrant spectral sequence $\{ \BGSSS^{\,p\,q}_r\}_r$ converging to $H^{p+q}(\Emb (S^1,M))$ such that
\begin{enumerate}
\item its $E_2$-page is isomorphic to the total homology of the normalization of a simplicial commutative differential bigraded algebra $A^{\star\,*}_\bullet(M)$ which is defined in terms of  the cohomology ring $H^*(\Delta_G)$ for various $G$ and maps between them, 
\[
\BGSSS^{\,p\,q}_2\cong  H(NA_\bullet^{\star\, *}(M))\Rightarrow H^{p+q}( \Emb(S^1,M)) \, ,\ 
\]
where the bidegree is given by $*=p, \ \star-\bullet=q$,
\item and moreover, if $H^*(M)$ is a free $\kk$-module, and the Euler number $\chi(M)$ is zero or invertible in $\kk$, the object $A^{\star\,*}_\bullet(M)$ is determined by the  ring $H^*(M)$.
\end{enumerate} 
\end{thm}
 We call this spectral sequence  {\em $\check{C}$ech spectral sequence} or in short, {\em $\check{C}$ech s.s.}  A feature of $\CECHC$ech s.s. is that its $E_1$ page and differential $d_1$ are explicitly determined by the cohomology of $M$. As  spectral sequences for $H^*(\Emb(S^1,M))$, we have the Bousfield-Kan type cohomology spectral sequence  converging to $H^*(\Emb(S^1,M))$, see Definition \ref{Dcosimplicial}, and  Vassiliev' s spectral sequence converging to the relative cohomology $H^*(\Omega_f(M),\Emb(S^1,M))$, where $\Omega_f(M)$ is the space of smooth maps $S^1\to M$. But no small (i.e. degree-wise finite dimensional) page of  these spectral sequences have been computed in general. $E_1$-page of the Bousfield-Kan type s.s. described by the cohomology of the ordered configuration spaces of points with a vector in $M$, which is difficult to compute, and Vassiliev's first term is also interesting but complicated.   By this feature, we can compute examples, see  section \ref{Sexample}. We obtain   new computational results in the case of the product of two spheres. While we only do elementary computation in the present paper, one of potential merits of $\CECHC$ech s.s. is that computation of higher differentials will be relatively accesible since we deal with  fat diagonals and $\CECHC$ech complex instead of configuration spaces. The other is that we will be able to enrich it with operations such as the cup product and square, and relate them to those on $H^*(M)$. We will deal with these subjects  in a future work. \\
\indent  Arone and Szymik studied $\Emb(S^1,M)$ for the case of dimension $\dd=4$ in \cite{aroneszymik}.  Let $\Imm (S^1,M)$ be the space of smooth immersions $S^1\to M$ with the $C^\infty$-topology and $i_M:\Emb(S^1,M)\to \Imm(S^1,M)$ be the inclusion. Among other results, they proved that $i_M$ is $1$-connected, so in particular,   surjective on $\pi_1$ in general.  (They proved interesting results for the non-simply connected case  $M=S^1\times S^3$, see also  Budney and Gabai  \cite{BG1}.) They   asked whether there is a simply connected 4-manifold $M$ such that  $i_M$ has non-trivial kernel on $\pi_1$.  Using Theorem \ref{Tmainss}, we give a restriction to this question. 
 \begin{cor}\label{Tmain4dim}
 Suppose that $M$ is simply connected,  $\dd=4$,   $H_2(M;\ZZ)\not=0$, and the intersection form on $H_2(M;\FF_2)$  is represented by a matrix  of which the inverse has at least one non-zero diagonal component. Then, the inclusion $i_M$ induces an isomorphism on $\pi_1$. In particular, $\pi_1(\Emb(S^1,M))\cong H_2(M;\ZZ)$.
 \end{cor}
 The assumption does not depend on choices of a matrix. For example, $M=\mathbb{C}P^2\#\mathbb{C}P^2$, the connected sum of complex projective planes, satisfies the assumption  while $M=S^2\times S^2$ does not. 
For the case of $H_2(M)=0$, By Proposition 5.2 of \cite{aroneszymik},  $\Emb(S^1,M)$ is simply connected. We cannot prove this completely similarly to Corollary \ref{Tmain4dim}. We might need a relation of $\CECHC$ech s.s.  and the space of long knots. The case of all of the diagonal components of the matrix being zero is unclear for the author. \\
\indent Sinha's cosimplicial model can be considered as a resolution of $\Emb(S^1,M)$ into simpler spaces. We  resolve it into further simpler pieces in the category of chain complexes as an application of Theorems \ref{Tmaincomodule} and  \ref{Tmainss}.  To state the result, we need additional notations. Let $\Psi_n^o$ be the category of planer trees given in \cite{sinha} (see Definition \ref{Dfunctor}).  We regard the set $\GG(n)$ as a category (poset) having  a  morphism $G\to H$ for $E(G)\subset E(H)$. We denote by $\emptyset$ the graph in $\GG(n)$ with no edge. Let $\GG(n)^+$ be the poset made by adding an object $*$ to $\GG(n)$ and a morphism $*\to G$ for each graph $G\not=\emptyset$. 
Let $\TPsi$ be a category of pairs $(T,G)$ of a tree $T\in \cup_n\Psi_n^o$ and an object $G\in \GG(\,|v_r|-1)^+$, where $|v_r|$ denotes the valence of the root vertex of $T$ (see Definition \ref{Dfunctor}). Let $\CH_{\kk}$ be the category of chain complexes and chain maps, and $\hocolim$ be the homotopy colimit, and $C^S_*:\SP\to \CH_{\kk}$ be a functor, see subsection \ref{SSnt} and Definition \ref{Dchainspectra}.
\begin{thm}[Theorem \ref{Tthomcolimit}]\label{Tmainhomotopycolimit}
There exists a functor $\THOM:(\TPsi)^{op}\to \SP$ such that
\begin{enumerate}
\item for each $(T,G)\in \TPsi$, if $G\not=*$, $\THOM(T,G)$ is a natural model of the Thom spectra 
\[
\Sigma^{-mK}Th(\nu^{\times m}|_{\Delta_G} ) \qquad \text{with} \qquad m=|v_r|-1,
\] and if $G=*$\,, $\THOM(T,G)=*$ and
\item  if $M$ is  simply connected, there exists a zigzag of quasi-isomorphisms of chain complexes
\[
C^*(\Emb(S^1,M))\simeq \underset{(\TPsi)^{op}}{\hocolim}\  C^S_*\circ \THOM .
\]

\end{enumerate}
\end{thm}
We give intuitive explanation for this theorem. There is a standard  quasi-isomorphism $C_*(\Delta_{\fat}(n))\simeq \underset{G\in C_1}{\hocolim}\, C_*(\Delta_G)$ where $C_1=\GG(n)^{op}-\{\emptyset\}$. Since the relative complex $C_*(\TM^{\times n}, \Delta_{\fat})$ is a homotopy cofiber of the inclusion $C_*(\Delta_{\fat})\to C_*(\TM^{\times n})(=C_*(\Delta_{\emptyset}))$, we have   quasi-isomorphisms  
\[
C^*(\CC^{n-1}\CPTM)\simeq C_*(\TM^{\times n}, \Delta_{\fat}(n))\simeq \underset{G\in C_2}{\hocolim}\, C_*(\Delta_G),
\] 
where  $C_2=(\GG(n)^+)^{op}$ and we set $C_*(\Delta_*)=0$. We regard this presentation as a resolution of $C^*(\CC^{n-1}\CPTM)$. The category $\cup_n\Psi^o_n$ is a lax analogue of the category $\BDelta$. Actually, homotopy limits over these categories are weak equivalent. So, intuitively speaking,  existence of the functor $\THOM$  means potential compatibility of the resolution and the cosimplicial structure.  \\
\indent We shall explain why the author uses spectra, which also serves as  an outline of the arguments in the paper.  The author's  motivation is to derive  a new spectral sequence from Sinha's cosimplicial model.  
 The idea is to combine the cosimplicial model and  a procedure of  constructing a spectral sequence for $H^*(C_n(M))$ due to Bendersky and Gitler \cite{BG}. So we consider the above  duality $(*)$
, and describe the chain complex $C_*(\TM^{\times n},\Delta_{\fat}(n))$ by an augmented  $\CECHC$ech complex as follows.
\[
C_*(\Delta_\emptyset)\stackrel{\partial}{\longleftarrow}\underset{G\in \GG(n,1)}{\oplus}C_*(\Delta_G )\stackrel{\partial}{\longleftarrow}\underset{G\in \GG(n,2)}{\oplus}C_*(\Delta_G )\stackrel{\partial}{\longleftarrow}\underset{G\in \GG(n,3)}{\oplus}C_*(\Delta_G )\stackrel{\partial}{\longleftarrow}\cdots, 
\]
where   $\GG(n,p)\subset \GG(n)$ denotes the subset of graphs  with  $p$ edges.
We want to extend this diagram to the following commutative diagram $(**)$ of semisimplicial chain complexes by defining suitable face maps $d_i$. (Here, 'semi' means lack of degeneracy maps.)
\[
\xymatrix{
C^*(\CC^{n}\CPTM)\ar@<0.5ex>[d]^{(d^i)^*}  & C_*(\Delta_\emptyset) \ar[d]^{d_i}\ar[l]_{\qquad P.D.}&  \underset{G\in \GG(n+1,1)}{\oplus}C_*(\Delta_G )\ar[l] \ar[d]^{d_i} &  \underset{G\in \GG(n+1,2)}{\oplus}C_*(\Delta_G )\ar[l] \ar[d]^{d_i}  &    \ar[l]     \cdots \\
C^*(\CC^{n-1}\CPTM)  & C_*(\Delta_\emptyset) \ar[l]_{\qquad P.D.} &  \underset{G\in \GG(n,1)}{\oplus}C_*(\Delta_G )\ar[l]  &  \underset{G\in \GG(n,2)}{\oplus}C_*(\Delta_G )\ar[l]  &     \ar[l]     \cdots. 
}
\]
where $d^i$ is the coface map of $\CC^\bullet\CPTM$, and  $P.D.$ actually denotes the zigzag 
\[
C^*(\CC^n\CPTM)\longrightarrow  C_*(\Delta_\emptyset, \Delta_{\fat})\longleftarrow C_*(\Delta_\emptyset)
\]
of the cap product with the fundamental class and the quotient map. 
If we could construct a semisimplicial double complex in the right hand side  of $P.D.$ in the  diagram $(**)$, by taking the total complex, we have certain triple complex $C_{\bullet\, \star\, *}$, where $\bullet$ (resp. $\star$, $*$) denotes the cosimplicial (resp. $\CECHC$ech, singular) degree. Then by filtering with $\star +\bullet$, we would obtain a spectral sequence as in Theorem \ref{Tmainss}.\\
\indent  Unfortunately, it is difficult for the author to define  degeneracy maps $d_i$  fitting into the diagram $(**)$.  This difficulty is essentially analogous to the one in construction of certain  chain-level intersection product on $C_*(M)$.   We shall explain this point more precisely. 
The coface map $d^i:\CC^n\CPTM\to \CC^{n+1}\CPTM$  is a deformed diagonal and  the usual diagonal induces the intersection product on homology. So the maps $d_i$ should be something like a deformed intersection product. The simplicial identities for $d_i$ are analogous to the associativity of  an intersection  product. In addition, the map $(d^i)^*$ on the cochain is analogous to the cup product. So construction of   $d_i$ is analogous to construction of chain-level intersection product which is associative and compatible with the cup product through the duality.
 The author could not find such a product in the literature.\\
\indent A nice solution to this  is found in a construction due to R. Cohen and Jones \cite{cohen, CJ} in string topology. They used spectra  to give a homotopy theoretic realization of the loop product, which led to a proof of an isomorphism between the loop product and a product  on  Hochschild cohomology  (see \cite{moriya} for a detailed account). Their key notion is the Atiyah duality which is an equivalence of the Spanier-Whitehead dual $M^\vee$  and the Thom spectrum $M^{-TM}=\Sigma^{-K}Th(\nu)$.   To prove their isomorphism, Cohen \cite{cohen} introduced a model of $M^{-TM}$  in the category  $\SP$, and refine the  duality to an equivalence of (non-unital) commutative symmetric ring spectra.  This equivalence can be regarded  as a multiplicative  version of the Poincar\'e duality. In fact, the multiplication on the model of $M^{-TM}$ works as an analogue of a chain level intersection product in their theory. So the author considers that it is efficient to construct necessary semisimplicial objects (or comodules) and their equivalence in $\SP$, then send them to $\CH_{\kk}$ by certain chain functor, and derive a spectral sequence. This is why we use spectra. \\
\indent Even if we use spectra, the (co)simplicial object is too rigid, and we use a laxer notion of  a left comodule  over an $A_\infty$-operad. \\
\indent As demonstrated in this paper, the duality is very useful to transfer structures on a configuration space to a Thom spectrum ( or  complex ) of fat diagonal, which is homotopically more accessible, and may be applied in many researches on configuration spaces. In future work, we will study collapse of Sinha's ( or Vassiliev's) spectral sequence for the space of long knots in $\RR^d$ \cite{sinha1} using the duality. \\

\noindent \textbf{Acknowledgment}:\  The author  thanks Keiichi Sakai and Tadayuki Watanabe for  valuable comments on a version of this paper. Their comments motivated him to consider the $4$-dimensional case. He  also thanks Ryan Budney for pointing out errors in previous version of this paper. He also thanks Dev P. Sinha, and Victor Turchin for valuable comments.  This work is supported by JSPS KAKENHI Grant Number JP17K14192.  

\section{Preliminaries}
In this section, we fix notations and introduce basic notions. Nothing is essentially new. 
\subsection{Notations and terminologies}\label{SSnt}
\begin{itemize}

\item We borrow the following notations from \cite{sinha}. $\BDelta$ denotes the category of standard simplices in \cite[Definition 4.21]{sinha}. Its objects  are  the  finite ordered sets $[n]=\{0,\dots,n\}$ and morphisms are weakly order preserving maps. $\BDelta_n$ denotes the full subcategory of $\BDelta$ consists of objects $[k]$ with $k\leq n$. $P_\nu(\underline{n})$ is the category (poset) of non-empty subsets $S\subset \underline{n}$ in \cite[Definition 3.2]{sinha}.  A functor $\mathcal{G}_n:P_\nu(\underline{n+1})\to \BDelta_n$ is defined in \cite[Definition 6.3]{sinha}. $\mathcal{G}_n$ sends a set $S$ to $[\# S-1]$ and an inclusion $S\subset S'$ to the composition $[\# S-1]\cong S\subset S'\cong[\# S'-1]$ where $\cong$ denotes the order preserving bijection.
\item For a category $\CC$, a morphism of $\CC$ is also called a {\em map of $\CC$}. A {\em symmetric sequence in $\CC$} is a sequence $\{X_k\}_{k\geq 0}$( or $\{X(k)\}_{k\geq 1}$) of objects in $\CC$ equipped with an action of the $k$-th symmetric group $\Sigma_k$ on $X_k$ (or $X(k)$) for each $k$. The group $\Sigma_k$  acts from the right throughout this paper. 
\item  Let $\GG(n)$ be the set of graphs  defined in Introduction. For a graph $G\in \GG(n)$, we regard $E(G)$ as an ordered set with the lexicographical order. To ease notations, we  use the notation  $(i,j)$ with $i>j$ to denote the edge $(j,i)$ of a graph in $\GG(n)$.  For a map $f:\underline{n}\to \underline{m}$ of finite sets, denote by the same symbol $f$  the map
$\GG(n)\to \GG(m)$ defined by 
\[
E(f(G))=\{(f(i),f(j))\mid (i,j)\in E(G),\ f(i)\not=f(j)\}.
\] $f$ also denotes the natural map $\pi_0(G)\to \pi_0(f(G))$.

\item Our notion of a {\em model category} is that of \cite{hovey}. $\Ho(\MM)$ denotes the homotopy category of a model category $\MM$.
\item $\CG$ denotes the category  of all compactly generated spaces and continuous maps, see \cite[Definition 2.4.21]{hovey}. $\CG_*$ denotes the category of pointed compactly generated spaces and pointed maps.  $\wedge$ denotes the smash product of pointed spaces.
\item For a category $\CC$, a {\em cosimplicial object $X^\bullet$ in $\CC$ } is a functor $\BDelta\to \CC$. A map of cosimplicial object is a natural transformation. $X^n$ denotes the object of $\CC$ at $[n]$. We define maps
\[
d^i:[n]\to [n+1]\ (0\leq i\leq n+1),\ \qquad s^i: [n]\to [n-1]\ (0\leq i\leq n-1)
\]
by 
\[
d^i(k)=\left\{
\begin{array}{cc}
k & (k<i) \\
k+1 & (k\geq i)
\end{array}\right.,\qquad
s^i(k)=\left\{
\begin{array}{cc}
k & (k\leq i) \\
k-1 & (k>i)
\end{array}\right.
\]
 $d^i,\ s^i:X^n\to X^{n\pm 1}$ denote the maps corresponding to the same symbol. As is well-known, a cosimplicial object $X^\bullet$ is identified with a sequence of objects $X_0,X_1,\dots, X_n,\dots$ equipped with a family of maps $\{ d^i,\ s^i\}$ satisfying the cosimplicial identity, see \cite{GJ}. We call a cosimplicial object in $\CG$ a {\em cosimplical space}. Similarly, a simplicial object $X_\bullet$ in $\CC$ is a functor $\BDelta^{op}\to \CC$. $d_i,\ s_i:X_{n\pm 1}\to X_n$ denote the maps corresponding to $d^i,s^i$.
\item Our notion of a {\em symmetric spectrum} is that of Mandell-May-Schwede-Shipley\cite{MMSS}. A symmetric spectrum consists of a symmetric sequence $\{X_k\}_{k\geq 0}$  in  $\CG_*$  and a map $\sigma_X:S^1\wedge X_k\to X_{k+1}$ for each $k\geq 0$ which subject to certain conditions. The category of symmetric spectra is denoted by $\SP$. We denote by $\wedge =\wedge_S$ the cannonical symmetric monoidal product on $\SP$ given in \cite{MMSS} and by $\Sphere$ the sphere spectrum, the unit for $\wedge$.  In the rest of the paper the term 'spectrum' means symmetric spectrum. For a  spectrum  we refer to  the numbering of the underlying sequence as the  level. 
\item For $K\in \CG$ and $X\in \SP$, we define a tensor $K\hotimes X\in \SP$ by $(K\hotimes X)_k=(K_+)\wedge X_k$ where $K_+$ is $K$  with disjoint basepoint. This tensor is  extended to a functor $\CG\times \SP\to \SP$ in an obvious manner. For $K, L\in \CG$ and $X, Y\in \SP$,  there are natural isomorphisms
\[
K\hotimes (L\hotimes X)\cong (K\times L)\hotimes X,\quad K\hotimes (X\wedge Y)\cong (K\hotimes X)\wedge Y
\]
which we call the {\em associativity isomorphisms}. A natural isomorphism \\
$(K\times L)\hotimes (X\wedge Y)\cong (K\hotimes X)\wedge (L\hotimes Y)$ is defined by 
successive composition of the associativity isomorphisms and the  symmetry one for $\wedge$. 
We define a mapping object $\Map(K, X)\in \SP$ by $\Map(K,X)_k=\Map_*(K_+,X_k)$, where the right hand side is the usual internal hom object (mapping space) of $\CG_*$. This defines a functor $(\CG)^{op}\times \SP\to \SP$. $K\hotimes (-)$ and $\Map(K,-)$ forms an adjoint pair. We set $K^{\vee}=\Map (K,\Sphere)$ for $K\in \CG$.   
\item We use  the {\em stable model structure on $\SP$}, see \cite{MMSS}. This is only used in subsection \ref{SSchain} and section \ref{Shomotopycolim}. Weak equivalences in this model structure are called {\em stable equivalences}. {\em Level equivalences} and {\em $\pi_*$-isomorphisms} are more restricted classes of maps in $\SP$ (see \cite{MMSS}). The former are  levelwise weak homotopy equivalences and the latter are  maps which induce an isomorphism between (naive) homotopy groups defined as the colimit of the sequence of  canonical maps $\iota_k:\pi_*(X_k)\to \pi_{*+1}(X_{k+1})$. 
\item We say a  spectrum $X$ is {\em strongly semistable} if there exits  a number $\alpha>1$ such that for any sufficiently large $l$, the  map $\iota_{l}:\pi_{k}(X_l)\to \pi_{k+1}(X_{l+1})$ is an  isomorphism for each $k\leq \alpha l$. A strongly semistable spectra is semistable in the sense of \cite{schwede} so a stable equivalence between strongly semistable spectrum is a $\pi_*$-isomorphism.
\item A {\em non-unital commutative symmetric ring spectrum} ( in short, {\em NUCSRS} ) is a spectrum $A$ with a commutative associative multiplication $A\wedge A\to A$ (but possibly without unit). A map of NUCSRS is a map of spectra prserving the multiplications.
\item $\CH_{\kk}$ denotes the category of (possibly unbounded) chain complexes over $\kk$ and chain maps. Differentials raise the degree (see next item for our degree rule). We endow $\CH_{\kk}$ the model structure where weak equivalences are quasi-isomorphisms and fibrations are surjections. $\otimes=\otimes_k$ denotes the standard tensor product of complexes.  
\item We deal with modules with multiple degrees (gradings). For modules having superscript(s) and/or subscript(s), its total degree is given by the following formula.
\[
\text{(total degree)}=\text{(sum of superscripts)}-\text{(sum of subscripts)}.
\] 
For example, singular chains in $C_p(M)$ have  degree $-p$, and the total degree of a triple graded module $A^{\star\ *}_\bullet$  is $*+\star-\bullet$. $|a|$ denotes the (bi)degree of $a$. We sometimes omit super or subscripts if unnecessary. 
\item For a simplicial chain complex $C_\bullet^*$ (i.e. a functor $(\BDelta)^{op}\to \CH_{\kk}$) the {\em normalized complex } ( or {\em normalization} ) $NC^*_\bullet$ is a double complex defined by taking the normalized complex of a simplicial $\kk$-module in each chain degree.
\item For a small category $C$ and a cofibrantly generated model category $\mathcal{M}$ (in the sense of \cite{hovey}), we denote by $\FUN(C,\mathcal{M})$ the category of functors $C\to \mathcal{M}$ and natural transformations, which is endowed with  the projective model structure (see \cite{hirschhorn}). 
The colimit functor $\underset{C}{\colim}:\FUN(C,\MM)\to \MM$ is a left Quillen functor. Its left derived functor  is denoted by $\underset{C}{\hocolim}$ and called the homotopy colimit over $C$.
\item A {\em commutative differential bigraded algebra} (in short, {\em CDBA}) is a bigraded module $A^{\star\,*}$ equipped with a unital multiplication which is graded commutative for the total degree, and  preserves the bigrading, and a differential $\partial :A^{\star\, *}\to A^{\star+1,\,*}$ which satisfy the Leibniz rule for the total degree. A {\em map} of CDBA is a map of differential graded algebra preserving bigrading.  

\end{itemize}

\subsection{$\check{\text{C}}$ech complex and homotopy colimit}
\begin{defi}\label{Dcheckfunctor}
Let $\MM$ be a cofibrantly generated model category. We define a functor 
\[
\CECHF:\FUN(P_\nu(\underline{n+1})^{op},\MM)\to \FUN(\Delta^{op},\MM)\quad\text{ by }\quad \CECHF X[k]=\bigsqcup_{f:[k]\to \underline{n+1}}X_{f([k])}\, , 
\]
where $f$ runs through all weakly order-preserving maps. For order preserving map $\alpha :[l]\to [k]\in \Delta$, the map $\CECHF X[k]\to \CECHF X[l]$ is the sum of the maps  $X_{f([k])} \to X_{f\circ \alpha ([l])}$ induced by the inclusion $f\circ \alpha ([l])\subset f([k])$. 
\end{defi}
The following lemma is clear.
\begin{lem}\label{Lcheckfunctor} We use the notations of Definition \ref{Dcheckfunctor}. Let $X\in\FUN(P_\nu(\underline{n+1})^{op},\MM)$ be a functor.
\begin{enumerate}
\item  There exists an isomorphism  $:\underset{P_\nu(\underline{n+1})^{op}}{\hocolim}\, X\cong \underset{\Delta^{op}}{\hocolim}\, \CECHF X$   in $\Ho(\MM)$ which is natural for $X$. 
\item $X$ is cofibrant in $\FUN(P_\nu(\underline{n+1})^{op},\MM)$ if the following canonical map is a cofibration in $\MM$ for each $S\in P_\nu(\underline{n+1})$.
\[
\underset{\tiny
\begin{array}{c}
S'\supset S \\
S'\not= S
\end{array}}
{\colim}\ X_{S'}
\to X_S \, .
\]
\end{enumerate}
\end{lem}
\begin{proof}
Let $(i_n\circ\mathcal{G}_n)^*: \FUN(\Delta^{op},\MM)\to \FUN(P_\nu(\underline{n+1})^{op},\MM)$ be the pullback by the composition of $\mathcal{G}_n$ and the inclusion $i_n:\BDelta_n\to \BDelta$. Clearly, the pair $(\CECHF,(i_n\circ \mathcal{G}_n)^*)$ is a Quillen adjoint pair, and it is also clear that $\colim_{P_\nu(\underline{n+1})^{op}}X$ and $\colim_{\Delta^{op}}\CECHF X$ is naturally isomorphic. The part 1 follows from these observations. The part 2 is also clear.
\end{proof}
\subsection{Goodwillie-Weiss embedding calculus and Sinha's cosimplicial model}
 In this subsection, we give the definition of the cosimplicial space $\CC^\bullet\CPTM$ modeling $\Emb (S^1,M)$, and state its property.  We begin with an  analogue of the punctured knot model in \cite[Definition 3.4]{sinha}, which is an intermediate object between $\Emb(S^1,M)$ and $\CC^\bullet\CPTM$. 
\begin{defi}\label{Dpunctured}
\begin{itemize}
 \item Let $S^1=[0,1]/0\sim 1$ and $J_i\subset S^1$ be the image of the interval $\left(1-\frac{1}{2^i}-\frac{1}{10^i}, 1-\frac{1}{2^i}\right)$ by the quotient map $[0,1]\to S^1$.  
\item We fix an embedding $M\to \RR^{N+1}$ for sufficiently large $N$. We endow $M$ with the Riemannian metric induced by the Euclidean metric on $\RR^{N+1}$ via this embedding. Let $\TM=STM$ denote the total space of the unit sphere tangent bundle of $M$. 
\item For a subset $S\subset \underline{n}$, let $E_S(M)$ be the space  of embeddings $S^1-\bigcup_{i\in S}J_i\to M$ of constant speed.
\item Define a functor $\mathcal{E}_n(M):P_{\nu}(\underline{n+1})\to \CG$ by assigning to a subset $S$ the space $E_S(M)$ and set 
\[
P_n\Emb (S^1,M):=\underset{P_{\nu}(\underline{n+1})}{\holim} \mathcal{E}_n(M). 
\]Let $\alpha_n:Emb(S^1,M)\to P_n\Emb(S^1,M)$ be the map induced from the restriction of domain. The category $P_\nu(\underline{n})$ is a subcategory of $P_\nu(\underline{n+1})$ via the standard inclusion $\underline{n}\to\underline{n+1}$. By our choice of $J_i$, we have a canonical restriction map $r_n:P_n\Emb(S^1,M)\to P_{n-1}\Emb(S^1,M)$. The maps $\alpha_n$ induces a map
\[
\alpha_\infty : \Emb(S^1,M)\to \underset{n}{\holim}\, P_n\Emb(S^1,M)
\]
where the right hand side is the homotopy limit of the tower $\cdots \stackrel{r_{n+1}}{\to} P_n\Emb(S^1,M)\stackrel{r_n}{\to}P_n\Emb(S^1,M)\stackrel{r_{n-1}}{\to}\cdots\stackrel{r_2}{\to}P_1\Emb(S^1,M)$.
\end{itemize}

\end{defi}
\begin{rem}\label{Rpunctured}
Our choice of $J_i$ is different from \cite{sinha} since we adopt the  reverse labeling of coface and codegeneracy maps of the cosimplicial model to \cite{sinha} for the author's preference. This does not cause any new problem.  
\end{rem}
\begin{lem}\label{Lpunctured}
The map $\alpha_n:\Emb(S^1,M)\to P_n\Emb (S^1,M)$ is $(n-1)(\dd-3)$-connected. In particular, $\alpha_\infty$ is a weak homotopy equivalence.
\end{lem}
\begin{proof}
Let $p:\Emb(S^1,M)\to \TM$ be the evaluation of value and tangent vector at $0\in S^1$. As is well-known, $p$ is a fibration. Let $D$ be a closed subset on $M$ diffeomorphic to closed $\dd$-dimensional disk. Let $\Emb([0,1], M- Int(D))$ be the space of embeddings $[0,1]\to M-Int (D)$ whose value and tangent vector at endpoints is a fixed value in $\partial D$ and vector. If we take a point of  $\TM$, for some choice of  the disk $D$, fixed endpoints, and embedded path between the points in $D$, we have the inclusion from $\Emb([0,1],M-Int(D))$ to the fiber of $p$ at the point. This inclusion is a weak homotopy equivalence. Its homotopy inverse is given by shrinking the disk $D$ to the point. Thus, we have a homotopy fiber sequence
\[
\Emb([0,1],M-Int(D))\to \Emb(S^1,M)\to \TM
\]   
Restricting the domain, we have similar fiber sequence $E_S(M-Int(D))\to E_S(M)\to \TM$, where the left hand side is the space defined in \cite[Definition.3.1]{sinha} with the obvious modification for $J_i$. (In \cite{sinha}, $M$ denotes a manifold with boundary so we apply the definitions to $M-Int(D)$ instead of our closed $M$.) Passing to homotopy limits, we have the following diagram
\[
\xymatrix{
\Emb([0,1],M-Int(D))\ar[r]\ar[d]&Emb(S^1,M)\ar[r]\ar[d]^{\alpha_n}&\TM\ar[d]^{id} \\
P_n\Emb([0,1],M-Int(D))\ar[r] & P_n\Emb(S^1,M)\ar[r]& \TM,}
\] 
where the both horizontal sequence are homotopy fiber sequences and  the left bottom corner is the punctured knot model in \cite[Definition.3.4]{sinha} (with the obvious modification for $J_i$). As in \cite[Theorem.3.5]{sinha}, by theorems of Goodwillie, Klein, and Weiss, the left vertical arrow is $(n-1)(\dd-3)$-connected, and so is the middle.
\end{proof}
\begin{rem}\label{Rtaylor}
Let $T_n\Emb(S^1,M)$ be the $n$-th stage of Taylor tower (or polynomial approximation). Restriction of the domain induces a map $P_n\Emb(S^1,M)\to T_n\Emb(S^1,M)$ which is compatible with canonical maps from $\Emb(S^1,M)$, but the author does not know this map is weak homotopy equivalence.
\end{rem}

Our cosimplicial space is analogous  to the well-known cosimiplical model of a free loop space  just like Sinha's original space is analogous to that of a based loop space. So the space $\CC^n\CPTM$ is related to a configuration space of $n+1$ points (not $n$ points).

\begin{defi}\label{Dcosimplicial}
Let $||-||$ denote the standard Euclidean norm in $\RR^{N+1}$.
\begin{itemize}
\item Let $C_n(M)=\{ (x_0,\dots ,x_{n-1})\in M^{\times n}\mid x_k\not=x_l\  \text{if}\ k\not= l\}$ be the ordered configuration space of $n$ points in $M$.  Similarly, we set $C_2([n])=\{(k,l)\in [n]^{\times 2}\mid k\not =l\}$.
\item  Let $C_n\langle [M]\rangle $ be the closure of the image of the map
\[
C_n(M)\to M^{\times n}\times (S^N)^{\times C_2([n-1])}\quad (x_k)_k\mapsto (x_k, u_{kl})_{kl}
\] 
where $u_{kl}=\frac{x_l-x_k}{||x_l-x_k||}$. $C_n\langle [M]\rangle $ is the same  as the space in Definition 1.3 of \cite{sinha}, though our labeling of points begins with $0$. Define a space $\CC^{n}\langle[M]\rangle$ by the following pullback diagram :
\[
\xymatrix{\CC^n\langle[ M]\rangle \ar[r] \ar[d] & \TM ^{\times n+1}\ar[d] \\
C_{n+1}\langle[M]\rangle \ar[r] & M^{\times n+1} }
\]
Here, the right vertical arrow is the product of standard projection and the bottom horizontal one is the composition of the canonical inclusion $C_{n+1}\CPTM\to  M^{\times n+1}\times (S^N)^{\times C_2([n])}$ and the projection.
\item Let $\tau : T_xM\to \RR^{N+1}$ be the linear monomorphism from the tangent space induced by the differential of the embedding fixed in Definition \ref{Dpunctured} and the identification $T_x\RR^{N+1}\cong \RR^{N+1}$ by the standard basis. Set $A'_{n+1}\CPTM:=M^{\times n+1}\times (S^N)^{\times ([n]^{\times 2})}$. Let $\beta'_{n+1}:\CC^n\CPTM \to A'_{n+1}\CPTM$ be the map given by
\[
\beta'_{n+1}(x_k, u_{kl}, y_k)=(x_k,u'_{kl}), \qquad u'_{kl}=\left\{
\begin{array}{cc}
u_{kl} & (k\not =l) \\
\tau(y_k) & (k=l)
\end{array}\right. ,
\]
where $y_k$ is a unit tangent vector at $x_k$. This is clearly a monomorphism.
For an integer $i$ with $0\leq i\leq n+1$, we define maps $d_i:[n+1]\to [n]$ by
\[
\begin{split}
d_i(k)&=\left\{
\begin{array}{cc}
k & (k\leq i)\\
k-1 & (k>i)
\end{array}\right. \qquad (0\leq i\leq n) \\
d_{n+1}& = d_0\circ \sigma, 
\end{split}
\]where $\sigma$ is the cyclic permutation $\sigma(k)=k+1 (\mathrm{modulo} \ n+2)$. (This $d_i$ is the same as $s^i$ in subsection \ref{SSnt} but we use the different notation to avoid confusion.) We define the map $d^i : A'_{n+1}\CPTM\to A'_{n+2}\CPTM$ by 
\[
d^i(x_k,u_{kl})_{0\leq k, l\leq n}=(x_{f(k)}, u_{f(k),f(l)})_{0\leq k, l\leq n+1},\qquad (f=d_i).
\] 
This map restricts to the  map $d^i:\CC^n\langle[M]\rangle\to \CC^{n+1}\langle[M]\rangle  $ via $\beta'_{n+1}$,\ $\beta'_{n+2}$. Similarly, we define the codegeneracy map $s^i:\CC^n\langle[M]\rangle\to \CC^{n-1}\langle[M]\rangle  \ (0\leq i\leq n-1)$ as the pullback by the map $s_i:[n-1]\to [n]$, $s_i(k)=\left\{
\begin{array}{cc}
k & (k\leq i) \\
k+1 & (k>i)
\end{array}\right.$. The collection $\CC^\bullet\CPTM=\{ \CC^n\CPTM, d^i, s^i\}$ forms a  cosimplicial space. Well-definedness  of this is verified in Lemma \ref{Lcosimplicialwell} below.
\item We call the Bousfield-Kan type cohomology spectral sequence associated to $\CC^\bullet\langle[ M]\rangle$  the {\em Sinha spectral sequence for $M$} in short, {\em Sinha s.s.}, and denote it by $\{\SINHASS_r\}_r$.

\end{itemize}
\end{defi}
Intuitively, an element of $C_n\langle [M]\rangle$ is a  configuration of $n$ points in $M$  some points of which are allowed to collide, or in other words, to be infinitesimally close, and the direction of collision is recorded as the unit vector $u_{kl}$  if $k$-th and $l$-th points collide. An element of   $\CC^n\CPTM$ is an element of $C_{n+1}\CPTM$, each point of which has a  tangent vector. For $0\leq i\leq n$, the map $d^i$ replaces the $i$-th point in a configuration with two  points colliding at the point along its vector. These points are  labeled  by $i,\ i+1$. Their  vectors are copies of the original vector  (see Figure \ref{Fcoface}). The map $d^{n+1}$ replaces the $0$-th points with two points similarly and labels them by $n+1, 0$ (and slides other labels). The map $s^i$ forgets $i+1$-th point and its vector.
\begin{figure}
\begin{center}
{\unitlength 0.1in%
\begin{picture}(29.2000,9.8000)(2.0000,-10.8000)%
%
\special{sh 1.000}%
\special{ia 830 1005 32 33 0.0000000 6.2831853}%
\special{pn 8}%
\special{ar 830 1005 32 33 0.0000000 6.2831853}%
%
\special{pn 20}%
\special{pa 835 1005}%
\special{pa 991 941}%
\special{fp}%
\special{sh 1}%
\special{pa 991 941}%
\special{pa 922 948}%
\special{pa 942 961}%
\special{pa 937 985}%
\special{pa 991 941}%
\special{fp}%
%
\special{sh 1.000}%
\special{ia 2536 1000 33 32 0.0000000 6.2831853}%
\special{pn 8}%
\special{ar 2536 1000 33 32 0.0000000 6.2831853}%
%
\special{pn 8}%
\special{ar 2536 364 488 264 0.0000000 6.2831853}%
%
\special{pn 8}%
\special{pa 2193 553}%
\special{pa 2221 571}%
\special{pa 2275 607}%
\special{pa 2327 645}%
\special{pa 2352 665}%
\special{pa 2398 709}%
\special{pa 2418 733}%
\special{pa 2437 759}%
\special{pa 2455 786}%
\special{pa 2470 814}%
\special{pa 2483 844}%
\special{pa 2495 874}%
\special{pa 2504 906}%
\special{pa 2511 937}%
\special{pa 2516 969}%
\special{pa 2519 1001}%
\special{pa 2520 1016}%
\special{fp}%
%
\special{pn 8}%
\special{pa 2874 558}%
\special{pa 2843 574}%
\special{pa 2813 590}%
\special{pa 2784 606}%
\special{pa 2755 623}%
\special{pa 2727 640}%
\special{pa 2701 659}%
\special{pa 2676 679}%
\special{pa 2653 700}%
\special{pa 2633 722}%
\special{pa 2614 747}%
\special{pa 2599 773}%
\special{pa 2586 801}%
\special{pa 2576 830}%
\special{pa 2567 861}%
\special{pa 2560 892}%
\special{pa 2553 925}%
\special{pa 2543 991}%
\special{pa 2542 1000}%
\special{fp}%
%
\special{sh 1.000}%
\special{ia 2628 267 32 32 0.0000000 6.2831853}%
\special{pn 8}%
\special{ar 2628 267 32 32 0.0000000 6.2831853}%
%
\special{pn 20}%
\special{pa 2633 267}%
\special{pa 2789 202}%
\special{fp}%
\special{sh 1}%
\special{pa 2789 202}%
\special{pa 2720 209}%
\special{pa 2740 223}%
\special{pa 2735 246}%
\special{pa 2789 202}%
\special{fp}%
%
\special{pn 8}%
\special{pa 2848 181}%
\special{pa 2123 477}%
\special{da 0.070}%
%
\special{sh 1.000}%
\special{ia 2263 418 32 32 0.0000000 6.2831853}%
\special{pn 8}%
\special{ar 2263 418 32 32 0.0000000 6.2831853}%
%
\special{pn 20}%
\special{pa 2268 418}%
\special{pa 2424 353}%
\special{fp}%
\special{sh 1}%
\special{pa 2424 353}%
\special{pa 2355 360}%
\special{pa 2375 374}%
\special{pa 2370 397}%
\special{pa 2424 353}%
\special{fp}%
\put(35.1000,-1.7000){\makebox(0,0){$u_{i,i+1}=y_i$}}%
%
\special{pn 8}%
\special{ar 1393 2100 1567 1573 3.8467830 4.6510457}%
%
\special{pn 8}%
\special{ar 3132 2121 1636 1643 3.8434419 4.2090905}%
%
\special{pn 8}%
\special{ar 3132 2121 1637 1645 4.6155426 4.6958868}%
\put(16.1900,-7.2000){\makebox(0,0){$\stackrel{d^i}{\longmapsto}$}}%
\put(3.5000,-6.5000){\makebox(0,0){$M$}}%
\put(24.0000,-4.9000){\makebox(0,0){$i$}}%
\put(28.0000,-3.6000){\makebox(0,0){$i+1$}}%
\put(8.3000,-11.2000){\makebox(0,0){$x_i$}}%
\put(9.6000,-8.4000){\makebox(0,0){$y_i$}}%
\end{picture}}%
\end{center}
\caption{intuition of the coface map $d^i$\ :\ $y_i$ is the vector at $x_i$}\label{Fcoface}
\end{figure}
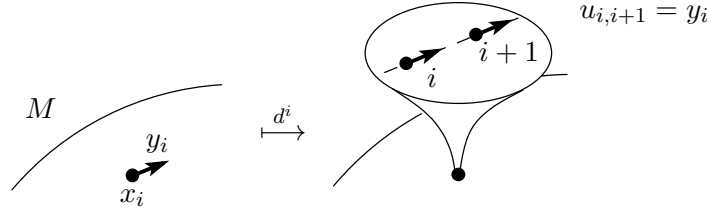

\begin{lem}\label{Lcosimplicialwell}

\begin{enumerate}
\item The map $C_n(M)\to M^{\times n}\times (S^N)^{\times C_2([n-1])}$ given in Definition  \ref{Dcosimplicial} restricts to homotopy equivalence $C_n(M)\to C_n\CPTM$.
\item The cosimplicial space $\CC^\bullet\langle [M]\rangle$ is well-defined.
\end{enumerate}
\end{lem}
\begin{proof}
The part 1 is proved in \cite[Corollary. 4.5, Theorem. 5.10]{sinha2}. For the part 2, by \cite[Proposition.6.6]{sinha2} the image of $d^i$ and $s^i$ is contained in $\CC^{n\pm 1}\CPTM$ ($C'_n\CPTM$ in the proposition is the same as $\CC^{n-1}\CPTM$ in our notation). Confirmation of the cosimplicial identities is a routine work. For example,  to confirm $d^{n+2}d^i=d^id^{n+1}:\CC^{n}\CPTM\to \CC^{n+2}\CPTM$ ($i<n+2$), it is enough to confirm the dual identity $d_id_{n+2}=d_{n+1}d_i:[n+2]\to [n]$. The both sides are equal to the map
\[
j\mapsto \left\{
\begin{array}{cc}
k & (k\leq i) \\
k-1 & (i<k<n+2) \\
0 & (k=n+2)
\end{array}
\right.\  (\text{if}\  i<n+1),  \qquad  
k\mapsto \left\{
\begin{array}{cc}
k & (k\leq n) \\
0 & (k=n+1,n+2)
\end{array}
\right.\ (\text{if}\  i=n+1)
\]

\end{proof}

\begin{lem}\label{Lcosimplicial}
Let $\mathcal{G}_n^*\CC^{\bullet}\langle [M]\rangle $ be the composition functor $P_\nu(\underline{n+1})\stackrel{\mathcal{G}_n}{\to} \BDelta_n\stackrel{\CC^\bullet\CPTM}{\to}\CG$. 
\begin{enumerate}

\item The homotopy limits of $\mathcal{E}_n(M)$  and $\mathcal{G}_n^*\CC^{\bullet}\langle [M]\rangle$ are connected by a zigzag of weak homotopy equivalences which are compatible with the inclusion $\underline{n}\to \underline{n+1}$.
\item The homotopy limit of $\CC^\bullet\langle[M]\rangle$ over $\BDelta_n$ and  that of $\mathcal{G}_n^*\CC^\bullet\langle [M]\rangle $ over $P_{\nu}(\underline{n+1})$ are connected by a zigzag of weak homotopy equivalences which are compatible with the inclusion $\underline{n}\to \underline{n+1}$.
\item The homotopy limit of $\CC^\bullet\langle[M]\rangle$ over $\BDelta$ and $\Emb(S^1,M)$ are  connected by a zigzag of weak homotopy equivalences.

\end{enumerate}
\end{lem}
\begin{proof}
The proof of the part 1 is completely analogous to the proof of Lemma 5.19 of \cite{sinha} so we omit details. Idea of the proof is consider the two space $\CC^{\# S-1}\CPTM$ and $E_S(M)$ as subspaces of a common space, where one can ``shrink components of embeddings until they become tangent vector'',as in Definition 5.14 of \cite{sinha}. The space is a subspace of the space of compact subspaces of $\CC^{\# S-1}\CPTM$ with the Hausdorff metric. This space and the inclusions can be chosen compatible with maps in $P_\nu(\underline{n+1})$. For example, the restriction $E_S(M)\to E_{S'}(M)$ corresponding to the inclusion $S=\underline{n+1}\subset S'=\underline{n+2}$ divides the component including the image of $0\in S^1$ into two components since the image of $J_{n+2}$ is removed. At the limit of shrinking components, this is consistent with the coface map $d^{n+1}$. These inclusions to the common space gives rise to zigzag of natural transformations which is a weak homotopy equivalence at each $S\subset \underline{n+1}$. This  induces the claimed zigzag.  
The part 2 follows from  the fact that the functor $\mathcal{G}_n$ is left cofinal (see Theorem 6.7 of \cite{sinha}). The part 3 follows from the part 1, 2 and Lemma \ref{Lpunctured}.
\end{proof} 

\subsection{Operads, comodules, and Hochschild complex}
The term {\em operad} means \textit{non-symmetric} (or \textit{non-$\Sigma$}) operad (see \cite{KM, muro}).   An operad $\oper=\{\oper(n)\}_{n\geq 1}$ in a symmetric monoidal category $(\CC, \otimes)$ is a symmetric sequence equipped with maps $(-\circ_i-):\oper(m)\otimes \oper(n)\to \oper(m+n-1)$  in $\CC$, called {\em partial compositions}, subject to certain conditions ($1\leq i\leq m$).  $\oper(n)$ is called the object at {\em arity} $n$. (We do not consider the object at arity $0$ so the sequence begins with $\oper(1)$.)  We mainly consider operads in $\CG$ (resp. in $\CH_{\kk}$), which are called topological operads (resp. chain operads), where the monoidal product is the standard cartesian product (resp. tensor product). Let $\oper$ be a topological operad.  $C_*(\oper)$ denotes the chain operad given by $C_*(\oper)(n)=C_*(\oper(n))$ with the induced structure.  We equip the sequence $\{\oper(n)\hotimes \Sph\}_n$ of spectra  with a   structure of an operad in $\SP$ as follows. The $i$-th partial composition  is given by
\[
\begin{split}
(\oper (m)\hotimes \Sph)\wedge (\oper (n)\hotimes \Sph )   &   \cong (\oper(m)\times \oper(n))\hotimes (\Sph\wedge \Sph) \\
  &   \cong  (\oper(m)\times \oper(n))\hotimes \Sph\stackrel{(-\circ_i-)\hotimes id}{\longrightarrow} \oper(m+n-1)\hotimes \Sph,
\end{split} 
\]
 see subsection \ref{SSnt} for the isomorphisms. The action of $\Sigma_n$ is the naturally induced  action. 
 We  denote this operad by the same symbol $\oper$.    We let  $\ASS$ denote the both of (discrete) topological  and $\kk$-linear versions of the associative operad by abuse of notations. For the $\kk$-linear version, we fix a generator $\mu\in \ASS(2)$ throughout this paper.   $\KK$ denotes  the Stasheff's associahedral operad, and $\ASS_\infty$ the cellular chain operad of $\KK$. Precisely speaking,  
$\ASS_\infty$ is generated by a set $\{\, \mu_k\in \ASS_\infty(k) \, \}_{k\geq 2}$ (\ $|\mu_k|=-k+2$\ )  with partial compositions. The differential is given by the following formula.
\[
d\mu _k=\sum_{{\footnotesize \begin{array}{c}
l,p,q\\ 
l+q=k+1
\end{array}}}(-1)^\zeta\, \mu_l\circ_{p+1}\mu_q
\]
where $\zeta=\zeta(l,p,q)=p+q(l-p-1)$. \\
\indent In the  following definition, we adopt the point-set description as if a category $\CC$ were the category of sets for simplicity.
\begin{defi}\label{Dcontra-module}
\begin{itemize}
\item Let $\oper$ be an operad over a symmetric monoidal category $\Cat$. A \textit{(left) $\oper$-comodule in $\Cat$} is 
 a symmetric sequence $X=\{X(n)\}_{n\geq 1}$ in $\Cat$ equipped with 
 a  map  
\[
(-\circ_i-):\oper(m)\otimes X(m+n-1)\to X(n)\in \Cat\, ,
\] called a {\em partial composition,} for each $m\geq 1,\ n\geq 1,\ 1\leq i\leq n$
which satisfy the following conditions.
\begin{enumerate}
\item For $a\in \oper(m), b\in \oper(l),$ and $x\in X(l+m+n-2)$,
\[
a\circ_i(b\circ_jx)=
\left\{
\begin{array}{cc}
b\circ_j(a\circ_{i+l-1}x) & \text{ if } j<i \\
(a\circ_{j-i+1}b)\circ_ix & \text{ if } i\leq j\leq i+m-1 \\
b\circ_{j-m+1}(a\circ_i x) &\text{ if } i+m-1<j
\end{array}
\right. \]

\item For the unit $1\in \oper(1)$ and $x\in X(n)$,
$1\circ _ix=x$
\item For $a\in \oper(m)$, $x\in X(m+n-1)$, and $\sigma \in \Sigma_n$,
\[
 (a\circ_ix)^{\sigma}=a\circ_{\sigma^{-1}(i)}(x^{\sigma_1})
\]
where  $\sigma_1\in \Sigma_{m+n-1}$ is the permutation induced by $\sigma$ with replacing a letter $\sigma^{-1}(i)$ with   $m$-letters $\sigma^{-1}(i),\dots, \sigma^{-1}(i)+m-1$. In other words, 
\[
\sigma_1(k)=\left\{
\begin{array}{ll}
\sigma(k) & (k<\sigma^{-1}(i) \text{ and }\sigma (k)<i) \\
\sigma(k)+m-1 & (k<\sigma^{-1}(i) \text{ and }\sigma (k)>i)  \\
i+k-\sigma^{-1}(i)  & (\sigma^{-1}(i)\leq k\leq \sigma^{-1}(i)+m-1)  \\
\sigma (k-m+1)  & (k>\sigma^{-1}(i)+m-1 \text{ and }\sigma (k-m+1)<i) \\
\sigma (k-m+1)+m-1 &(k>\sigma^{-1}(i)+m-1 \text{ and }\sigma (k-m+1)>i) 
\end{array}\right.
\]  
\end{enumerate}
A {\em map $f:X_1\to X_2$ of $\oper$-comodules} is a sequence of maps in $\Cat$ $\{f_n:X_1(n)\to X_2(n)\}_{n}$ which is compatible with the actions of symmetric groups and the partial compositions.   
\item A {\em (right) $\oper$-module  in $\Cat$} is a symmetric sequence $Y=\{Y(n)\}_{n\geq 1}$ equipped with a set of partial compositions $Y(n)\otimes \oper (m)\to Y(m+n-1)$ which satisfy the following conditions. 
\begin{enumerate}
\item For $a\in \oper(m), b\in \oper(l),$ and $y\in y(n)$,
\[
(y\circ_ja)\circ_i b=
\left\{
\begin{array}{cc}
(y\circ_i b)\circ_{j+l-1} a & \text{ if } i<j \\
y\circ_j(a\circ_{i-j+1}b) & \text{ if } j\leq i\leq j+m-1 \\
(y\circ_{i+m-1}b)\circ_j a &\text{ if } i >j+m-1
\end{array}
\right. .\]
\item For the unit $1\in \oper(1)$ and $y\in X(n)$,
$y\circ _i1=y$
\item For $a\in \oper(m)$, $y\in X(n)$, and $\sigma \in \Sigma_n$,
\[
 y^\sigma \circ_i a=(y\circ_{\sigma(i)} a)^{\sigma_2}
\]
where  $\sigma_2\in \Sigma_{m+n-1}$ is the permutation induced by $\sigma$ with replacing a letter $i$ with   $m$-letters $i,\dots,i+m-1$. 
In other words, 
\[
\sigma_2(k)=\left\{
\begin{array}{ll}
\sigma(k) & (k<i \text{ and }\sigma (k)<\sigma (i)) \\
\sigma(k)+m-1 & (k< i \text{ and }\sigma (k)>\sigma (i))  \\
\sigma(i) +k-i  & (i\leq k\leq i+m-1)  \\
\sigma (k-m+1)  & (k>i+m-1 \text{ and }\sigma (k-m+1)<\sigma (i) ) \\
\sigma (k-m+1)+m-1 &(k>i+m-1 \text{ and }\sigma (k-m+1)>\sigma (i)) 
\end{array}\right. .
\]

\end{enumerate}
A {\em map of  modules} is defined similarly to that of comodules.
\item For a topological operad $\oper$ (regarded as an operad in $\SP$), A {\em  $\oper$-comodule of NUCSRS} is a  $\oper$-comodule $X$ in $\SP$ such that each $X(n)$ is equipped with a structure of a NUCSRS and the action of $\Sigma_n$ on $X(n)$ and the partial composition $(a\circ_i-):X(n+m-1)\to X(n)$  is a map of NUCSRS for each $a\in \oper(m)$. A {\em map of comodules of NUCSRS} is a map of comodules which is also a map of NUCSRS at each arity.  
\item For a topological operad $\oper$ and a  $\oper$-module $Y$, we define a  $\oper$-comodule $Y^\vee$ of NUCSRS as follows.
\begin{enumerate}
\item We set $Y^\vee (n)=Y(n)^\vee$ (see subsection \ref{SSnt}).
\item For $f\in Y^\vee (n)$ and $\sigma \in \Sigma_n$, we define an action $f^\sigma$ by $f^\sigma(y)=f(y^{\sigma^{-1}})$ for each $y\in Y(n)$. 
\item For $a\in \oper(m)$ and $f\in Y^\vee (m+n-1)$, we define a partial compostion $a\circ_if$ by $a\circ_if(y)=f(y\circ_ia)$ for each $y\in Y(n)$.
\item We define a multiplication $Y^\vee(n)\wedge Y^\vee(n)\to Y^\vee(n)$ as the pushforward by the multiplication of $\Sphere$. (This is actually unital.)
\end{enumerate}
This construction is natural for maps of  $\oper$-modules.
\item A {\em  $\ASS$-comodule $X$ of CDBA} is a  $\ASS$-comodule (in $\CH_{\kk}$) such that each $X(n)$ is a CDBA  and the partial composition $\mu\circ_i(-):X(n)\to X(n-1)$ (with the fixed generator $\mu\in \ASS(2)$) and the action of $\sigma\in \Sigma_n$ preserves the differential , bigrading,  multiplication, and unit. 
\end{itemize}

\end{defi} 
The axioms for the partial composition of modules in Definition \ref{Dcontra-module} are the standard ones, which is naturally interpreted in terms of concatenation of trees. The action of $\sigma\in \Sigma_n$ is  interpreted as replacement of  labels $i$ on leaves with   labels $\sigma^{-1}(i)$, and the axioms is the natural one with this interpretation. The axioms for comodule is simply dual to those for module. The comodule in Example \ref{Ealgebrahochschild} may give some intuition for it.\\ 
\indent Composing the unity and  associativity isomorphisms, we obtain a natural isomorphism $K\hotimes X\cong (K\hotimes \Sphere)\wedge X$ in $\SP$. Let $\oper$ be a topological operad. Via this isomorphism,  a structure of $\oper$-comodule in $\SP$ on a symmetric sequence $X$ is  equivalent to a set of maps
\[
\oper (m)\hotimes X(m+n-1)\to X(n)
\]
which satisfy the  conditions completely similar to those given in Definition \ref{Dcontra-module}. We also call these maps  partial compositions, and in the rest of paper, we will define comodules in $\SP$ with these maps.
\begin{rem}
Precisely speaking, the notion of comodule in Definition \ref{Dcontra-module} should be named as contracomodule because our comodule to module is the same as contramodule to comodule in \cite{Positselski}, but for simplicity we adopt the terminology.
\end{rem}
\indent The following definition is essentially due to \cite{GJ} though we adopt a different  sign rule. 
\begin{defi}\label{Dhochschild}
Let $X^*$ be a  $\ASS_\infty$-comodule in $\CH_{\kk}$.
We define a chain complex  $(\Hoch_\bullet X^*, \tilde d)$ called \textit{Hochschild complex of $X$},  as follows.
Set $\Hoch_nX^*=X^*(n+1)$. By our convention, the total degree is $*-\bullet$. The  differential $\tilde d$ is given as a map
\[
\tilde d =d-\delta : \underset{a-n=k}{\bigoplus}\Hoch_n X^a\longrightarrow \underset{a-n=k+1}{\bigoplus}\Hoch_n X^a.
\] 
Here $d$ is the internal (original) differential on $X^a(n+1)$ and $\delta$ is given by the following formula 
\[
\delta (x)=\sum_{i=0}^n\sum_{k=2}^{n-i+1}(-1)^{\epsilon}\mu_k\circ_{i+1}x 
 +\sum_{s=1}^n\sum_{k=s+1}^{n+1}(-1)^{\theta}\mu_k \circ_1 x^s
\] 
for $x\in X^a(n+1)$, where $\epsilon=\epsilon(a,i,k)=(a+i)(k+1)$, and $\theta=\theta(s,n,k,a) =sn+(k+1)a$, and $x^s$ denotes the image of $x$ by the action of  permutation in $\Sigma_{n+1}$ which transposes first $n-s+1$ letters and last $s$ letters. 
\end{defi}
The following example gives some intuition of definitions of comodule and Hochschild complex, while is not used later.
\begin{exa}\label{Ealgebrahochschild}
Let $\Cat$ denote the category of $\kk$-modules, and $A$ denote $\kk$-algebra.
Let $m_n\in \ASS(n)$ be the element defined by successive partial compositions of the generator $\mu\in \ASS(2)$.  Define a  $\ASS$-comodule $X_A$ by
\[
X_A(n)=A^{\otimes n}, \qquad m_k\circ_i(x_1\otimes\cdots\otimes x_{k+n-1})=x_1\otimes\cdots x_{i-1}\otimes (x_i\cdots x_{i+k-1})\otimes x_{i+k}\otimes\cdots\otimes x_{k+n-1},
\]
where $x_i\cdots x_{i+k-1}$ is the product in $A$. We regard $X_A$ as a $\ASS_\infty$-comodule via a map $\ASS_\infty\to \ASS$ of operads. The Hochschild complex of $X_A$  is the usual (unnormalized) Hochschild complex of the associative algebra $A$.
\end{exa}

\begin{lem}\label{Ldifferential}
With the notation of Definition \ref{Dhochschild}, $(\td)^2=0$.
\end{lem}
\begin{proof}
Roughly writing, 
\[
\begin{split}
(\tilde d)^2(x) =&\tilde d(dx-\delta x)=ddx-d\delta x-\delta dx-\delta\delta x \\
=&d(\mu_k\circ_{i+1}x+\mu_k\circ_1 x^s)+ (\mu_k\circ_{i+1}dx+\mu_k\circ_1 dx^s) \\
&-\mu_l\circ_{j+1}(\mu_k\circ_{i+1}x)+\mu_l\circ(\mu_k\circ_1x^s)
+\mu_l\circ_1(\mu_k\circ_{i+1}x)^t+\mu_l\circ_1(\mu_k\circ_1x^s)^t\\
=&(d\mu_k)\circ_{i+1}x+(d\mu_k)\circ_1 x^s\\
&-\mu_l\circ_{j+1}(\mu_k\circ_{i+1}x)+\mu_l\circ(\mu_k \circ_1 x^s)
+\mu_l\circ_1(\mu_k\circ_{i+1}x)^t+\mu_l\circ_1(\mu_k\circ_1x^s)^t\\
\end{split}
\]
(Here we already canceled the terms containing $dx$ since the cancellation of signs is obvious.)  
So we have six types of terms. To see which terms cancel with each other, we divide these terms into the following smaller classes.
\begin{enumerate}
\item $(d\mu_k)\circ_{i+1}x$, $d\mu_k=\sum\mu_l\circ_{p+1}\mu_q$
\item $(d\mu_k)\circ_1x^s$, $d\mu_k=\sum\mu_l\circ_{p+1}\mu_q$\ : \ (a) $s<p+1$\quad (b)  $p+q\leq s$\quad  (c) $p=0$ and \ $q>s$\quad (d) $p>0$ and $p+q>s\geq p+1$,
\item $\mu_l\circ_{j+1}(\mu_k\circ_{i+1}x)\ $\ :\ (a) $i<j$\quad (b) $j+l-1<i$\quad (c) $j\leq i\leq j+l-1$

\item $\mu_l\circ_{j+1}(\mu_k \circ_{1}x^s)$\ :\ (a) $j=0$\quad  (b)  $j>0$
\item $\mu_l\circ_1(\mu_k\circ_{i+1}x)^t$\ :\ (a) $i+1<n-k-t+3$ and  $l<s+i+1$\quad (b) $i+1<n-k-t+3$ and $l\geq s+i+1$\quad (c) $i+1\geq n-k-t+3$

\item $\mu_l\circ_1(\mu_k\circ_1x^s)^t$

\end{enumerate} 
Now we claim that the terms in (1) cancel with the terms in (3-c),\  (2-a) with (5-b),\  (2-b) with (5-c),\  (2-c) with (4-a),\  (2-d) with (6),\  (3-a) with (3-b),\  and (4-b) with (5-a).\\
\indent We shall verify the first and third cancellations. Other verifications are similar and omitted. For the first one, the coefficient of $(\mu_l\circ_{p+1}\mu_q)\circ_{i+1}x$ in the terms in (1) is $(-1)^{\alpha_1}$ where 
\[
\alpha_1=\zeta(l,p,q)+\epsilon(a,i,l+q+1)+1.
\]
For the terms in (3-c), by the rules of the partial composition , we have $\mu_l\circ_{j+1}(\mu_k\circ_{i+1}x)=(\mu_l\circ_{i-j+1}\mu_k)\circ_{j+1}x$. In order to match these terms with the terms in (1), we set $q'=k$, $p'+1=i-j+1$, and $i'+1=j+1$. This change of subscripts implies $\mu_l\circ_{j+1}(\mu_k\circ_{i+1}x)=(\mu_l\circ_{p'+1}\mu_{q'})\circ_{i'+1}x$. Clearly we have $j=i'$, $i=p'+i'$. The coefficient of $\mu_l\circ_{j+1}(\mu_k\circ_{i+1}x)$ in the terms (3-c) is $(-1)^{\alpha_2}$ where
\[
\alpha_2=\epsilon (a,i,k)+1+\epsilon (a+k-2,j,l)+1=\epsilon(a,p'+i',q')+\epsilon(a-q'+2,i',l)+2.
\]
When we substitute $q'=q$, $p'=p$, $i'=i$ in the last expression, elementary computation shows $\alpha_1+\alpha_2 \equiv 1\ (\mathrm{mod}\  2)$. Thus the terms (1) cancel with the terms (3-c). \\
\indent For the third case, the coefficient of $(\mu_l\circ_{p+1}\mu_q)\circ_1 x^s$ in the terms (2-b) is $(-1)^{\beta_1}$, where
\[
\beta_1=\zeta(l,p,q)+\theta(s,n,l+q-1,a)+1.
\] 
For the terms (5-c), the condition $i+1\geq n-k-t+3$ implies that $\mu_k$ acts on a part of the last $t$ letters. By this and the rule of the partial composition, we have
\[
\mu_l\circ_1 (\mu_k\circ_{i+1}x)^t=\mu_l\circ_1(\mu_k\circ_{i-n+k+t-1}(x^{t+k-1}))=(\mu_l\circ_{i-n+k+t-1}\mu_k)\circ_1 x^{t+k-1}.
\]
In order to match these terms with the terms in (2-b), we set $p'+1=i-n+k+t-1$, $q'=k$ and  $s'=t+k-1$. This change of subscripts implies 
$\mu_l\circ_1 (\mu_k\circ_{i+1}x)^t=(\mu_l\circ_{p'+1}\mu_{q'})\circ_1x^{s'}$. Clearly we have $t=s'-q'+1$ and $i=p'+n-s'+1$. The coefficient of $\mu_l\circ_1 (\mu_k\circ_{i+1}x)^t$ is $(-1)^{\beta_2}$, where
\[
\beta_2=\epsilon(a,i,k)+1+\theta(t,a-k+2,n-k+1,l)+1=\epsilon(a,p'+n-s'+1,q')+\theta(s'-q'+1,n-q'+1,a-q'+2,l)+2.
\]
When we subsutitute $q'=q$, $p'=p$, $s'=s$ in the last expression,  elementary computation shows $\beta_1+\beta_2 \equiv 1\ (\mathrm{mod}\  2)$. Thus the terms (2-b) cancel with the terms (5-c). 
\end{proof}

\section{Comodule $\CCM$}
The purpose of this section is to define the comodule $\CCM$.
\subsection{A model of a Thom spectrum}
We introduce a model of the Thom spectrum $N^{-TN}$ as a symmetric spectrum for a closed manifold $N$. This model is essentially due to Cohen \cite{cohen}, and  slightly different from Cohen's original non-unital model mainly in that we use expanding embeddings. 
\begin{defi}\label{Dtubulernbd}
Let $N$ be a closed manifold. We fix a Riemannian metric on $N$ and denote  by $d_N(-,-)$ the distance on $N$ induced by the metric.  The standard Euclidean norm on $\RR^k$ is denoted by $||-||$. The distance in $\RR^k$ is  induced by $||-||$.
\begin{itemize}

\item For a smooth embedding $e:N\to L$ to a Riemannian manifold $L$ we set a number 
\[
r(e)=\inf\left\{\left. \frac{d_L(e(x),e(y))}{d_N(x,y)}\  \right|\  x,y\in N,\ x\not= y\right\}.
\]
It is easy to see $r(e)>0$. 
We say $e$ is {\em expanding} if the inequality $r(e)\geq 1$ holds. $\Emb^{ex}(N,L)$ denotes the space of all expanding embeddings from $N$ to $L$ with the  topology induced by the $C^\infty$-topology.
\item For a smooth embedding $e:N\to \RR^k$, we define a number $|e|$ as the number of $i\in \underline{k}$ such that the composition 
\[
N\stackrel{e}{\longrightarrow}\RR^k\stackrel{\text{$i$-th projection}}{\longrightarrow} \RR 
\] is not a constant map.
\item Let $e:N\to \RR^{k}$ be a smooth embedding. For $\epsilon >0$, we denote by $\nu_\epsilon(e)$ the open subset of $\RR^k$ consisting of the points whose Euclidean distance from $e(N)$ are smaller than $\epsilon$. Let $L(e)$ denote the minimum of $1$ and the least upper bound of $\epsilon>0$ such that
there exists a retraction $\pi_e:\nu_\epsilon (e)\to e(N)$ satisfying the following conditions.
\begin{itemize}
\item For any $u\in\nu_\epsilon (e)$ and any $y\in N$, $||\pi_e(u)-u||\leq ||e(y)-u||$ and the equality holds if and only if $\pi_e(u)=e(y)$. 
\item For any $y\in N$, $\pi_e^{-1}(\{e(y)\})=B_\epsilon(e(y))\cap(e(y)+(T_yN)^{\perp })$. Here $B_\eps(e(y))$ is the open ball with center $e(y)$ and radius $\eps$.
\item The closure $\bar{\nu}_\eps(e)$ of $\nu_\eps(e)$ is a smooth submanifold of $\RR^{k}$ with boundary.
\end{itemize} 
(Such a retraction exists for a sufficiently small $\eps>0$ by a version of tubular neighborhood theorem, see \cite{MT}.) The retraction $\pi_e$ satisfying the above three conditions is unique.  We regard the map $\pi_e:{\nu}_\eps(e)\to e(N)$ as a disk bundle over $N$, identifying $N$ and $e(N)$.\\
\item Let $\TNT_k$ be the subspace of $\Emb^{ex}(N,\RR^k)\times \RR\times \RR^k$  consisting of the triples $(e,\eps, u)$ with $0<\eps<L(e)$.
 Define a subspace
$\partial \TNT_k\subset \TNT_k$ by
$(e,\epsilon, u)\in\partial \TNT_k\iff u\not \in \nu_{\epsilon}(e)$. We put 
\[
N^{-\tau}_k=\TNT_k/\partial \TNT_k
\]
We define a structure of a symmetric spectrum on $\NT$ as follows.
\begin{itemize}
\item We let $\Sigma_k$ act on $\RR^k$ and $\Emb^{ex}(N,\RR^k)$ by the standard permutation on components. The action of $\Sigma_k$ on $\NT_k$ is given by $[e,\eps, u]^\sigma=[e^\sigma,\eps, u^\sigma ]$. 
\item The map $S^1\wedge \NT_k\to \NT_{k+1}$ is given by $t\wedge [e,\eps, u]\mapsto [0\times e,\eps, (t,u)]$ where we regard $S^1=\RR \cup \{\infty\}$, and $0\times e:M\to \RR^{k+1}$ is given by $(0\times e)(x)=(0,e(x))$. 
\end{itemize}
\item We shall define a structure of NUCSRS on $\NT$. An element of $(\NT\wedge \NT)_k$ is represented by a data $\langle [e_1,\eps_1,u_1],[e_2,\eps_2,u_2];\sigma\rangle $ consisting of $[e_i,\eps_i,u_i]\in\NT_{k_i}$ ($i=1,2,\ k_1+k_2=k$) and $\sigma\in\Sigma_k$. We define a commutative associative multiplication $\mu:\NT\wedge \NT\to \NT$ by
\[
\mu(\langle [e_1,\eps_1,u_1],[e_2,\eps_2,u_2];\sigma\rangle  )=[e_{12},\ \eps_{12},\  (u_1,u_2)]^\sigma . 
\]
Here, $e_{12}=(e_1\times e_2) \circ \Delta $ where $\Delta:N\to N\times N$ is the  diagonal map, and 
\[
\eps_{12}=\min\left\{\left. \frac{\eps_1}{8^{|e_2|}},\, \frac{\eps_2}{8^{|e_1|}},\, L(e_{12}),\, \frac{L(e_1')}{8^{|e_{12}|-|e_1'|}},\dots,  \frac{L(e_m')}{8^{|e_{12}|-|e_m'|}}\ \right| \, m\geq 2,\, e_1':N\to \RR^{l_1},\dots, e_m':N\to \RR^{l_m}\right\}
\] where the finite sequence $(e'_1,\dots, e'_m)$ runs through the sequence of expanding embeddings satisfying 
$(e'_1\times \cdots \times e'_m)\circ \Delta^m=(e_{12})^{\tau}$ for a permutation $\tau\in \Sigma_{k_1+k_2}$, where $\Delta^m:N\to N^{\times m}$ denotes the diagonal map. 
\end{itemize}
\end{defi}
\begin{lem}\label{Lthomwelldef}
The structure of   NUCSRS on $\NT$ given in Definition \ref{Dtubulernbd} is well-defined
\end{lem}
\begin{proof}
Most part of the proof is the same as the proof of \cite[Theorem 3]{cohen}. We shall only verify the associativity of the number $\eps_{12}$. Let $[e_i,\eps_i,u_i]$ be an element of $\NT_{k_i}$ for $i=1,2,3$. We denote by $\eps_{(12)3}$ (resp. $\eps_{1(23)}$ ) the number for the result of  the elements of $i=1,2$ (resp. $i=2,3$ ) being multiplied at first. By definition, we have
\[
\eps_{(12)3}=\min\left\{\left. \frac{\eps_{12}}{8^{|e_3|}},\, \frac{\eps_3}{8^{|e_{12}|}},\, L(e_{123}),\, \frac{L(e_1')}{8^{|e_{123}|-|e_1'|}},\dots,  \frac{L(e_m')}{8^{|e_{123}|-|e_m'|}}\ \right| \, m\geq 2,\, e_1':N\to \RR^{l_1},\dots, e_m':N\to \RR^{l_m}\right\},
\] 
where $e_{123}=(e_1\times e_2\times e_3)\circ \Delta^3$, and the finite sequence $(e'_1,\dots, e'_m)$ runs through the sequence of expanding embeddings satisfying 
$(e'_1\times \cdots \times e'_m)\circ \Delta^m=(e_{123})^{\tau}$ for some $\tau\in \Sigma_{k_1+k_2+k_3}$. By the obvious equality $|e_{12}|=|e_1|+|e_2|$, we have 
\[
\eps_{(12)3}=\min\left\{\left. \frac{\eps_{1}}{8^{|e_2|+|e_3|}},\, \frac{\eps_2}{8^{|e_1|+|e_3|}},\, \frac{\eps_3}{8^{|e_1|+|e_2|}},\, L(e_{123}),\, \frac{L(e_1')}{8^{|e_{123}|-|e_1'|}},\dots,  \frac{L(e_m')}{8^{|e_{123}|-|e_m'|}}\ \right| \, m\geq 2,\, e_1',\dots, e_m'\right\},
\]
where the finite sequence $(e'_1,\dots, e'_m)$ runs through the same set as above. The number $\eps_{1(23)}$ is also seen to be equal to the number of the right hand side.
 \end{proof}
\subsection{Construction of a comodule $\TCCM$}\label{SSTCCM}
\begin{defi}\label{Dlittleinterval}
\begin{itemize}
\item For a closed interval $c=[a,b]$, we set $|c|=b-a$, and call a point $(a+b)/2\in c$ the {\em center of} $c$. 
\item We define a version of little interval operad, denoted by $\DD_1$ as follows.
Let $\DD_1(n)$ be the set of $n$-tuples $(c_1,c_2,\dots,c_n)$ of closed subintervals $c_i\subset [-1/2,1/2]$ such that 
 $c_1\cup\cdots\cup c_n=[-1/2,1/2]$ and $c_i\cap c_j$ is a one-point set or empty if $i\not = j$, and 
 the labeling of $1,\dots, n$ is consistent with the usual order of the real line $\RR$ (so $-1/2\in c_1$ and $1/2\in c_n$). We topologize $\DD_1(n)$ as a subspace of $\RR^n$ by the inclusion sending each interval to its center. 
The partial composition is given by the way completely analogous to the usual little interval operad. 
\item We identify $H_0(\DD_1(2)$ with $\ASS(2)$ by  sending the generator represented by a  topological point to the generator $\mu$. 
\end{itemize}
\end{defi}

Recall that we fixed a Riemannian metric on $M$ in Definition \ref{Dcosimplicial}. In the rest of the paper, we equip the space $\TM$ with the Sasaki metric, and the product $\TM^{\times n}$ the product metric. We fix a so small positive number $\rho$ that a geodesic of length $\rho$ exists for any initial value in $M$. After Lemma \ref{Ldiagonaltube}, we impose an additional assumption on $\rho$.  
\begin{defi}\label{Ddelta}
We shall define a map 
\[
\Delta'=\Delta[\configd,\configc;i]:\TM\to \TM^{\times m}
\]
 for each $\configd=(d_1,\dots,d_n)\in \DD_1(n)$, $\configc=(c_1,\dots,c_m)\in \DD_1(m)$
and $1\leq i\leq n$. Let $(x,y)$ denote a point of $\TM$ with $x\in M$ and $y\in ST_xM$. Let $s:[-\rho/2,\rho/2]\to M$ denote the geodesic segment with length parameter such that $s(0)=x$ and the  tangent vector of $s$ at $0$ is $y$.
Let $t_j\in [-1/2,1/2]$ be the center of $c_j$ and put $x_j=s(\rho\cdot |d_i|\cdot t_j)$ and set $y_j$ to be the tangent vector of $s$ at $\rho\cdot |d_i|\cdot t_j$. We set $\Delta'(x,y)=((x_1,y_1),\dots, (x_m,y_m))$, see Figure \ref{Fdiagonal}. 
\end{defi}
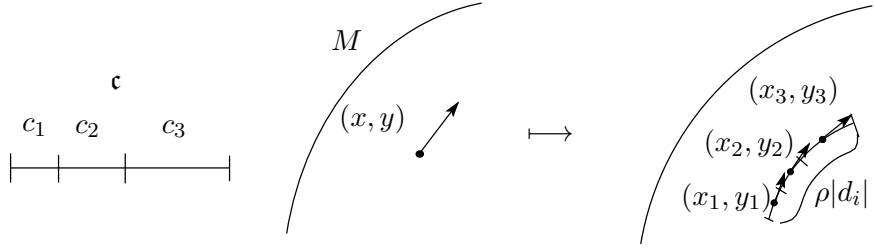
\begin{figure}
\begin{center}
{\unitlength 0.1in%
\begin{picture}(93.2300,16.3100)(1.9000,-16.3100)%
%
\special{pn 8}%
\special{pa 930 863}%
\special{pa 2076 863}%
\special{fp}%
%
\special{pn 8}%
\special{ar 3583 1490 1230 1508 3.3320367 4.5581052}%
%
\special{pn 8}%
\special{ar 5718 1516 1511 1528 3.3103094 4.5262119}%
%
\special{pn 8}%
\special{pa 930 937}%
\special{pa 930 789}%
\special{fp}%
\special{pa 2076 930}%
\special{pa 2076 801}%
\special{fp}%
\special{pa 1530 789}%
\special{pa 1530 942}%
\special{fp}%
\special{pa 1181 797}%
\special{pa 1181 937}%
\special{fp}%
\put(14.8000,-4.1400){\makebox(0,0){$\configc$}}%
\put(10.5500,-6.4900){\makebox(0,0){$c_1$}}%
\put(13.3300,-6.5600){\makebox(0,0){$c_2$}}%
\put(17.8700,-6.5600){\makebox(0,0){$c_3$}}%
\put(47.9600,-7.3600){\makebox(0,0){$(x_2,y_2)$}}%
\put(50.2600,-4.5900){\makebox(0,0){$(x_3,y_3)$}}%
\put(36.3300,-7.1900){\makebox(0,0)[lb]{$\longmapsto$}}%
%
\special{pn 8}%
\special{ar 5596 1311 719 730 3.3916995 4.3472241}%
%
\special{pn 8}%
\special{pa 9513 1631}%
\special{pa 9492 1571}%
\special{fp}%
%
\special{pn 8}%
\special{pa 4920 1142}%
\special{pa 4877 1122}%
\special{fp}%
%
\special{pn 4}%
\special{sh 1}%
\special{ar 4926 1046 16 16 0 6.2831853}%
\special{sh 1}%
\special{ar 5012 883 16 16 0 6.2831853}%
\special{sh 1}%
\special{ar 5180 714 16 16 0 6.2831853}%
\special{sh 1}%
\special{ar 5180 714 16 16 0 6.2831853}%
%
\special{pn 8}%
\special{pa 4931 1046}%
\special{pa 4980 893}%
\special{fp}%
\special{sh 1}%
\special{pa 4980 893}%
\special{pa 4941 950}%
\special{pa 4964 944}%
\special{pa 4979 963}%
\special{pa 4980 893}%
\special{fp}%
\special{pa 5009 883}%
\special{pa 5109 738}%
\special{fp}%
\special{sh 1}%
\special{pa 5109 738}%
\special{pa 5055 782}%
\special{pa 5079 782}%
\special{pa 5088 804}%
\special{pa 5109 738}%
\special{fp}%
\special{pa 5182 707}%
\special{pa 5328 588}%
\special{fp}%
\special{sh 1}%
\special{pa 5328 588}%
\special{pa 5264 615}%
\special{pa 5287 622}%
\special{pa 5289 646}%
\special{pa 5328 588}%
\special{fp}%
%
\special{pn 4}%
\special{sh 1}%
\special{ar 3068 790 16 16 0 6.2831853}%
\special{sh 1}%
\special{ar 3075 790 16 16 0 6.2831853}%
%
\special{pn 8}%
\special{pa 3068 790}%
\special{pa 3271 523}%
\special{fp}%
\special{sh 1}%
\special{pa 3271 523}%
\special{pa 3215 564}%
\special{pa 3239 565}%
\special{pa 3247 588}%
\special{pa 3271 523}%
\special{fp}%
\put(46.7800,-10.1800){\makebox(0,0){$(x_1,y_1)$}}%
\put(28.2500,-6.0700){\makebox(0,0){$(x,y)$}}%
\put(26.8900,-1.9200){\makebox(0,0){$M$}}%
%
\special{pn 8}%
\special{pa 4959 1152}%
\special{pa 4991 1145}%
\special{pa 5020 1133}%
\special{pa 5041 1112}%
\special{pa 5056 1085}%
\special{pa 5068 1054}%
\special{pa 5079 1021}%
\special{pa 5092 989}%
\special{pa 5107 959}%
\special{pa 5125 933}%
\special{pa 5145 908}%
\special{pa 5166 885}%
\special{pa 5189 863}%
\special{pa 5213 842}%
\special{pa 5238 823}%
\special{pa 5265 806}%
\special{pa 5295 791}%
\special{pa 5327 777}%
\special{pa 5351 760}%
\special{pa 5362 732}%
\special{pa 5363 698}%
\special{pa 5362 686}%
\special{fp}%
\put(52.8100,-9.9700){\makebox(0,0){$\rho |d_i|$}}%
%
\special{pn 8}%
\special{pa 4937 966}%
\special{pa 4989 998}%
\special{fp}%
%
\special{pn 8}%
\special{pa 5045 804}%
\special{pa 5100 852}%
\special{fp}%
%
\special{pn 8}%
\special{pa 5328 581}%
\special{pa 5370 700}%
\special{fp}%
\put(4.0000,-6.0000){\makebox(0,0){$\quad$}}%
\end{picture}}%
\end{center}
\caption{the map $\Delta'$\ :\ the geodesic segment is devided into the pieces  of rate of length $|c_1|:|c_2|:|c_3|$}\label{Fdiagonal}
\end{figure}
The following lemma is clear from the definition of $\Delta[\configd,\,\configc\, ;i]$.
\begin{lem}\label{Ldiagonalcomposition}
For any configurations $\configd, \configc_1, \configc_2$ and numbers $i, j$, the following equality holds.
\[
\Delta[\configd,\, \configc_1\circ_j\configc_2\,;i\,]=\Delta[\configd,\,\configc_1\,;i\,]\circ\Delta[\configd\circ_i\configc_1,\, \configc_2\,;i+j-1\,].
\]
{\rightline{\qedsymbol}}
\end{lem}
\begin{lem}\label{Lexpanding}
For  any sufficiently small positive number $\rho$, the following condition holds. The map $\Delta[\configd,\configc;i]$ is expanding for any numbers $n\geq 1$, $m\geq 1$,  and $i$ with $1\leq i\leq n$, and  elements $\configd\in \DD_1(n),\ \configc\in \DD_1(m)$, 
\end{lem}
\begin{proof}
It is enough to prove the case of $m=2$ since for $m\geq 3$, $\Delta'$ is equal to a successive composition of $(\Delta')$'s of  arity 2 by Lemma \ref{Ldiagonalcomposition}. We set $\rho_0=|d_i|\rho $. We shall consider the case that $M$ is a metric vector space $V$ as a local model. Take points $(x,y), (v,w)\in \widehat{V}=V\times SV$, where $SV$ is the  unit sphere in $V$. Put $\configc=(c_1,c_2)$.  Let $-s$,\ $t$ be the centers of $c_1$, $c_2$ respectively ($0<s,t <1/2,\ s+t=1/2$). By definition,
we have $\Delta'(x,y)=[(x-\rho_0 sy, y), (x+\rho_0 ty,y)]$. When we set $a=||x-v||$ and $b=||y-w||$, we easily see
\[
\begin{split}
||\Delta' (x,y)-\Delta' (v,w)||^2 &  \geq 2 a^2-\rho_0 |s-t|ab+\{ \rho_0^2(s^2+t^2)/4 +2\} b^2 \\
                                &  \geq  2 a^2 -\rho_0 |s-t|(a^2+b^2)/2+ \{ \rho_0^2(s^2+t^2)/4 +2\} b^2. 
\end{split}
\]
So  we have  
\[
 \frac{||\Delta' (x,y)-\Delta' (v,w)||}{||(x,y)-(v,w)||}\geq \frac{\sqrt{7}}{2} \ \text{\  for \ }\rho<1 \quad \cdots (*). 
\]
There exists a  number $r>0$ such that for sufficiently small $\rho$,  for any point $p\in M$and any pair $(x,y), (v,w)\in T_pM\times ST_pM$ with $||x||,||v||\leq r$, the following inequality
holds.
\[
\frac{d(\Delta'_{M}(\exp x,\exp ' y),\Delta'_M (\exp v,\exp ' w))}{d(\Delta'_{T_pM}
(x,y),\Delta'_{T_pM}(v,w))}>1-\frac{1}{100}\quad \cdots (**),  
\]
where $\exp $ is the exponential map at $p$ and $\exp'$ is its differential. Combining the inequalities (*) and (**), for $(x,y),(v,w)\in \TM$,  we see

$d_{\TM^2}(\Delta'(x,y),\Delta'(v,w)) >d_{\TM} ((x,y),(v,w))$ if $d_{M}(x,v)\leq r$.  For the case of $d_M(x,v)> r$,   if we take  $\rho$ sufficiently small relative to $r$,  the following inequality holds. 
\[
\frac{d(\Delta'(x,y),\Delta'(v,w))}{d(\Delta(x,y),\Delta(v,w))}>1-\frac{1}{100}\text{\  for \ }(x,y),\ (v,w)\in \TM \text{\ with\ }d(x,v)>r.
\]
Here, $\Delta:\TM\to\TM^{\times 2}$ is the usual diagonal. Then,  if $d_{M}(x,v)> r$, we have the  inequality
\[
d(\Delta'(x,y),\Delta'(v,w))>\Bigl(1-\frac{1}{100}\Bigr)\,\sqrt[]{2}\,d((x,y),(v,w)).
\]
Thus, we have shown the lemma.
\end{proof}
The following lemma is an exercise of Riemannian geometry. 
\begin{lem}\label{Ldiagonaltube}
For any sufficiently small positive number $\rho$ the following condition holds. For any $n\geq 2$, $G\in \GG(n)$ and set of positive numbers $\{\eps_{ij}\mid i<j,\ (i,j)\in E(G)\}$ satisfying $\underset{(i,j)\in E(G)}{\sum} \eps_{ij}<\rho$, the inclusion of subspaces of $M^{\times n}$
\[
\{(x_1,\dots, x_n)\mid \  ^\forall (i,j)\in E(G)\ \ x_i=x_j\  \}\longrightarrow  \{(x_1,\dots, x_n)\mid \  ^\forall (i,j)\in E(G)\ \ d(x_i,x_j)\leq\eps_{ij}\  \}
\]
is a homotopy equivalence. \hspace{\fill} {\qedsymbol}
\end{lem}
\noindent {\bf Assumption:}\ In the rest of paper, we fix the number $\rho$ so that Lemmas \ref{Lexpanding} and \ref{Ldiagonaltube} hold. \\

\indent We shall define a $\DD_1$-comodule $\TCCM$ of NUCSRS. 
We set 
\[
\MT(n)=\NT \quad \text{for}\quad N= \TM ^{\times n},
\]
see Definition \ref{Dtubulernbd}. We first define a subspectrum $\TCCM(\configc )\subset \MT(n)$ as follows.
\[
\TCCM(\configc )_k=\Bigl\{[e,\epsilon ,u]\in \MT(n)_k\  \Bigl| \ \epsilon<\frac{\rho }{2}\cdot \min \{\,|c_1|,\dots, |c_n|\,\}\  \Bigr\}
\] 
We define a subspectrum $\TCCM(n)\subset \Map(\DD_1(n), \MT(n))$ as follows.
\[
\phi\in \TCCM(n)_k\iff {}^{\forall} \configc\in\DD_1(n),\ \phi(\configc)\in \TCCM(\configc)_k\, . 
\]
It is clear that the inclusion $\TCCM(n)\to \Map(\DD_1(n), \MT(n))$ is a level-equivalence for any $n\geq 1$. We denote the sequence $\{\TCCM (n)\}$ by $\TCCM$.\\
\indent We shall define an action of $\Sigma_n$ on $\TCCM(n)$, with which we regard $\TCCM$ as a symmetric sequence. For $\configc=(c_1,\dots,c_n)\in\DD_1(n)$ and $\sigma \in\Sigma_n$, we define $\configc^{\, \sigma}\in \DD_1(n)$ to be the configuration of the sub-intervals of length $|c_{\sigma(1)}|,\ |c_{\sigma(2)}|,\dots,|c_{\sigma(n)}|$ placed from the side of $-1/2$ to the side of $1/2$.  For $[e,\epsilon,u]\in \MT(n)_k$ and $\sigma \in \Sigma_n$, we set $[e,\epsilon,u]_\sigma=[e\circ \underline{\sigma}, \epsilon, u]$ where $\underline{\sigma}:\TM^{\times n}\to \TM^{\times n}$ is given by $(z_1,\dots, z_n)\mapsto (z_{\sigma^{-1}(1)},\dots z_{\sigma^{-1}(n)})$. (To distinguish the action of $\Sigma_k$ which is a part of structure of spectrum, we use the subscript $[-]_\sigma$.)  
\begin{defi}\label{DTCCMaction}
With the above notations, for $\phi\in \TCCM(n)_k$ and $\sigma\in \Sigma_n$, we define an element $\phi^\sigma\in \TCCM(n)_k$
\[
\phi^\sigma(\configc)= \{ \phi(\configc^{\sigma^{-1}}) \}_\sigma .
\]
Clearly, $\phi\mapsto \phi^\sigma$ gives a  $\Sigma_n$-action on $\TCCM(n)$.
\end{defi}
\indent In order to define a partial composition on $\TCCM$, we shall define a map 
\[
\Xi=\Xi[\configd,\configc;i\,]:\MT (n+m-1) \to \MT (n).
\]
For an element $[e,\eps, u] \in \MT(n+m-1)_k$, 
we put 
\begin{itemize}
\item $e'=e \circ (1_{i-1}\times  \Delta'\times 1_{n-i}): \TM^{\times n} \to \RR^{k}$,  \\
where $\Delta'=\Delta[\configd,\configc;i]$ and $1_{l}$ is the identity on $\TM^{\times l}$, and  \vspace{1mm}
\item $\displaystyle \epsilon'=\frac{1}{8^{m-1}}\min\{\eps,\  L(e,\configd\circ_i\configc )\}$,  where $L(e,\configc')$ is the minimum of the numbers \\
$L(e\circ \Delta[\configc_1,\configc_2;j])$ where the triple $(\configc_1,\configc_2, j)$ runs through those which satisfy \\$\configc'=(\configc_1\circ_j\configc_2)\circ_l\configc_3$ for some configuration $\configc_3$ and number $l$. 
\end{itemize}
 By Lemma \ref{Lexpanding}, $e'$ is expanding. We set $\Xi ([e,\eps,u])=[e',\eps', u]$. 
Clearly, $\Xi$ is well-defined map of spectra. 
\begin{defi}\label{DTCCMcomposition} We use the above notations.
\begin{itemize}
\item We define  a partial composition
\[
(-\circ_i-): \DD_1(m)\hotimes \TCCM(n+m-1)\longrightarrow \TCCM(n)
\]
 on $\TCCM$ by setting 
\[
(\configc\circ_i\phi)(\configd)=\Xi\bigl(\,\phi(\configd\circ_i\configc)\, \bigr),\quad \text{where}\quad \Xi=\Xi[\configd,\configc;i\,], 
\]
for elements $\phi \in \TCCM(n+m-1),\ \configc\in \DD_1(m),\ \configd\in \DD_1(n)$.\\
\item We  define a multiplication $\tilde \mu:\TCCM(n)\wedge \TCCM(n)\to \TCCM(n)$ by 
\[
\tilde\mu(\langle \phi_1,\phi_2;\sigma\rangle )(\configd)=\mu(\langle \phi_1(\configd),\phi_2(\configd); \sigma\rangle  )
\] where $\mu$ denotes the multiplication given in Definition \ref{Dtubulernbd}.
\end{itemize}
With these operations and the action of $\Sigma_n$ in Definition \ref{DTCCMaction}, we regard $\TCCM$ as a $\DD_1$-comodule of NUCSRS.
\end{defi}
\begin{lem}\label{LTCCMwelldef}
The structure of  $\DD_1$-comodule of NUCSRS on $\TCCM$ given in Definition \ref{DTCCMcomposition} is well-defined. 
\end{lem}
\begin{proof}
By Lemma \ref{Ldiagonalcomposition}, we see the equality of (1) in Definition \ref{Dcontra-module} holds. The equality in (2) in the same definition is clear.\\
\indent We shall prove the equality in (3).  Take elements $\configc \in \DD_1(m), \configd\in \DD_1(n),\ \phi \in \TCCM(m+n-1), $ and $\sigma\in \Sigma_n$. By definition,
\[
\begin{split}
(\configc\circ_i \phi)^\sigma(\configd)& =\{\configc\circ_i\phi(\configd^{\,\sigma^{-1}})\}_{\sigma} \\
&=\bigl\{\Xi_1(\phi\bigl(\configd^{\,\sigma^{-1}}\circ_i\configc)\bigr)\bigr\}_{\sigma},\\
\configc\circ_{\sigma^{-1}(i)}(\phi^{\sigma_1})(\configd )&=\Xi_2\bigl\{
\phi\bigl((\configd\circ_{\sigma^{-1}(i)}\configc)^{\sigma_1^{-1}}\bigr)_{\sigma_1}\bigr\},
\end{split}
\]
where $\Xi_1=\Xi{[\configd^{\,\sigma^{-1}},\configc;i\,]}$, and $\Xi_2= \Xi{[\configd,\configc;\sigma^{-1}(i)\,]}$. It is easy to check the following equalities:
\[
\configd^{\,\sigma^{-1}}\circ_i\configc  = (\configd\circ_{\sigma^{-1}(i)}\configc)^{\sigma_1^{-1}}, \qquad
\bigl\{\Xi_1(x)\bigr\}_\sigma =\Xi_2 (x_{\sigma_1}) 
\]
By these equalities, we have verified the desired equality. Compatibility of the multiplication with the partial composition is obvious.
\end{proof}

\subsection{Construction of the comodule $\CCM$}
 For an element $[e,\eps, u]\in \MT(n)$  which is not the base point,  put $\pi_e(u)=((x_1,y_1),\dots, (x_n,y_n))$ with $x_i\in M$ and $y_i\in ST_{x_i}M$.   Then we define $\DeltaT_{p q}(\configc)\subset \TCCM (\configc)$ by the following equivalence.
\[
[e,\eps, u ]\in \DeltaT_{pq}(\configc )_k \ \iff \ [e,\eps, u]=*, \text{ or }\  d_M(x_p,x_q)\leq \delta_{pq}(\configc \,,\eps )
\]
where
\[
\delta_{pq}(\configc\, ,\eps )=
\frac{\rho}{2}\bigl(\,|c_p|+|c_q|\,\bigr)-\eps ,
\]
which is a positive number by the definition of $\TCCM (\configc)$.
Define a subspectrum $\DeltaT_{pq}(n)\subset \TCCM(n)$ by
\[
\phi\in \DeltaT_{pq}(n)_k \ \iff \ {}^{\forall} \configc\in\DD_1(n),\  \phi(\configc)\in \DeltaT_{pq}(\configc )_k
\]
The following lemma is a key to define the comodule $\CCM$. Most of  the technical definitions by now are necessary to make this lemma hold.
\begin{lem}\label{Ldiagonalinclusion}
\begin{enumerate}
\item For any numbers $n\geq 1$, $m\geq 2$ and element $\configc\in \DD(m),$ let $\configc\circ_i\DeltaT_{pq}(n+m-1)$ denotes the subspectrum of $\TCCM(n)$ consisting of the images of $\configc\circ_i(-)$. We have the following inclusion at each level $k$.
\[
\configc\circ_i\DeltaT_{pq}(n+m-1)
\subset
\left\{
\begin{array}{ll}
\{*\} & (i\leq p<q\leq i+m-1)\vspace{1mm}\\
\DeltaT_{p\,i}(n) & (p<i\leq q\leq i+m-1)\vspace{2mm}\\
\DeltaT_{p,\, q-m+1}(n) & (p<i,i+m-1<q)\vspace{2mm}\\
\DeltaT_{i,\, q-m+1}(n) & (i\leq p\leq i+m-1<q)\vspace{2mm}\\
\DeltaT_{p-m+1,\, q-m+1}(n) & (i+m-1<p<q)
\end{array}
\right.
\]
More precisely, for example, the second inclusion means $\configc\circ_i\DeltaT_{pq}(n+m-1)_k\subset \DeltaT_{p\,i}(n)_k$ for each $k$.
\item The image  of $\DeltaT_{pq}(n)\wedge \TCCM(n)$ by the multiplication $\tilde \mu$ given in Definition \ref{DTCCMcomposition} is contained in $\DeltaT_{pq}(n)$.
\end{enumerate}
\end{lem}
\begin{proof}
We shall show the part 1. The proof of the part 2 is similar and omitted.  Let $\configc\in \DD_1(m)$ and $\configd\in \DD_1(n)$ and $\phi\in \DeltaT_{pq}(n+m-1)_k$. Let $(e,\eps, u)$ be a representative of $\phi(\configd\circ_i\configc)$. 
Put 
\[
\begin{split}
((x_1,y_1),\dots, (x_{n+m-1},y_{n+m-1}))&=\pi_e(u), \\ 
((x'_1,y'_1),\dots, (x'_{n+m-1},y'_{n+m-1}))&=\{(1_{i-1})\times \Delta'\times (1_{n-i})\}(\pi_{e'}(u))
\end{split}
\]
under the notations given in the paragraph above Definition \ref{DTCCMcomposition}.
 We shall show the first inclusion, the case of $i\leq p<q\leq i+m-1$.  
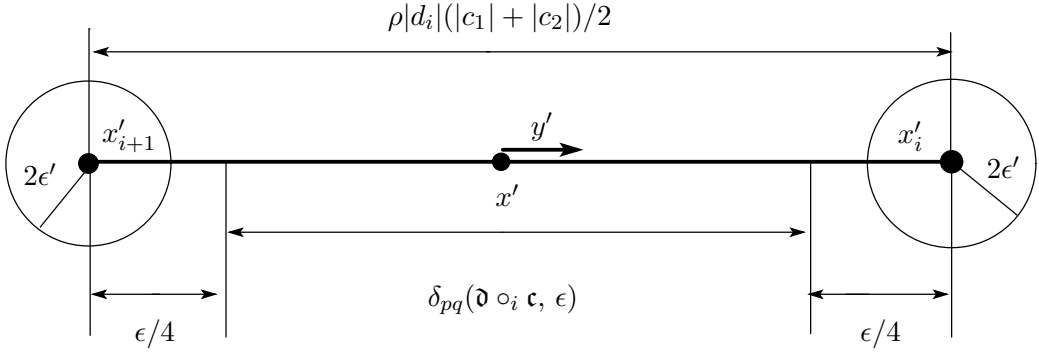
\begin{figure}
\begin{center}
{\unitlength 0.1in%
\begin{picture}(94.8300,17.3900)(0.3000,-18.9400)%
%
\special{pn 8}%
\special{pa 9513 1631}%
\special{pa 9492 1571}%
\special{fp}%
%
\special{pn 20}%
\special{pa 506 976}%
\special{pa 5006 976}%
\special{fp}%
%
\special{sh 1.000}%
\special{ia 2657 976 45 45 0.0000000 6.2831853}%
\special{pn 8}%
\special{ar 2657 976 45 45 0.0000000 6.2831853}%
%
\special{sh 1.000}%
\special{ia 5015 976 58 58 0.0000000 6.2831853}%
\special{pn 8}%
\special{ar 5015 976 58 58 0.0000000 6.2831853}%
%
\special{sh 1.000}%
\special{ia 497 985 51 51 0.0000000 6.2831853}%
\special{pn 8}%
\special{ar 497 985 51 51 0.0000000 6.2831853}%
%
\special{pn 20}%
\special{pa 2666 913}%
\special{pa 3053 913}%
\special{fp}%
\special{sh 1}%
\special{pa 3053 913}%
\special{pa 2986 893}%
\special{pa 3000 913}%
\special{pa 2986 933}%
\special{pa 3053 913}%
\special{fp}%
%
\special{pn 8}%
\special{ar 5015 985 441 441 0.0000000 6.2831853}%
%
\special{pn 8}%
\special{ar 497 985 0 0 4.7327943 4.7123890}%
%
\special{pn 8}%
\special{ar 497 985 432 432 0.0000000 6.2831853}%
%
\special{pn 8}%
\special{pa 515 985}%
\special{pa 245 1318}%
\special{fp}%
%
\special{pn 8}%
\special{pa 506 976}%
\special{pa 506 1885}%
\special{fp}%
%
\special{pn 8}%
\special{pa 5015 976}%
\special{pa 5015 1894}%
\special{fp}%
%
\special{pn 8}%
\special{pa 1217 976}%
\special{pa 1217 1894}%
\special{fp}%
%
\special{pn 8}%
\special{pa 4277 985}%
\special{pa 4277 1876}%
\special{fp}%
%
\special{pn 8}%
\special{pa 5015 976}%
\special{pa 5357 1255}%
\special{fp}%
%
\special{pn 8}%
\special{pa 2666 1345}%
\special{pa 1235 1345}%
\special{fp}%
\special{sh 1}%
\special{pa 1235 1345}%
\special{pa 1302 1365}%
\special{pa 1288 1345}%
\special{pa 1302 1325}%
\special{pa 1235 1345}%
\special{fp}%
\special{pa 2666 1345}%
\special{pa 4241 1345}%
\special{fp}%
\special{sh 1}%
\special{pa 4241 1345}%
\special{pa 4174 1325}%
\special{pa 4188 1345}%
\special{pa 4174 1365}%
\special{pa 4241 1345}%
\special{fp}%
\put(26.9000,-11.5000){\makebox(0,0){$x'$}}%
\put(28.7000,-7.9000){\makebox(0,0){$y'$}}%
\put(26.6600,-16.9600){\makebox(0,0){$\delta_{pq}(\configd\circ_i\configc,\,\eps)$}}%
\put(8.4500,-18.7500){\makebox(0,0){$\eps/4$}}%
\put(46.4300,-18.7500){\makebox(0,0){$\eps/4$}}%
\put(2.4500,-10.3900){\makebox(0,0){$2\eps'$}}%
\put(52.8500,-10.0300){\makebox(0,0){$2\eps'$}}%
\put(7.1000,-8.4000){\makebox(0,0){$x'_{i+1}$}}%
\put(48.0000,-8.4000){\makebox(0,0){$x'_{i}$}}%
%
\special{pn 8}%
\special{pa 500 970}%
\special{pa 500 280}%
\special{fp}%
\special{pa 500 280}%
\special{pa 500 280}%
\special{fp}%
%
\special{pn 8}%
\special{pa 5010 960}%
\special{pa 5010 300}%
\special{fp}%
%
\special{pn 8}%
\special{pa 2600 400}%
\special{pa 530 400}%
\special{fp}%
\special{sh 1}%
\special{pa 530 400}%
\special{pa 597 420}%
\special{pa 583 400}%
\special{pa 597 380}%
\special{pa 530 400}%
\special{fp}%
\special{pa 530 400}%
\special{pa 530 400}%
\special{fp}%
%
\special{pn 8}%
\special{pa 2580 400}%
\special{pa 5010 400}%
\special{fp}%
\special{sh 1}%
\special{pa 5010 400}%
\special{pa 4943 380}%
\special{pa 4957 400}%
\special{pa 4943 420}%
\special{pa 5010 400}%
\special{fp}%
\put(26.5000,-2.2000){\makebox(0,0){$\rho |d_i|(|c_1|+|c_2|)/2$}}%
%
\special{pn 8}%
\special{pa 870 1670}%
\special{pa 520 1670}%
\special{fp}%
\special{sh 1}%
\special{pa 520 1670}%
\special{pa 587 1690}%
\special{pa 573 1670}%
\special{pa 587 1650}%
\special{pa 520 1670}%
\special{fp}%
\special{pa 830 1670}%
\special{pa 1160 1670}%
\special{fp}%
\special{sh 1}%
\special{pa 1160 1670}%
\special{pa 1093 1650}%
\special{pa 1107 1670}%
\special{pa 1093 1690}%
\special{pa 1160 1670}%
\special{fp}%
\special{pa 1160 1670}%
\special{pa 1160 1670}%
\special{fp}%
%
\special{pn 8}%
\special{pa 4580 1670}%
\special{pa 4270 1670}%
\special{fp}%
\special{sh 1}%
\special{pa 4270 1670}%
\special{pa 4337 1690}%
\special{pa 4323 1670}%
\special{pa 4337 1650}%
\special{pa 4270 1670}%
\special{fp}%
%
\special{pn 8}%
\special{pa 4460 1670}%
\special{pa 4990 1670}%
\special{fp}%
\special{sh 1}%
\special{pa 4990 1670}%
\special{pa 4923 1650}%
\special{pa 4937 1670}%
\special{pa 4923 1690}%
\special{pa 4990 1670}%
\special{fp}%
\end{picture}}%
\end{center}
\caption{the first inclusion of Lemma \ref{Ldiagonalinclusion} (1) with $n=2$\ :\ The bold line is a part of the geodesic segment used to define $\Delta'$. $(x',y')$ is the $i$-th component of $\pi_{e'}(u)\in \TM^{\times n}$. $x_i$ and $x_{i+1}$ exist in the interior of the disks at $x_i'$ and $x_{i+1}'$ if  $(\configc \circ_i\phi)(\configd)\not=*$. }\label{Fcomposition}
\end{figure}

The situation of the case $n=2$ is as in Figure \ref{Fcomposition} (so $p=i,\ q=i+1$).
We first give a sketch of the proof for $n=2$.
We suppose $(\configc \circ_i\phi)(\configd)\not=*$ and will show a contradiction. Since the map $\Delta'$ arranges points along a geodesic and the length of the geodesic segment between $x'_i$ and $x'_{i+1}$ is $\rho |d_i|(|c_1|+|c_2|)/2$, we have $d_M(x'_i,x'_{i+1})>\delta(\configd\circ\configc,\eps)$.  As we have taken $\eps'$ in the  definition of $\Xi$ sufficiently small , $x_i$ and $x'_{i}$ (resp. $x_{i+1}$ and $x'_{i+1}$) are sufficiently close. These observations imply $d_M(x_i,x_{i+1})>\delta(\configd\circ\configc,\eps)$ or  equivalently, $\phi(\configd\circ_i\configc)\not\in\DeltaT_{pq}(\configd\circ_i \configc)$. \\
\indent We shall give the formal proof. Since the image of $e'$ is contained in the image of $e$, Under the assumption $(\configc \circ_i\phi)(\configd)\not=*$, since the map $\pi_e$ sends $u$ to its closest point in $e(M)=M$, we have 
\[
||u-e(\pi_e(u))||\leq ||u-e'(\pi_{e'}(u))||<\epsilon'.
\]
So we have 
\[
||e'(\pi_{e'}(u))-e(\pi_e(u))||\leq ||e'(\pi_{e'}(u))-u||+||u-e(\pi_e(u))||<2\epsilon'.
\] As $e'=e\circ (1_{i-1})\times \Delta'\times (1_{n-i})$, and $e$ is expanding, we have
\[
d\{\, ((x'_1,y'_1),\dots, (x'_{n+m-1},y'_{n+m-1})),\ ((x_1,y_1),\dots, (x_{n+m-1},y_{n+m-1}))\,\}<2\eps',
\] 
where $d$ denotes the distance in $(\TM)^{\times n+m-1}$. So we have
\[
d_M(x_j,x'_j)\leq d_{\TM}((x_j,y_j),(x'_j,y'_j))<2\epsilon' \text{\ for\ }j=1,\dots,n+m-1
\]
By this inequality and the definition of the map $\Delta'$, we have the following inequality.
\[
\begin{split}
d_M(x_p,x_q)& \geq d_M(x'_p,x'_q)-d_M(x_p,x'_p)-d_M(x_q,x'_q) \\
            & \geq d_M(x'_p,x'_q)-4\epsilon' \\
            & \geq  \frac{\rho}{2}\cdot  |d_i|\bigl(\,|c_{p-i+1}|+|c_{q-i+1}|\,\bigr)-4\epsilon'\\
            & = \frac{\rho}{2} \bigl(\,|(\configd\circ_i\configc)_{p}|+|(\configd\circ_i\configc)_{q}|\,\bigr)-4\epsilon'\\
            & \geq \frac{\rho}{2} \bigl(\,|(\configd\circ_i\configc)_{p}|+|(\configd\circ_i\configc)_{q}|\,\bigr)-\eps /2 \\
            &    > \delta_{pq}(\configd \circ_i \configc\, ,\eps) 
\end{split}
\]
This inequality means $\phi(\configd\circ_i\configc)\not\in \DeltaT_{pq}(\configd\circ_i \configc)$, which is a contradiction. So $(\phi\circ_i\configc)(\configd)=*$ and we have proved the first inclusion.\\
\indent We shall show the second inclusion, the case of $p<i\leq q\leq i+m-1$. Let $(x',y')\in \TM$ be the $i$-th component of $\pi_{e'}(u)$. Clearly, we have 
\[
((x'_i,y'_i),\dots, (x'_{i+m-1},y'_{i+m-1}))=\Delta '(x',y').
\]
By an argument similar to the above, we have the following inequality.
\[
\begin{split}
d_M(x'_p,x') & \leq d_M(x'_p,x_p)+d_M(x_p,x_q)+d_M(x_q,x_q')+d_M(x'_q,x') \\
            & \leq 2\eps'+ \delta_{pq}(\configd\circ_i\configc\, ,\epsilon )+2\epsilon'+\frac{\rho}{2}\cdot  |d_i|\bigl(\, 1-|c_{q-i+1}|\, \bigr)\\
            &=\frac{\rho}{2} \bigl(\, |d_p|+|d_i||c_{q-i+1}|\, \bigr)-\epsilon
+4\epsilon'+\frac{\rho}{2} |d_i|\bigl(\, 1-|c_{q-i+1}|\, \bigr)\\
&\leq \frac{\rho}{2}\bigl( \, |d_p|+|d_i|\,\bigr)-\epsilon /2\\
&< \delta_{pq}(\configd\, ,\epsilon').
\end{split}
\]
This implies the second inclusion. The other cases are similar to the first and second cases.
\end{proof}
Let $\DeltaT_{\fat}(n)$ be the subspectrum of  $\TCCM(n)$ whose space at level $k$ is given by
\[
\DeltaT_{\fat}(n)_k=\underset{1\leq p<q\leq n}{\bigcup} \DeltaT_{pq}(n)_k.
\]
Since $\{\DeltaT_{pq}(n)\}^\sigma=\DeltaT_{\sigma^{-1}(p),\sigma^{-1}(q)}(n)$, $\DeltaT_{\fat}(n)$ is stable under action of $\Sigma_n$.
By Lemma \ref{Ldiagonalinclusion}, the sequence $\{\Delta_{\fat}^{-\tau}(n)\}_{n\geq 0}$ is stable under partial compositions and is an ideal for the multiplication $\tilde \mu$. So the sequence $\{\Delta_{\fat}^{-\tau}(n)\}_{n\geq 0}$ inherits structure of a { comodule } from $\TCCM$ we can define the quotient comodule as follows.
\begin{defi}
We define a  spectrum $\CCM(n)$ by the quotient (collapsing to $*$)
\[
\CCM (n)_k =\TCCM (n)_k/\DeltaT_{\fat}(n)_k
\] 
for each $k\geq 0$. We regard the sequence $\CCM=\{\CCM(n)\}_{n\geq 1}$ as a comodule of NUCSRS with the structure  induced by that on $\TCCM$.
\end{defi}

\section{Atiyah duality for comodules}\label{Satiyahdualty}
\begin{defi}\label{Dconfigurationmodule}
 We shall define the following zigzag consisting of $\DD_1$-comodules of NUCSRS and  maps between them.
\[
(\CSINHA)^\vee\stackrel{(i_0)^\vee}{\longleftarrow}(\widetilde{F}^M)^\vee \stackrel{(i_1)^\vee }{\longrightarrow} (F^M)^\vee \stackrel{q_*}{\longleftarrow} 
F'_M\stackrel{p_*}{\longrightarrow} F_M\stackrel{\Phi}{\longleftarrow}\CCM
\]

\begin{itemize}
\item Set $\CSINHA (n)=\CC^{n-1}\langle [M]\rangle $. When we regard a configuration as an element of $\CSINHA (n)$, we label its points by $1,\dots,n$ instead of $0,\dots,n-1$. We give the sequence $\CSINHA=\{\CSINHA(n) \}_{n\geq 1}$ a structure of a  $\ASS$-module as follows. For the unique element $\mu\in \ASS(2)$ and an element $x\in \CSINHA(n)$, we set $x\circ_i\mu=d^{i-1}(x)$ where, $d^{i-1}$ is the coface operator of $\CC^\bullet \langle [M]\rangle$. The action of $\Sigma_n$ on $\CSINHA(n)$ is given by  permutations of labels.  $(\CSINHA)^{\vee}$ is the $\ASS$-comodule of NUCSRS given in Definition \ref{Dcontra-module}. By pulling back the action by the unique operad morphism $\DD_1\to \ASS$, we also regard $(\CSINHA)^\vee$ as a  $\DD_1$-comodule. 
\item Let $F^M(n)$ be the subspace of $\DD_1(n)\times \TM^{\times n}$ defined by the following condition. For an element $(\configc;(x_1,y_1), \dots, (x_n,y_n))\in \DD_1(n)\times \TM^{\times n}$ with $x_i\in M$ and $y_i\in ST_{x_i}M$, 
\[
\begin{split}
(\configc ;(x_1,y_1), & \dots, (x_n,y_n))\in F^M(n)\\
& \iff d(x_i,x_j)\geq \frac{\rho}{2} (\,|c_i|+|c_j|\, ) \text{ for each pair $(i,j)$ with $i\not =j$,} 
\end{split}
\]
where $\rho$ is the  number fixed in subsection \ref{SSTCCM}.
\item The sequence $\{F^M(n)\}$ has a structure of a $\DD_1$-module. For $\configc\in \DD_1(n)$ and $(\configd;z_1,\dots,z_n)\in F^M(n)$, we set
$(\configd;z_1,\dots, z_n)\circ_i\configc=(\configd\circ_i\configc;z_1,\dots,\Delta'(z_i),\dots,z_n)$, where $\Delta'=\Delta[\configd\, ,\configc\,;i]$ is given in Definition \ref{Ddelta}. The symmetric group acts on $F^M(n)$ by permutations of little intervals and components
.  The $\DD_1$-comodule of NUCSRS $(F^M)^\vee$ is the one induced by $F^M$. 

\item We shall define a symmetric sequence of spectra $\{\Sphere_M(n)\}_{n}$. Set  $\tilde{\Sphere}_M(n)_k=\TNT_k$  for $N=\TM^{\times n}$ (see Definition \ref{Dtubulernbd}).
Define a subspace
$\partial (\tilde{\Sphere}_M(n))_k\subset \tilde{\Sphere}_M(n)_k$ by
$(e,\epsilon, v)\in\partial \tilde{\Sphere}_M(n)_k\iff ||v||\geq \epsilon$. We put
\[
\Sphere_M(n)_k=\tilde{\Sphere}_M(n)_k/\partial \tilde{\Sphere}_M(n)_k\ .
\]
We regard $\Sphere_M(n)$ as a NUCSRS by a multiplication defined similarly to the one of $\NT$ given in Definition \ref{Dtubulernbd}.
\item Set $F_M(n):=\Map(F^M(n),\Sphere_M(n))$. We give the sequence $\{F_M(n)\}_{n}$ a structure of $\DD_1$-comodule as follows:
For $\configc\in \DD_1(n)$ and $f\in F_M(n+m-1)$, set $ \configc\circ_i f$ to be the following composition
\[
\xymatrix{F^M(m)\ar[r]^{(-\circ_i\configc)\qquad }&F^M(n+m-1)\ar[r]^{f}&\Sphere_M(n+m-1)\ar[r]^{\qquad \alpha}&\Sphere_M(n)}\ .
\]
Here, $\alpha$ is given by 
\[
\alpha([e,\eps, v])=[e',\eps', v]\ ,
\]
where $e'$ and $\epsilon'$ are those defined in the paragraph above Definition \ref{DTCCMcomposition}. Similarly to $(\CSINHA)^\vee$, we define a multiplication on $F_M(n)$ as the pushforward by the multiplication on $\Sphere_M(n)$.
\item We define a map $\widetilde{\Phi}_n:\TCCM(n)\to F_M(n)$ of spectra by 
\[
\widetilde{\Phi}_n (\phi)  ((\configc;z_1,\dots,z_n))
=[e,\bar \epsilon,u-e(z_1,\dots, z_n)]
\]
Here, we write $\phi(\configc)=[e,\eps, u]$ and we set 
$\bar \epsilon=\epsilon/4$. $\widetilde{\Phi}_n$ induces a morphism $\Phi_n:\CCM (n)\to F_M(n)$ which forms a morphism of comodules, as is proved in Lemma \ref{LAtiyahdual} below.

\item 
We shall define a $\DD_1$-module $\widetilde{F}^M$. Set 
\[
\widetilde{F}^M_1(n)=[0,1]\times \DD_1(n)\times \CSINHA (n)/\sim ,
\]
where the equivalence relation is generated  by the relation $(t,\configc, z) \sim (s,\configd,z')\iff (s=t=0$\ and $z=z')$. $\widetilde{F}^M(n)$ is the subspace of $\widetilde{F}^M_1(n)$ consisting of elements $(t,\configc, z=(x_k,u_{kl}, y_k))$ satisfying the condition that
 \[ 
t\not=0 \Rightarrow z \in Int (\CSINHA(n)) \text{\ and\ } d_M(x_i,x_j)\geq t\cdot \frac{\rho}{2}(\,|c_i|+|c_j|\, ).
\] 
Here, $Int(\CSINHA(n))$ is the subspace consisting of the elements $(x_k,u_{kl},y_k)$ such that $x_k\not=x_l$ if $k\not= l$, or equivalently, $(x_k,u_{kl})$ belongs to $C_n(M)$ via the canonical inclusion $C_n(M)\subset C_n\CPTM$.   
We endow the sequence $\{\widetilde{F}^M(n)\}_n$ with the  $\DD_1$-module structure analogous to that of $F^M$. The difference is that we use the number $t\rho$ instead of $\rho$ in the definition of $\Delta'$ for $t>0$ and use the module structure on $\CSINHA$ for $t=0$. The obvious inclusions $i_0:\CSINHA(n)\to \widetilde{F}^M(n)$ and $i_1:F^M(n)\to \widetilde{F}^M(n)$ to $t=0,1$  give rise to morphisms of $\DD_1$-modules $i_0:\CSINHA\to \widetilde{F}^M$, \ $i_1:F^M\to \widetilde{F}^M$.\\
\item  In order to define $F'_M$, and $p_*$,\ $q_*$, we shall define a symmetric sequence of symmetric spectra $\{{\Sphere}'_M(n)\}_n$.
Let  $\widetilde{\Sphere}'_M(n)$ be  the subspace of $\Emb((\TM)^{\times n},\RR^k)\times \RR\times S^k$ consisting of triples $(e,\epsilon, v)$ with $0<\eps<L(e)$. 
We put 
\[
\Sphere'_M(n)_k=\widetilde{\Sphere}'_M(n)_k/\{(e,\epsilon, \infty)\mid e,\ \eps\ \text{arbitrary} \}.
\]
where we regard $S^k=\RR^k\cup\{\infty\}$ We regard $\Sphere'_M(n)$ as a symmetric spectrum analogously to $\Sphere_M(n)$. Let $p:\Sphere'_M(n)\to \Sphere_M(n)$  be the map induced by the collapsing map \\
$S^k\to \RR^k/\{v\mid ||v||\geq \epsilon \}$ and $q:\Sphere'_M\to \Sphere$ be the map forgetting the data $(e,\epsilon)$. Set 
$F'_M(n)=\Map (F^M(n),\Sphere'_M(n))$. We regard $\{F'_M(n)\}$ as a $\DD_1$-comodule of NUCSRS analogously to $F_M$. the pushforwards $p_*$ and $q_*$ are clearly morphisms of comodules of NUCSRS.  
\end{itemize}
\end{defi}
Verification of well-definedness of the objects defined in Definition \ref{Dconfigurationmodule} is  routine work. For example, the associativity of the composition of $\CSINHA$ follows from the  cosimplicial identities of $\CC^\bullet\CPTM$, and that of $F^M$ can be verified similarly to the associativity of little cubes operads. We omit details.
\begin{lem}\label{LAtiyahdual}

$\widetilde{\Phi}_n$ uniquely factors through a map $\Phi_n:\CCM(n)\to F_M(n)$, and the sequence $\{\Phi_n\}$ is a map of  $\DD_1$-comodules of NUCSRS.
\end{lem}
\begin{proof}
We shall show that for any element $\phi\in\DeltaT_{pq}(n)$, $\widetilde{\Phi}_n(\phi)=*$. Suppose that there exists an element $(\configc\,;z_1,\dots,z_n)\in F^M(n)$ such that $\widetilde{\Phi}_n(\phi)(\configc\,;z_1,\dots,z_n)\not=*\in \Sphere_M(n)$. If we put $\phi(\configc)=[e,\eps,u]$, the inequality $|| u-e(z_1,\dots, z_n)||<\eps/4$ holds. So we have $||u_i-e(\pi_{e}u)||<\eps/4$. Thus we have
\[
||e(\pi_e u)-e(z_1,\dots, z_n)||\leq ||e(\pi_e u)-u||+||u-e(z_1,\dots, z_n)||<\eps/2
\]
As $e$ is expanding, we have $d(\pi_e(u),(z_1,\dots, z_n))<\eps/2$ where $d$ denotes the distance in $\TM^{\times n}$. If we write $z_i=(x_i,y_i)$ and $\pi_e(u)=((\bar x_1,\bar y_1),\dots, (\bar x_n,\bar y_n))$ as  pairs of a point of $M$ and a tangent vector, it follows that $d_M(\bar x_i,x_i)<\eps/2$, and  the following inequality.
\[
\begin{split}
d(\bar x_p,\bar x_q) & \geq  d(x_p,x_q)-d(x_p,\bar x_p )-d(x_q,\bar x_q )\\
 & > \frac{\rho}{2} (\,|c_p|+|c_q|\, )-\eps  
  = \delta_{pq}(\configc\, ,\eps ).
\end{split}
\]
This inequality contradicts the assumption $\phi\in \DeltaT_{pq}(n)$. Thus we have proved $\widetilde{\Phi}_n(\DeltaT_{pq}(n))=*$. This implies the former part of the lemma. The latter part is obvious.
\end{proof}
\begin{defi}\label{Dsemistablemodule}
A $\DD_1$-comodule of NUCSRS is {\em strongly semistable} if the spectrum $X(n)$ is strongly semistable for each $n$. A map $f:X\to Y$ of $\DD_1$-comodules of NUCSRS is a {\em $\pi_*$-isomorphism} if each map $f_n:X(n)\to Y(n)$ is a $\pi_*$-isomorphism. 
\end{defi}

The following is a version of Atiyah duality which respects our comodules. We devote the rest of this section to its proof.

\begin{thm}\label{TAtiyahdual}
As  $\DD_1$-comodule of NUCSRS, $(\CSINHA)^\vee$ and $\CCM$ are $\pi_*$-isomorphic. Precisely speaking, all the comodules in the zigzag in Definition \ref{Dconfigurationmodule} are strongly semistable and all the maps in the same zigzag are $\pi_*$-isomorphisms.\\

\end{thm}

\begin{defi}\label{Dmap}
\begin{itemize}
\item For $G\in\GG(n)$, and $\configc\in \DD_1(n)$, we define two subspectra $\DeltaT_G(\configc),\  \DeltaT_{\fat}(\configc) \subset \TCCM(\configc)$ by 
\[
\DeltaT_G(\configc )=\left\{
\begin{array}{cc}
\underset{(p,q)\in E(G)}{\bigcap}\DeltaT_{pq}(\configc) & (G\not=\emptyset) \vspace{2mm}\\ 
\TCCM(\configc)  & (G=\emptyset),
\end{array}\right.  \qquad \DeltaT_{\fat}(\configc)=\underset{1\leq p<q\leq n}{\bigcup}\DeltaT_{pq}(\configc).
\]
 Similarly, we define a subspectrum $\DeltaT_G\subset \TCCM(n)$ by 
\[
\DeltaT_G=\left\{
\begin{array}{cc}
\underset{(p,q)\in E(G)}{\bigcap}\DeltaT_{pq} &  (G\not=\emptyset) \vspace{2mm} \\
 \TCCM(n)  &  (G=\emptyset).
\end{array}\right.
\]
Here, the union and intersections are taken in the level-wise manner.
\item We fix an expanding embedding $e_0:\TM\to \RR^K$  and a positive number $\eps_0<L(e_0)$ and a configuration $\configc_0\in\DD_1(n)$ such that $\eps_0<\frac{1}{4}\min\{|c_1|,\dots,|c_n|\}$. We set $\nu=\nu_{\epsilon_0}(e_0)$. We impose an additional condition on $\eps_0$  in Definition \ref{Diroiro}, which is satisfied by any sufficiently small $\eps_0$,
  and  we will assume $K$ is a multiple of $4$ in the proof of Theorem \ref{Tam}. (We may impose the assumption on $K$ from the beginning, but for the convenience of verification of signs, we do not do so.)  

\item Consider $\nu^{\times n}\subset \RR^{nK}$ as a disk bundle over $\TM^{\times n}$ and denote by $\nu_G$ the restriction $\nu^{\times n}|_{\Delta_G}$ (see Introduction).  Let $\lambda_G:Th(\nu_G)\to \DeltaT_G(\configc_0)_{nK}$ be the map $[u]\mapsto  [(e_0)^{\times n},\epsilon_0,u]$. $\lambda_G$ induces a morphism $\lambda_G:\Sigma^{nK}Th(\nu_G)\to \DeltaT_G(\configc_0)$ in  $\Ho(\SP)$, where $\Sigma$ denotes the suspension.

\end{itemize}
\end{defi}
\begin{lem}\label{Lexpandinghomotopy}
For any closed smooth manifold $N$ and $k\geq 1$, the inclusion $I:\Emb^{ex}(N,\RR^k)\to \Emb(N,\RR^k)$ is a homotopy equivalence.
\end{lem}
\begin{proof}
Let $f:\RR_{>0}\to \RR$ be a $C^\infty$-function which satisfies the following inequality
\[
f(x)>\frac{1}{x}\quad (x<1),\qquad f(x)\geq 1 \quad (x\geq 1)
\]
We define a continuous map $F:\Emb(N,\RR^k)\to \Emb^{ex}(N,\RR^k)$ by $e\mapsto f(r(e))\cdot e$ where $r(e)$ is the number given in Definition \ref{Dtubulernbd}, and $\cdot$ denotes the component-wise scaler multiplication. A homotopy from $F\circ I$ to $id$ is given by $(t,e)\mapsto \{t+(1-t)f(r(e))\}\cdot e$, and a homotopy from $I\circ F$ to $id$ is given by the same formula. 
\end{proof}
\begin{lem}\label{Tsemistable}
We use the notations in Definitioin \ref{Dmap}. For each $n\geq 1$  and $G\in\GG(n)$, $\CCM(n)$ and $\DeltaT_G$ are strongly semistable, and each map in the following zigzags  in  $\Ho(\SP)$ is an isomorphism.
\[
\begin{split}
\Sigma^{nK}Th(\nu_G)    &    \stackrel{\lambda_G}{\longrightarrow}\DeltaT_G(\configc_0)\longleftarrow \DeltaT_G\, ,    \\
\Sigma^{nK}\bigl\{ Th(\nu^{\times n})/Th(\nu^{\times n}|_{\Delta_{\fat}(n)}) \bigr\}  &    \stackrel{\lambda_G}{\longrightarrow}\DeltaT_\emptyset(\configc_0)/\{ \DeltaT_{\fat}(\configc_0)\} \longleftarrow \CCM (n).
\end{split}
\] 
Here, see Introduction for $\Delta_{\fat}(n)$, and   the right  maps are the evaluations at  $\configc_0 $. 
\end{lem}
\begin{proof}
For simplicity, we shall prove the claim for the maps in the first line for the case of $G=\emptyset$. The same proof works for general $G$ thanks to the assumption on $\rho$ given in subsection \ref{SSTCCM}. Set $N=(\TM)^{\times n}$. The evaluation at $\configc_0$ and the inclusion $\DeltaT_\emptyset(\configc_0)\subset \NT$ are clearly level equivalences. So all we have to prove is that $\DeltaT_\emptyset$ is strongly semistable and that the composition of  $\lambda_G$ and the inclusion, which is also denoted by $\lambda_G:\Sigma^{nK}Th(\nu_G)\to \NT$, is an isomorphism in $\Ho(\SP)$.    We define a space $\E_k$ by
\[
\E_k=\{(e,\eps) \mid e\in \Emb^{ex}(N,\RR^k),\ 0<\eps< L(e)\}.
\]
By Lemma \ref{Lexpandinghomotopy} and Whitney's theorem, $\E_k$ is $(k/2-n(2\dd -1)-1)$-connected. Let $P:\BNT_k\to \E_k$ be the fiber bundle obtained from the obvious projection $\TNT_k\to \E_k$ by collapsing  the complements of $\nu_\eps(e)$'s in a fiberwise manner (see Definition \ref{Dtubulernbd}). So each fiber of the map $P$ is a Thom space homeomorphic to $Th(\nu_G)$.  $P$ has a section $s:\E_k\to \BNT_k$  to the basepoints, and there is an obvious homeomorphism 
\[
\BNT_k /s(\E_k)\cong \NT_k.
\]
With this observation, by observing the Serre spectral sequence for $P$, we see the composition
\[
S^{k-nK}\wedge Th(\nu_G)\stackrel{\lambda_G}{\longrightarrow}S^{k-nK}\wedge \NT_{nK}\stackrel{\text{action of \ }\Sph}{\longrightarrow} \NT_k
\]
is $(3k/2-2n(2\dd -1)-2)$-connected. This implies $\NT$ is strongly semistable and $\lambda_G$ is an isomorphism. The same proof works for general $G$ and the maps in the second line thanks to the assumption on $\rho$ given in subsection \ref{SSTCCM}.
\end{proof}

\begin{proof}[Proof of Theorem \ref{TAtiyahdual}]
Similarly to the proof of Lemma \ref{Tsemistable}, it is easy to see $\Sphere_M$ and $\Sphere'_M$ are strongly semistable, which  implies each comodule in the zigzag in Definition \ref{Dconfigurationmodule} is strongly semistable, combined with the fact that the spaces $F^M(n)$, $\widetilde{F}^M(n)$, $\CSINHA(n)$ have homotopy types of finite CW complexes.  It is clear that $p$ and $q$ are $\pi_*$-isomorphisms, so are $p_*$ and $q_*$. $i_0$ and $i_1$ are homotopy equivalences for each $n$ since $\widetilde{F}^M(n)$ is homotopy equivalent to the mapping cylinder of the inclusion $C_n(M)\subset C_n\CPTM$ which is also a homotopy equivalence, so $(i_0)^\vee$ and $(i_1)^\vee$ are $\pi_*$-isomorphisms. $\Phi_n$ is a $\pi_*$-isomorphism since it reduces the  equivalence of the original Atiyah duality in the (homotopy) category of classical spectra via Lemma \ref{Tsemistable} (see \cite{browder}). 
\end{proof}

\section{Spectral sequences}
\subsection{A chain functor}\label{SSchain}
\begin{defi}\label{Dchainspectra}
\begin{itemize}
\item For a chain complex $C_*$, $C[k]_*$ is the chain complex given by $C[k]_l=C_{k+l}$ with the same differential as $C_*$ (without extra sign.). 
\item Fix a fundamental cycle $w_{S^1}\in C_1(S^1)$. We shall define a chain complex $C_*^S(X)$ for a symmetric spectrum $X$. Define a chain map 
$i_k^X:\bar C_*(X_k)[k]\to \bar C_*(X_{k+1})[k+1]$ by $i_k^X(x)=(-1)^l\sigma_*(w_{S^1}\times x)$ for $x\in \bar C_l(X_k)$, where $\sigma:S^1\wedge X_k\to X_{k+1}$ is the structure map of $X$. We define $C_*^S(X)$ as the colimit of the sequence $\{\bar C_*(X_k)[k]; i_k^X\}_{k\geq 0}$. Clearly the procedure $X\mapsto C_*^S(X)$ is extended to a functor $\SP\to \CH_{\kk}$ in an obvious manner. 
\item We denote by $H^S_*(X)$ the homology group of $C^S_*(X)$.
\item Let $\FCW$ denote the full subcategory of $\CG$ spanned by finite CW complexes.  We define a functor $C^*_S:(\FCW)^{op}\to \CH_{\kk}$ by $C^q_S(X)=C_{-q}^S(X^\vee)$.
\end{itemize}
\end{defi}

\begin{lem}\label{Lchaininvariance}
\begin{enumerate}
\item If $f:X\to Y$ is a stable equivalence between strongly semistable  spectra, the induced map $f_*:C_*^S(X)\to C_*^S(Y)$ is a quasi-isormorphism. 
\item If $X\to Y\to Z$ is a homotopy cofiber sequence of strongly semistable spectra, $C_*^S(X)\to C_*^S(Y)\to C_*^S(Z)$ is a homotopy cofiber sequence in $\CH_{\kk}$.
\end{enumerate}
\end{lem}
\begin{proof}
The first part clearly follows from the Hurewicz Theorem and the fact that stable equivalences between strongly semistable spectra are $\pi_*$-isomorphisms. To prove the second part, we take the  following replacements, using the factorization of the model category $\SP$.  
\[
\xymatrix{
\emptyset \ar[r] & X' \ar[r]\ar[d] & Y' \ar[d] \\
& X\ar[r]^f & Y
}
\]
Here, the top horizontal arrows are cofibrations and the vertical arrows are trivial fibrations. Since trivial fibrations in $\SP$ are level trivial fibrations, $X'$ and $Y'$ are strongly semistable. This implies the cofiber $Y'/X'$ is strongly semistable and the canonical map $C_*^S(Y'/X')\to C_*(Z)$ is a quasi-isomorphism by the first part. We have completed the proof of the second part. 
\end{proof}
 
\begin{lem}\label{Lcochaintheory}
There exists a zigzag of natural transformations between $C^*$ and $C^*_S:(\FCW)^{op}\to \CH_{\kk}$, in which each natural transformation is an objectwise quasi-isomorphism.  
\end{lem}

\begin{proof}
All we have to do is to verify $C^*_S$ satisfies the axioms of {\em cochain theory} in \cite{mandell}. (In \cite{mandell}, it was proved that any functor satisfying the axioms can be connected with  the singular cochain by a zigzag as in the claim.) The only non-trivial axiom is the {\em Extension/Excision axiom} which states if $X\in \FCW$ and $A\subset X$ is an inclusion of a subcomplex, the map from the homotopy fiber $F(T(X/A)\to T(*))$ to the  homotopy fiber $F(T(X)\to T(A))$ is a quasi-isomorphism (see \cite{mandell} for the notations).
We shall verify this axiom for $T=C^*_S$. For a pointed space $Y$, Let $Y^\vee_*$ denote the  spectrum given by $(Y^\vee_*)_k=\Map_*(Y,S^k)$ where $\Map_*$ is the space of based maps. The fiber sequence 
\[
(X/A)_*^\vee \longrightarrow (X_+)^\vee \longrightarrow (A_+)^\vee
\]
is a homotopy fiber sequence in $\SP$ because  $A\subset X$ is an inclusion of finite CW-complexes and $\Sphere$ is strongly semistable, and also a homotopy cofiber sequuence because the two  sequences are the same in $\SP$.  By (2) of Lemma \ref{Lchaininvariance}, the induced sequence
\[
C^S_{*}((X/A)^\vee_*)\longrightarrow C^*_S(X) \longrightarrow C^*_S(A)
\]  
is a  homotopy cofiber sequence and also  a fiber one since the two  sequences are the same in $\CH$.  By a similar argument, we see the sequence 
\[
C^S_{*}((X/A)^\vee_*)\longrightarrow C^*_S(X/A) \longrightarrow C^*_S(*)
\]
is a homotopy fiber sequnce. Thus, we have verified the axiom.

\end{proof}
The functor $C^S_*$ does not have any compatibility with symmetry isomorphisms of the monoidal products $\wedge$ in $\SP$, but it have some compatibility with the tensor $\hotimes$ with a space. 
\begin{lem}\label{Lchaincompatible}

\begin{enumerate}
\item For $U\in \CG$ and $X\in \SP$, the collection of Eilenberg-Zilber shuffle map $\{EZ:C_*(U)\otimes \bar C_*(X_k)[k]\to \bar C_*((U_+)\wedge  X_k)[k]\}_k$ induces a quasi-isomorphism 
\[
C_*(U)\otimes C_*^S(X)\to C_*^S(U\hotimes X).\]
\item Let $\oper$ be a topological operad and $Y$ be a  $\oper$-comodule in $\SP$. A natural  structure of a chain $C_*(\oper)$-comodule on the collection $C^S_*Y=\{C^S_*(Y(n))\}_{n}$  is defined as follows. The partial composition is given by the composition
\[
C_*(\oper(m))\otimes C^S_*(Y(m+n-1))\to  C^S_*(\oper(m) \hotimes Y(m+n-1)) \to C^S_*(Y(n)),
\]
where the left map is the one defined in the first part  and the right map is induced by the partial composition on $Y$. The action of $\Sigma_n$ on $C_*^S(Y)(n)$ is the natural induced one.  
\end{enumerate}
\end{lem}

\begin{proof}
The cross product $w_{S^1}\times x$ is equal to $ EZ(w_{S^1}\otimes x)$ by definition, and the the shuffle maps are associative and compatible with symmetry isomorphisms of monoidal products without any chain homotopy for normalized singular chain, so the maps $EZ$ are compatible with the maps $i_k^X$ in Definition \ref{Dchainspectra} (the sign  commuting an element of $C_*(U)$ and $w_{S^1}$ is canceled with the sign attached in the definition of $i^X_k$). This imply the first part. The second part follows from commutativity of the following diagram which is clear from the property of the shuffle map mentioned above. 
\[
\xymatrix{
C_*(U)\otimes C_*(V)\otimes C_*^S(X)\ar[r]\ar[d] &  C_*(U)\otimes C_*^S(V\hotimes X)\ar[d] \\
C_*(U\times V)\otimes C_*^S(X)\ar[r] & C_*^S((U\times V)\hotimes X),
}
\]
where $U,V\in\CG$ and $X\in\SP$ and the left vertical arrow is induced by the EZ shuffle map and other arrows are  given by the first part. 

\end{proof}
\subsection{Construction of $\check{\text{C}}$ech spectral sequence}\label{SSconstructionss}
\begin{defi}\label{Dcheckmodule}
We shall define a {\em $C_*(\DD_1)$-comodule $\CECH_{\star\,*}^M$ of double complexes}, which is  consisting of the following data. 
\begin{itemize}
\item A  sequence of double complexes $\{\CECH_{\star\,*}^M(n) \}_{n\geq 1}$ with two differentials  $d$ and $\partial$ of degree $(0,1)$ and $(1,0)$ respectively, 
\item an action of $\Sigma_n$ on $\CECH_{\star\,*}^M(n)$ which preserves the bigrading, and \vspace{1mm}
\item a partial composition $(-\circ_i-):C_k(\DD_1(m))\otimes \CECH_{\star \,*}^M(m+n-1)\to \CECH_{\star,\,*+k}^M(n)$ 
\end{itemize} 
which satisfy the following compatibility conditions in addition to the conditions in Definition \ref{Dcontra-module}.
\[
d\partial =\partial d,\quad d(\alpha\circ_ix)=d\alpha \circ_ix+(-1)^{|\alpha|}\alpha \circ_i dx,\quad \partial(\alpha\circ_i x)=\alpha\circ_i \partial x \, .
\]
We define the double complex   $\CECH_{\star\,*}^M(n) $  by
\[
\CECH_{p\,*}^M(n)=
\underset{
G\in \GG(n,p)}{\bigoplus} C^S_*(\DeltaT _G) 
\]
for $p\geq 0$, and $\CECH_{p,*}^M(n)=0$ for $p<0$, where $\GG(n,p)\subset \GG(n)$ is the set of graphs with $p$ edges (see Definition \ref{Dmap} for $\DeltaT_G$). The differential  $d$ is the original differential of $C^S_*(\DeltaT_G)$.  
The other differential $\partial$ is given by the signed sum 
\[
\partial=\sum_{t=1}^p(-1)^{t+1}\partial_t
\] where $\partial_t$ is induced by the inclusion $\DeltaT_G\to \DeltaT_{G_t}$ where the graph $G_t$  is defined by removing the $t$-th edge from $G$ (in the lexicographical order).   The action of $\sigma$ on $\TCCM(n)$ restricts to a map $\sigma : \DeltaT_G\to \DeltaT_{\sigma^{-1}(G)}$ (see subsection \ref{SSnt} for $\sigma^{-1}(G)$).  This map induces a chain map $\sigma_*:C_*^S(\DeltaT_G)\to C_*^S(\DeltaT_{\sigma^{-1}(G)})$ by the standard pushforward of chains.  For $G\in G(n,p)$, let $\sigma_G\in \Sigma_p$ denote the following composition 
\[
\underline{p}\cong E(\sigma^{-1}(G))\to E(G)\cong \underline{p},
\]
where $\cong$ denotes the order preserving bijection and the middle map is given by $(i,j)\mapsto (\sigma(i),\sigma(j))$.  We define the action of $\sigma$ on $\CECH^M(n)$ as $ \mathrm{sgn}(\sigma_G) \cdot \sigma_*$ on each summand. We shall define the partial composition.
 Let $f_i:\underline{m+n-1}\to \underline{n}$ be the order preserving serjection which satisfy $f(i+t)=f(i)$ for $t=1,\dots, m-1$. For  elements $\alpha\in C_*(\DD_1(m))$ and $x\in C_*^S(\DeltaT_G)$ with $G\in\GG(n+m-1)$, if $\# E(f_i(G))=\# E(G)$,  the partial composition $\alpha\circ_ix\in C_*^S(\DeltaT_{f_iG})$ is defined similarly to Lemma \ref{Lchaincompatible} with the map $(-\circ_i-):\DD_1(m)\hotimes \DeltaT_G\to \DeltaT_{f_iG}$, and  if $\#E(f_i(G))<\#E(G)$,  $\alpha\circ_ix$ is zero. This partial composition is well-defined by Lemma \ref{Ldiagonalinclusion}. We have completed the definition of $\CECH^M$. The compatibility between $d,\ \partial,$ and $(-\circ_i-)$ is obvious. \\
\indent Let $\tot \CECH^M_{\star\, *}(n)$ denote the total complex. Its differential is given by $d+(-1)^q\partial$ on $\CECH^M_{\star\, q}(n)$. We regard the sequence $\tot \CECH^M_{\star\, *}=\{\tot \CECH^M_{\star\, *}(n)\}_n$ as a chain $C_*(\DD_1)$-comodule with the induced structure.  We fix an operad map $f:\ASS_\infty\to C_*(\DD_1)$, and regard $\tot\CECH^M$ as a $\ASS_\infty$-comodule by pulling back the partial compositions by $f$. We consider the Hochschild complex $\Hoch_{\bullet}(\tot\CECH^M_{\star\,*})$ associated to this $\ASS_\infty$-comodule, see Definition \ref{Dhochschild}.
 The total degree of elements of  $\Hoch_\bullet(\tot\CECH_{\star\,*}^M)$ is $-*-\star-\bullet$. We define two filtrations $\{F^{-p}\}$ and $\{\bar F^{-p}\}$ on this complex as follows.
  $F^{-p}$ (resp. $\bar F^{-p}$) is generated by the homogeneous parts  whose degree satisfies $\star+\bullet\leq p$ (resp. $\bullet\leq p$). We call the spectral sequence associated to $\{F^{-p}\}$  the {\em $\check{C}$ech spectral sequence}, in short, {\em $\check{C}$ech s.s.} and denote it by $\{\BGSSS^{\,-p,\,q}_r\}_r$. The spectral sequence associated to $\{\bar F^{-p}\}$ is denoted by $\{\overline{\SINHASS}^{\,-p,\,q}_r\}_r$. 
\end{defi}
\begin{lem}\label{Lsinhass}
The spectral sequence $\overline{\SINHASS}_r$  in Definition \ref{Dcheckmodule} and Sinha spectral sequence $\SINHASS_r$ in Definition \ref{Dcosimplicial} are isomorphic after the $E_1$-page.
\end{lem}
\begin{proof}
Put $N_0=\# \{(i,j)\mid i,j\in \underline{n},\ i<j\}$.  By applying Lemma \ref{Lcheckfunctor} to the functor $X:P_\nu(N_0)=\GG(n)-\{\emptyset\}\to \SP$ given by $X_G=\DeltaT_G$, we see the map $\tot\CECH_{\star,*}^M(n)\to \C^S_*(\CCM(n))$ induced by the collapsing (quotient) map $\TCCM (n)\to \CCM(n)$ is a quasi-isomorphsim. Combining this with Theorem \ref{TAtiyahdual} and Lemma \ref{Lchaininvariance}, two comodules $C^S_*(\CSINHA^\vee)$ and $\tot\CECH_{\star,*}$ are quasi-isomorphic. Clearly, $\Hoch_\bullet C^S_*(\CSINHA)$ is quasi-isomorphic to the normalized complex of $C^S_*(\CC^\bullet\langle[M]\rangle^\vee)$, which is quasi-isomorphic to the normalized  total complex of $C^*(C^\bullet\langle[M]\rangle)$ by Lemma \ref{Lcochaintheory}. Thus,  $\Hoch_\bullet\tot \CECH_{\star,*}^M$ and  the normalized total complex of $C^*(C^\bullet\langle[M]\rangle)$ are connected by a zigzag of quasi-isomorphisms which  preserve the filtration. This zigzag induces a zigzag of morphisms of spectral sequences which are isomorphisms after the $E_1$-page because the homology of $\tot \CECH_{\star,*}(n+1)$ is isomorphic to $H^*(\CC^{n}\langle[M]\rangle)$ under the  zigzag.
\end{proof}
\subsection{Convergence}\label{SSconvergence}
In  this subsection, we assume $M$ is orientable. 
We shall prepare some notations and terminologies which is necessary to analyze the $E_1$-page of $\CECHC$ech s.s.

\begin{defi}\label{Diroiro}
\begin{itemize}

\item We fixed an embedding $e_0:\TM\to \RR^K$  and a number $\eps_0$ in Definition \ref{Dmap}. 
We also fix an isotopy $\iota_t:\TM\to\RR^{2K}$ with $\iota_0 = 0\times e_0$ and $\iota_1=\Delta_{\RR^K}\circ e_0$ where $0\times e_0:\TM\to \RR^{2K}$ is given by $(0\times e_0)(z)=(0,e_0(z))$, and $\Delta_{\RR^K}$ is the diagonal map on $\RR^K$. We impose the additional condition that  $\epsilon_0$ is smaller than 
$\min\{L(\iota_t)\mid 0\leq t\leq 1\}$. We also fix 1-parameter family of bundle map $\kappa_t:\nu_{\epsilon_0}(0\times e_0)\to \nu_{\epsilon_0}(\iota_t)$ with $\kappa_0=id$. 

\item We fix the following  classes.
\[
\begin{split}
\tw & \in H_{2\dd -1}(\TM),\qquad      \omega_\Delta\in H^{2\dd -1}(\TM\times \TM, \Delta(\TM)^c),\qquad  w_{S^K}\in H_K(S^K), \qquad 
    \omega_{S^K} \in H^K(S^K),\\
  \omega_\nu & \in  H^{K-2\dd +1}(Th(\nu)),\qquad \omega(n)\in H^{n(K-2\dd +1)}(Th(\nu^{\times n})),\qquad \gamma\in H^{\dd}(\TM\times \TM, (\TM\times_M\TM)^c)\, .
\end{split}
\]
Here, $\tw$ is a fundamental class of $\TM$, and $\omega_\Delta$ is a diagonal class 
  such that  the equality 
\[
(\tw\times \tw)\cap \omega_\Delta=\Delta_* (\tw)  \in H_{2\dd-1}(\TM^{\times 2} )
\] holds ($\Delta(\TM)^c$ is the complement of the tubular neighborhood of the (standard, non-deformed) diagonal), and  $w_{S^K}$ is the $K$ times product $(w_{S^1})^{\times n}$ of the class $w_{S^1}$ fixed in Definition \ref{Dchainspectra}, and $\omega_{S^K}$ is the class such that $w_{S^K}\cap \omega_{S^K}$ is the class represented by a point.  $\omega_\nu$ is the Thom class  satisfying the equality 
\[
\kappa_1^*(\omega_\Delta\cdot (\omega_\nu\times \omega_\nu))=\omega_{S^K}\times \omega_\nu, 
\] where  $\omega_{\Delta}\cdot (\omega_\nu\times\omega_\nu)$ is naturally regarded as  a Thom class for the bundle $\nu_{\epsilon_0}(\Delta_{\RR^K}\circ e_0)$.  We set $\omega(n)=\omega_\nu^{\times n}$. $\gamma$ is also a Thom class of a tubuler neighborhood of $\TM\times _M\TM$ in $\TM\times \TM$. 
\item We call a graph in $\GG(n)$ which does not contain a cycle (a closed path) a {\em tree}. For a graph $G\in \GG(n)$,  vertices $i$ and $j$ are said to be {\em disconnected} in $G$ if $i$ and $j$ belong to different connected components of $G$. 
\item For $i<j$, let $\pi_{ij}:\TM^{\times n}\to \TM^{\times 2}$ be the projection given by $\pi_{ij}(z_1,\dots, z_n)=(z_i,z_j)$.  Set $\Delta_{ij}=\Delta_G$ for $E(G)=\{(i,j)\}$, and 
\[
\gamma_{ij}=\pi_{ij}^*(\gamma)\in H^{\dd}(\TM^{n},(\Delta_{ij})^c ).
\]
For a tree $G\in \GG(n)$, write  $E(G)$ as $\{\, (i_1,j_1)<\cdots <(i_r,j_r)\,\}$ with $i_t<j_t$ for $t=1,\dots r$. We put 
\[
w_G=\tw ^{\,\times  n}\cap \gamma_{i_1,\, j_1}\cdots\gamma_{i_r,\, j_r}\in H_{n(2\dd -1)-r\dd }(\Delta_G).
\] Clearly, $w_G$ is a fundamental class of $\Delta_G$.
\item Let $G\in \GG(n,r)$ be a tree. Suppose $i$ and $i+1$ are disconnected in $G$. Let $d_i:\underline{n}\to \underline{n-1}$ be the map given by 
\[
d_i(j)=\left\{\begin{array}{ll}
j & (j\leq i\ )\\
j-1 &  (j\geq i+1)
\end{array}\right.
\] and 
set $H=d_i(G)\in \GG(n-1)$.  We   define   maps
\begin{align*}
\phi_G: & \bar H_{*}(Th(\nu_G))\to H_{*-nK}^S(\DeltaT_G),\quad &  \zeta_G: &H_*^S(\DeltaT_G)\to H^{-*-dr}(\Delta_G), \\
 \mu_i :& H^S_*(\DeltaT_G)\to H^S_*(\DeltaT_H),\quad & m_i:&H^{*}(\Delta_G)\to H^{*}(\Delta_H).
\end{align*}
  $\phi_G$ is the composition 
\[
\bar H_*(Th(\nu_G))\stackrel{(\lambda_G)_*}{\rightarrow} \bar H_*(\DeltaT_G(\configc_0 )_{nK})\to H_{*-nK}^S(\DeltaT_G(\configc_0) )
\to H_{*-nK}^S(\DeltaT_G) \, ,
\] where $\lambda_G$ is the map defined in Definition \ref{Dmap}, and  the second map is the canonical one and the third is the inverse of evaluation at $\configc_0$. Clearly $\phi_G$ is an isomorphism. $\zeta_G$ is the composition $(w_G\cap-)^{-1}\circ (-\cap \omega(n))\circ \phi_G^{-1}$ consisting of
 \[
\xymatrix{H^S_*(\DeltaT_G)\ar[r]^{\phi_G^{-1}}  &  \bar H_{*+nK}(Th(\nu_G))\ar[r]^{-\cap \omega(n)}   & H_{*+n(2\dd -1)}(\Delta_G )\ar[r]^{(w_G\cap-)^{-1}} &H^{-*-\dd r}(\Delta_G ).  }
\]
$\mu_i$ is the map induced by the partial composition $\mu\circ_i-$ where $\mu\in H_0(\DD_1(2))=\ASS (2)$  is the  fixed generator.
 $m_i$ is given by $(-1)^A\Delta_i^*$ where $A=*+\dd r+n$ with $r=\# E(G)$, and $\Delta_i^*$ denotes the pullback by the restriction to $\Delta_H$ of the  diagonal 
\[
\Delta_i:\TM^{\times n-1}\to \TM^{\times n}, \qquad (z_1,\dots, z_{n-1})\mapsto (z_1,\dots, z_i,z_i,\dots, z_{n-1}).
\] 
\item We denote by $H^S\CECH^M_{\star\, *}(n)$  the bigraded chain complex  obtained taking homology  of $\CECH^M_{\star\, *}(n)$ for the differential $d$, see Definition \ref{Dcheckmodule}. Its differential is induced by the  differential $(-1)^q\partial$ on $\CECH^M_{\star\, q}(n)$.  We regard the collection $H^S\CECH^M=\{H^S\CECH^M(n)\}$ as an $\ASS$-comodule with the structure induced by $\CECH^M$.
As a $\kk$-module, $H^S\CECH^M(n)$ is the direct sum $\bigoplus_{G\in \GG(n)}H^S_*(\DeltaT_G)$. We denote by $aG$ the element of $H^S\CECH^M(n)$ corresponding to  $a\in H_{*}^S(\DeltaT_G)$. 
\item The homology of the Hochschild complex  $\Hoch_\bullet(H^S\CECH^M_{\star\,*})$ has the bidegree  $(-\bullet-\star,-*)$. We denote the homogeneous part of bidegree $(p,q)$ by $H_{-p,-q}(\Hoch (H^S\CECH^M))$.
\item For two graphs $G, H\in \GG(n)$ with $E(G)\cap E(H)=\emptyset$,  the product $GH\in \GG(n)$ denotes the graph with $E(GH)=E(G)\cup E(H)$.   Let $i,j,k\in \underline{n}$ be distinct vertices, and  $[ijk]\in \GG(n)$ denote the graph with $E([ijk])=\{(i,j),\ (j,k)\}$ For a graph $G\in\GG(n)$, the products $G[ijk]$ and $G[jki]$ and $G[kij]$ have the same connected component (if they are defined) so we have $\nu_{G[ijk]}=\nu_{G[jki]}=\nu_{G[kij]}$. Using these  equalities, and the isomorphisms $\phi_{G'}$ for $G'=G[ijk], G[jki],$ and  $G[kij]$, we identify the three groups $H^S_*(\DeltaT_{GH[ijk]})$, $H^S_*(\DeltaT_{G[jki]})$, and $H^S_*(\DeltaT_{G[kij]})$ with one another.  Under this identification, let $I(n)\subset H^S\CECH^M(n)$ be the sub-module generated by 
\begin{itemize}
\item the summands of graphs which are not trees, and 
\item the elements of the form $aG[jki]+(-1)^s aG[ijk]+(-1)^{s+t}aG[kij] $ for $(i,j),(j,k),(i,k)\not\in E(G)$, where $a\in H^S_*(\DeltaT_{G[ijk]})$ and $s+1$ is the number of edges of $G$ between $(i,j)$ and $(i,k)$, and $t+1$ is the one between $(i,k)$ and $(j,k)$,  
\end{itemize}
\item  We say a graph $G\in \GG(n) $ with a  edge set $E(G)=\{(i_1,j_1)<\dots <(i_r,j_r)\}$ is {\em distinguished} if  the following inequalities hold. 
\[
i_1<j_1,\dots, i_r<j_r,\quad i_1\leq \cdots \leq  i_r.
\] We denote by $\GG(n)^{dis}\subset \GG(n)$ the subset of the distinguished graphs.

\end{itemize}
\end{defi}

The following lemma is obvious by the definition of $\CECHC$ech s.s.
\begin{lem}\label{LcechE2}
With the notations in  Definition \ref{Diroiro}, the $E_2$-page of $\check{C}$ech s.s. is isomorphic to the  homology of the Hochschild complex of $H^S\CECH^M_{\star\, *}$. More precisely, there exists an isomorphism of $\kk$-modules 
\[
\BGSSS_2^{p\, q}\cong H_{-p,-q}(\Hoch (H^S\CECH^M ))\quad \text{for each } (p,q). 
\]\hfill \qedsymbol

\end{lem}
\begin{lem}\label{Lcechreduction}
With the notations in Definition \ref{Diroiro}, $I(n)$ is acyclic, i.e., $H_\partial (I(n))=0$,  and  the sequence $\{I(n)\}_n$ is closed under the partial compositions and symmetric group actions.
\end{lem}
\begin{proof}
 Since $\GG(n)^{dis}$ is stable under removing edges, the submodule $\bigoplus_{G\in \GG(n)^{dis}}H^S_*(\DeltaT_G)$ of $H^S\CECH^M(n)$ is a subcomplex. By the argument similar to (the dual of) \cite{FT}, we see the inclusion $\CECH(\GG(n)^{dis}):=\bigoplus_{G\in \GG(n)^{dis}}H^S_*(\DeltaT_G)\subset H^S\CECH^M(n)$ is a quasi-isomorphism. We easily see the map $\CECH(\GG(n)^{dis})\to \CECH(n)/I(n)$ induced by the inclusion is an isomorphism (see the proof of Lemma \ref{Lgraphincl}).  
\end{proof}

\begin{lem}\label{Lisotopythom}
Let $\bar e_t :\TM\to \RR^{2K}$ be an isotopy with $\bar e_0=0\times e_0$ and $\bar e_1=e_0\times 0$ and $F_t:\nu_{\eps_0}(\bar e_0)\to \nu_{\eps_0}(\bar e_t)$ be an isotopy which is also a bundle map covering $\bar e_t$, and satisfy $F_0=id$. Then we have the equality
\[
(F_1)^*(\omega_\nu \times \omega_{S^K})=(-1)^K\omega_{S^K}\times \omega_\nu\,  .
\]
Here $\omega_\nu \times \omega_{S^K}$ is considered as a class of $H^{2K-2\dd +1}(Th(\nu_{\eps_0}(\bar e_1)))$ via the  map collapsing the subset $ \nu_{\eps_0}(e)\times \RR^K-\nu_{\eps_0}(\bar e_1)$ and  $\omega_{S^K}\times \omega_\nu$ is similarly understood.
\end{lem}
\begin{proof}
Since the only problem is the orientation, it is enough to see a variation of a basis via a local model. Let $e_0:\RR^{2\dd-1}\to \RR^K$ be the inclusion to the subspace of elements with the last $K-2\dd+1$ coordinates being zero. A covering isotopy is given by $F_t(u,v)=((1-t)u-tv,\ tu+(1-t)v)$ for $u,v\in \RR^K$. Since $F_1(u,v)=(-v,u)$, the derivative $(F_1)_*$ maps a basis $\{\mathbf{ a},\mathbf{b}\}$ of the tangent space of $\RR^{2K}$ to $\{\mathbf{ b},-\mathbf{a}\}$, where $\mathbf{a}$ and  $\mathbf{b}$ denote  basis of $T\RR^K\times 0$ and $0\times T\RR^K$ respectively. This implies $(F_1)^*(\omega_\nu\times \omega_{S^K})=(-1)^K(-1)^{K(K-2\dd+1)}\omega_{S^K}\times \omega_\nu=(-1)^K\omega_{S^K}\times \omega_\nu$. 
\end{proof}
\begin{lem}\label{Lcomposition}
We use the notations in Definition \ref{Diroiro}. Let $G\in \GG(n,r)$ be a tree with $i$ and $i+1$ are disconnected in $G$. 
and set $H=d_i(G)\in \GG(n-1)$. Then, the following  diagram is commutative.
\[
\xymatrix{ H^S_*(\DeltaT_G) \ar[r]^{\mu_i}\ar[d]^{\zeta_G}  & H^S_*(\DeltaT_H)\ar[d]^{\varepsilon_1\zeta_H} \\
 H^{-*-\dd r}(\Delta_G) \ar[r]^{m_i}   & H^{-*-\dd r}(\Delta_H), }
\]
where  $\varepsilon_1=(-1)^B$ with  $B= K(*+1+(K-1)/2)$.

\end{lem}
\begin{proof}
The claim follows from the commutativity of the following diagram $(*)$.
\[
\xymatrix{H^S_*(\DeltaT _G)\ar[r]^{\mu_i}&     H^S_*(\DeltaT _{H})    &   \\
\bar H_{*+nK}(Th(\nu_G))\ar[r]^{\mu'}\ar[u]^{\phi_G}\ar[d]^{\omega(n)}&        \bar H_{*+nK}(Th(\nu'))\ar[u]^{\phi'_H}\ar[d]^{ \omega'}
   &   \bar H_{*+(n-1)K}(Th(\nu_H))\ar[ul]^{\phi_H}\ar[l]^{\alpha}\ar[dl]^{ \varepsilon_1\omega (n-1)}   \\
H_{*+n(2\dd-1)}(\Delta_G)\ar[r]^{\mu''}&H_{*+(n-1)(2\dd-1)}(\Delta_{H}) &    \\
H^{-*-dr}(\Delta_G)\ar[u]^{w_G}\ar[r]^{m_i} & H^{-*-dr}(\Delta_{H})\ar[u]^{w_{H}} & \\
}
\]
Here,  
\begin{itemize}
\item  $\nu'$ is the disk bundles over $\Delta_H$ of fiber dimension $nK-(n-1)(2\dd-1)$ defined by 
\[
\nu'=\nu_{\eps_0}(e_0^{\times n}\circ \Delta_i)|_{\Delta_H},
\]
where  the restriction is taken as a disk bundle over $\TM^{\times n-1}$ and see Definition \ref{Diroiro} for $\Delta_i$. 
\item $\omega'\in \bar H^{nK-(n-1)(2\dd -1)}(Th(\nu'))$ is given by 
\[
\omega'=(-1)^C (\omega_\nu)^{\times i-1}\times \omega_\Delta\cdot (\omega_\nu\times\omega_\nu)\times (\omega_\nu)^{\times n-i-1} \ \text{with}\   C=(n+i+1)K.
\]
\item $\phi'_H$ is defined using the following map $\lambda_H'$ similarly to $\phi_H$.
\[
\lambda'_H:\nu^{\times n}\ni u\longmapsto \bigl(\,e_0^{\times n}\circ \Delta_i ,\ \eps_0\   , u\,\bigr)\in \DeltaT(\configc_0)_{nK}\,.
\]

\item $\mu'$ is the map collapsing the subset $\nu_G-\nu'$, where $\nu'$ and $\nu_G$ are regarded as subsets in $\RR^{nK}$.
 \item $\mu''$ is the  composition
\[
H_*(\Delta_G)\to H_*(\Delta_G, \Delta_i (\Delta_{H})^c)\to H_{*-2\dd+1}(\Delta_i(\Delta_{H}))\cong H_{*-2\dd+1}(\Delta_{H}).
\]
Here the first map is the standard quotient map, and the third  is the inverse of the diagonal, and the second  is the cap product with the class
\[
(-1)^{i+1+n}1\times \cdots \times \omega_{\Delta}\times \cdots \times 1 \qquad (\omega_\Delta\ \text{ in the $i$-th factor }).
\]
\item  $\alpha$ is the  composition $(1\times \kappa_1\times 1)_*\circ T\circ (\varepsilon_2 w_{S^K}\times -)$ of the maps
\[
\begin{split}
\bar H_{*'}(Th(\nu_{H})) &   \stackrel{\varepsilon_2 w_{S^K}\times -}{\longrightarrow}  \bar H_{*'+K}(S^K\wedge Th(\nu_{H})) \stackrel{ T }{\longrightarrow} \bar H_{*'+K}(Th(\nu ''))\\ 
&\stackrel{(1\times \kappa_1\times 1)_*}{\longrightarrow} \bar H_{*'+K}(Th(\nu' )), 
\end{split}
\]
where $\nu''$ is the disk bundle over $\Delta_H$ of the same fiber dimension as $\nu'$ given by
\[
 \nu''=\nu_{\eps_0}(e'')|_{\Delta_H},\ \text{with}\  e''=e_0^{\times i-1}\times (0\times e)\times e^{\times n-i}_0:\TM^{\times n-1}\to \RR^{nK}, 
\]
and    
\[
\varepsilon_2=(-1)^{D}, \quad D= {K(*'+(K-1)/2+i+1)}, 
\]  
 and $T$ is the composition of the transposition of $S^K$ from the first to the $i$-th component with the map induced by the  map  collapsing 
the subset $(\nu ^{\times i-1}\times \RR^K \times \nu ^{\times n-i+1})|_{\Delta_H}-\nu''$, and $1\times \kappa_1\times 1$ is induced by the restriction of  the product map 
\[
1\times \kappa_1\times 1:\RR^{(i-1)K}\times \nu_{\eps_0}(0\times e_0)\times \RR^{(n-i-1)K}\to \RR^{(i-1)K}\times \nu_{\eps_0}(\Delta_{\RR^K}\circ e_0)\times \RR^{(n-i-1)K}
\] with $\kappa_1$ in $i$-th component. 
\item The arrows to which a (co)homology class assigned denote the cap products with the class.
\end{itemize}
\indent Our sign rules for graded products are the usual graded commutativity except for the compatibility of cross and cap products for which we use the following  rule
\[
(a\times b)\cap (x \times y )=(-1)^{(\,|a|-|x|\, )|y|}(a\cap x)\times (b\cap y).
\]
These are the rules based on the definitions in \cite{hatcher}.
With these  rules, the commutativity of the squares in the  diagram $(*)$ is  clear since the map $\Delta'$ defined in subsection \ref{SSTCCM} is isotopic to the usual diagonal. We shall prove  commutativity of the two triangles in the same diagram. 
The commutativity of the upper triangle follows from the commutativity of the following diagram.
\[
\xymatrix{\bar H_l(Th(\nu )_{H})\ar[r]^{T\circ (\varepsilon_2 w_{S^K}\times-) }\ar[d]^{(\lambda_G)_*}  &  \bar H_{l+K}(Th(\nu'') )\ar[r]^{(1\times \kappa_1\times 1)_*}\ar[rd]^{(\lambda''_H)_*} & \bar H_{l+K}(Th(\nu') ) \ar[d]^{(\lambda'_H)_*}
\\ 
\bar H_l(\DeltaT_{H}(\configc_0)_{(n-1)K})\ar[rr]^{\varepsilon_2 w_{S^K}\times}   && H_{l+K}(\DeltaT_{H}(\configc_0)_{nK}).}
\] 
Here,  $\lambda''_H$ is given by 
$u \longmapsto   \bigl(\,e_0^{\times i-1}\times (0\times e_0)\times e_0^{\times n-i},\, \epsilon_0,\, u\, \bigr)$. Commutativity of the left trapezoid follows from Lemma \ref{Lisotopythom} (the sign $\varepsilon_2$ is the product of the sign in $i_k^X$ in Definition \ref{Dchainspectra} and the sign in Lemma \ref{Lisotopythom}), and that of the right triangle follows from homotopy between $\lambda'_H\circ \kappa_1$ and $\lambda''_H$ constructed from the isotopy $\kappa_t$ in Definition \ref{Dcheckmodule}.   
We shall show the lower triangle is commutative. We see
\[
\begin{split}
\varepsilon_1 \alpha (x)\cap \omega'&=\{(\kappa_1)_*T_*(w_{S^K}\times x)\} \cap (\omega \times \cdots  \omega_\Delta(\omega\times \omega)\cdots \times \omega) \\
                      &= \{ (\kappa_1)_*T_*(w_{S^K}\times x) \}\cap (\omega \times \cdots  (\kappa_1^{-1})^*(\omega_{S^K}\times \omega ) \cdots \times \omega) \\
&=(\kappa_1)_*\{T_*(w_{S^K}\times x)\cap (\omega \times \cdots  (\omega_{S^K}\times \omega ) \cdots \times \omega)\ \} \\
&=( \kappa_1)_* T_* \{ (w_{S^K}\times x)\cap T^*(\omega \times \cdots  (\omega_{S^K}\times \omega ) \cdots \times \omega) \} \\
&=(\kappa_1)_*  T_*\{(w_{S^K}\times x)\cap \omega_{S^K}\times\omega\times\cdots\times \omega) \}\\
&= (w_{S^K}\times x)\cap (\omega_{S^K}\times\omega \times\cdots\times \omega) \\
&=x\cap \omega(n-1) .
\end{split}
\]
Here, $(\kappa_1)_*$ is  abbreviation of $(1\times \kappa_1\times 1)_*$ and  $\omega$  of  $\omega_\nu$.  All capped classes are considered as  elements of the homology of the base space $\Delta_{H}$ of involved disk bundles by projections. The second equality follows from the definition of $\omega_\nu$. As an endomorphism on the base space, $T_*$  and $(1\times \kappa_1\times 1)_*$ are the identity hence the sixth equality holds. 
\end{proof}
The following  lemma is easily verified and a proof is omitted.
\begin{lem}\label{Lcechdiff}
Let $G\in \GG(n,r)$ be a tree and $K\in \GG(n,r-1)$ be the tree made by removing the $t$-th edge $(i,j)$ from $G$. Under the notations in Definition \ref{Diroiro}, the following diagram is commutative.
\[
\xymatrix{
H^S_*(\DeltaT_G)\ar[d]^{\zeta_G}\ar[r]& H^S_*(\DeltaT_{K})\ar[d]^{\zeta_{K}} \\
H^{-*-\dd r}(\Delta_G)\ar[r]&H^{-*-\dd (r-1)}(\Delta_{K})}
 \]
{ where the top horizontal arrow is induced by the inclusion and  the bottom one is given by $(-1)^{(r-t)\dd}\Delta_{ij}^{!}$ with $\Delta_{ij}^{!}(x)=\gamma_{ij}\cdot x$.} \hspace{\fill} \qedsymbol
\end{lem}
\begin{defi}\label{Dcdga}
\begin{itemize}

\item In the following, we deal with modules $X$ which is a direct sum of submodules labeled by graphs in $\GG(n)$ we denote by $X^{tr}\subset X$ the modules consisting of the summands labeled by trees in $\GG(n)$.

\item We define a  $\ASS$-comodule  $A^{\star\, *}_M$ of CDBA (see Definition \ref{Dcontra-module}).
Put $H^*_G=H^*(\Delta_G)$. Let $\wedge(g_{ij})$ be the free bigraded commutative algebra generated by 
elements $g_{ij}\ (1\leq i<j\leq n)$ with bidegree $(-1,\dd)$.  For notational convenience, we set $g_{ij}=(-1)^{\dd}g_{ji}$ for $i>j$ and $g_{ii}=0$.  For $G\in \GG(n)$ with  $E(G)=\{(i_1,j_1)<\cdots <(i_r,j_r)\}$, we set 
$g_G=g_{i_1,j_1}\cdots g_{i_r,j_r}$. Put 
\[
\tilde A_M^{\star\, *}(n)=\underset{G\in \GG(n)}{\bigoplus} H^*_G\, g_G. 
\]
Here, $H^*_Gg_G$ is a copy of $H^*_G$ with degree shift. For $G\in \GG(n,r)$ and $a\in H^l_G$, the bidegree of the element $ag_G\in \tilde A_M(n)$ is $(-r,l+\dd r)$. 
We give a graded commutative multiplication on $\tilde A_M(n)$ as follows.
For $a\in H^l_G$, $b\in H^m_H$, we set 
\[
(ag_G)\cdot (bg_H)=\left\{
\begin{array}{ll}
(-1)^{mr(\dd-1)+s}\, (a\cdot b)\, g_{GH} \ \in \ H^{l+m}_{GH}\, g_{GH} & \quad (E(G)\cap E(H)=\emptyset) \vspace{2mm}\\
0   & \quad (\text{otherwise})
\end{array}\right.
\]
 Here, we set $r=\# E(G)$, and $a$ is regarded as an element of $H^*_{GH}$ by pulling back by the map $i_G:\Delta_{GH}\to \Delta_G$ induced by the quotient map $\pi_0(G)\to \pi_0(GH)$, and similarly for $b$, and the product $a\cdot b$ is taken in $H^*_{GH}$.  $s$ is the number determined by the equality $g_G\cdot g_H=(-1)^sg_{GH}$ for the product in $\wedge(g_{ij})$. \\
 \indent 
Let $J(n)\subset \tilde A_M(n)$ be the ideal  generated by the following element
\[
a(g_{ij}g_{jk}+g_{jk}g_{ki}+g_{ki}g_{ij})g_G, \qquad  bg_K\ 
\]
where $G, K\in \GG(n), \ a \in H^*_{G[ijk]},\ b\in \in H^*_K$ are elements such that $(i,j), (j,k), (k,i)\not\in E(G)$,  and $K$ is not a tree. Here, by definition, $\Delta_G$ depends only on $\pi_0(G)$, so we have $\Delta_{G[ijk]}=\Delta_{G[jki]}=\Delta_{G[kij]}$. With these identities, we regard $a$ as an element of $H_{G[jki]}^*=H_{G[kij]}$, and the first type of the generators as an element of 
\[
H_{G[ijk]}g_{G[ijk]}\oplus H_{G[jki]}g_{G[jki]}\oplus H_{G[kij]}g_{G[kij]}.
\] 
We define an algebra $A_M^{\star\, *}(n)$ as the following quotient.
\[
A_M^{\star\, *}(n)=\tilde A_M^{\star\, *}(n)/J(n)
\]
 Since the restriction of the quotient map $\tilde A_M(n)^{tr}\to A_M(n)$ is surjective, we may define  a differential, a partial composition, and an action of $\Sigma_n$ on the sequence $A_M=\{A_M(n)\}_{n}$ through $\tilde A_M(n)^{tr}$. We define a map $\tilde \partial :\tilde A_M(n)^{tr}\to \tilde A_M(n)^{tr}$ by
\[
\tilde \partial (ag_G)=\sum_{t=1}^{r}(-1)^{(l+t-1)(\dd -1)}\Delta^!_{i_t,j_t}(a)g_{i_1,j_1}\cdots \check g_{i_t,j_t}\cdots g_{i_r,j_r} \qquad (G\in \GG(n),\  a\in H^l_G)
\]
where $\Delta^!_{ij}(a)= \gamma_{ij}\cdot a$ and   $\check g_{ij}$ means removing $g_{ij}$.  It is easy to see 
$\tilde \partial(\tilde A_M(n)^{tr}\cap J(n))\subset \tilde A_M(n)^{tr}\cap J(n)$. We define the differential $\partial$ on $A_M(n)$ to be the map induced by $\tilde \partial$.
For the generator $\mu \in\ASS(2)$ fixed in Definition \ref{Diroiro}, and an element  $ag_G\in \tilde A_M(n)^{tr}$,  we define the partial composition $\mu\circ_i(ag_G)$ by 
\[
\mu\circ_i(ag_G)=\left\{
\begin{array}{ll}
\Delta^*_i(a)g_H & \quad \text{if  $i$ and $i+1$ are disconnected in $G$,} \\
0& \quad \text{otherwise},
\end{array}\right.
\]
where $H=d_i(G)$   (see Definition \ref{Diroiro}). The action of $\sigma\in \Sigma_n$ on $\tilde A_M(n)^{tr}$ is given by 
$(ag_G)^\sigma=a^\sigma(g_G)^\sigma$ where $a^\sigma$ is  the pullback of $a$ by $(\sigma_G)^{-1}$ (see Definition \ref{Dcheckmodule}) and $(g_G)^\sigma$ denotes $g_{\tau(i_1)\tau(j_1)}\cdots g_{\tau(i_r)\tau(j_r)}$ with $\tau=\sigma^{-1}$.  The partial composition and the action of $\Sigma_n$ on $\{\tilde A_M(n)^{tr}\}_n$ are easily seen to preserve the submodule  $\{J(n)\cap \tilde A_M(n)^{tr}\}_n$ and induces a strucutre of a $\ASS$-comodule on $A_M$.
\item Let $s_i:\underline{n}\to \underline{n+1}$ denote the order preserving monomorphism skipping $i+1$ for $1\leq i\leq n$.  $s_i$ naturally induces the monomorphism $s_i:\pi_0(G)\to \pi_0(s_iG)$ (see subsection \ref{SSnt}), which in turn, induces $(s_i)^*:\Delta_{s_iG}\to \Delta_G$. Let $s_i$ also denote the induced map $(s_i^*)^*:H^*(\Delta_G)\to H^*(\Delta_{s_iG})$. By further abuse of notations, we also denote by $s_i$ the map $A_M(n)\to A_M(n+1)$ given by $s_i(ag_G)=s_i(a)g_{s_iG}$.
\item Define a simplicial CDBA $A^{\star\, *}_\bullet(M)$ (a functor from $(\BDelta)^{op}$ to the category of CDBA's) as follows. We set 
\[
A^{\star\, *}_n(M)=A^{\star\, *}_M(n+1). 
\]
If we consider an element of $A_M(n+1)$ as an element of $A_n(M)$, we relabel its subscripts with $0,1,\dots, n$ instead of $1,2,\dots, n+1$. For example, $g_{01}\in A_n(M)$ corresponds to $g_{12}\in A_M(n+1)$.   The partial compositions and the maps $s_i$ (defined in the previous item) are also considered as beginning  with $(-\circ_0-)$  and $s_0$ ( originally written as $(-\circ_1-)$ and $s_1$). The face map $d_i:A_n(M)\to A_{n-1}(M)$ ($0\leq i\leq n$) is given by $d_i=\mu\circ_i(-)$ ($i<n$) and  $d_n=\mu\circ_0(-)^\sigma$ where $\sigma=(n,0,1,\dots, n-1)$. The degeneracy map $s_i:A_n(M)\to A_{n+1}(M)$ ($0\leq i\leq n$) is the map defined in the previous item.
\end{itemize}
\end{defi}

\begin{lem}\label{Ldiagonal}
Let $i,j,k$ be numbers with $i<j<k$. The equalities $\gamma_{ij}\gamma_{ik}=\gamma_{ij}\gamma_{jk}=\gamma_{ik}\gamma_{jk}$ hold.
\end{lem}
\begin{proof}
The three classes are Thom class in $H^*(\TM^{\times n},\Delta_{[ijk]}^c)$. So to prove the equality, it is enough to identify the corresponding orientations. This is easily done by observing the corresponding basis.
\end{proof}
\begin{thm}\label{Tam}
Suppose $M$ is orientable.
\begin{enumerate}
\item  The two $\ASS$-comodule $H^S\CECH^M_{\star\, *}$ and $A^{\star\,*}_M$ of differential bigraded $\kk$-modules are quasi-isomorphic in a manner where $H^S\CECH^M_{-p,-q}$ and $A^{p,q}_M$ correspond for integers $p,q$. (For $H^S\CECH^M$, see Definition \ref{Diroiro}.) 
\item The $E_2$-page of $\check{C}$ech s.s. in Definition \ref{Dcheckmodule} is isomorphic to the total homology of the normalized complex $N A^{\star\,*}_\bullet (M)$. The homogeneous part $\BGSSS^{\,p\,q}_2$ consists of  the summands whose degrees  satify $p=\star-\bullet,\ q=*$. 
\end{enumerate}
\end{thm}
\begin{proof}
For the part 1, we consider the composition
\[
\xymatrix{
H_{-*}^S(\DeltaT_G)\ar[r]^{\quad \zeta_G}
& H^{*-\dd r}_G \ar[r] &   H^{*-\dd r}_Gg_G}.
\]
The right map is given by $a\mapsto \varepsilon_3ag_G$ with the  sign  
\[
\varepsilon_3=(-1)^E, \quad E=E(*',n,r)={*'(n+\dd r)+\dd rn+n(n+1)/2+\dd r(r+1)/2}
\]
on $H^{*'}_G$.  This composition defines an isomorphism as bigraded $\kk$-modules between $H^S\CECH^M(n)^{tr}$ and $\tilde{A}_M(n)^{tr}$. By Lemma \ref{Ldiagonal}, this isomorphism maps $H^S\CECH^M_{-\star,-*}(n)^{tr}\cap I(n)$ into  $\tilde{A}^{\star\,*}_M(n)^{tr}\cap J(n)$ isomorphically. A quasi-isomorphism $H^S\CECH^M(n)\to A_M(n)$ is defined by the following composition
\[
\begin{split}
H^S\CECH^M_{-\star,-*}(n)& \to H^S\CECH^M_{-\star,-*}(n)/I(n)\cong H^S\CECH^M_{-\star,-*}(n)^{tr}/\{ H^S\CECH^M_{-\star, -*}(n)^{tr}\cap I(n) \} \\
& \cong \tilde A^{\star\,*}_M(n)^{tr}/\{ \tilde{A}^{\star\,*}_M(n)^{tr}\cap J(n)\}\cong A^{\star\,*}_M(n).
\end{split}
\]

where the first map is the quotient map which is a quasi-isomorphism by Lemma \ref{Lcechreduction},  the second and forth maps are induced by inclusions, the third map is the isomorphism defined above. For the above number $E$, we see
\[
\begin{split}
E(*',n-1,r)-E(*',n,r) & \equiv *'+\dd r+n \\
E(*'+\dd, n, r-1)-E(*',n,r) & \equiv (*' +1)\dd 
\end{split}
\]
module $2$.
Now we shall take the integer $K$ as a multiple of $4$. With this assumption and the above equalities for $E$,  compatibility of the quasi-isomorphism with the partial composition follows from the Lemma \ref{Lcomposition} as $\varepsilon_1=1$. Compatibility with the ($\CECHC$ech) differentials follows from Lemma \ref{Lcechdiff}. Compatibilty with the actions of $\Sigma_n$ is clear. The sign $\mathrm{sgn} (\sigma_G)$ in Definition \ref{Dcheckmodule}, the sign occurring in permutation of $\gamma_{ij}$ and the sign occurring in permutation of $g_{ij}$ are cancelled. Thus the isomorphism is a morphism of $\ASS$-comodule. 
 For the part 2, by part 1, the $E_2$-page is isomorphic to the homology of Hochschild complex $\Hoch_\bullet(A_M)$, which is isomorphic to unnormalized total complex of $A_\bullet (M)$ so is quasi-isomorphic to the normalized complex.
\end{proof}
Sinha proved the convergence of his spectral sequences using the Cohen-Tayhor spectral sequences. Here, we prove the convergence of our spectral sequences simultaneously by a independent method.
\begin{thm}\label{Tconvergence}
If $M$ is   simply connected, both of the $\check{C}$ech s.s.  and Sinha s.s. for $M$ converge to $H^*(\Emb(S^1,M))$.
\end{thm}
\begin{proof}
We set a number $s_{\dd}$ by $s_{\dd}=\min\{\frac{\dd}{3},2\}$. If $\dd\geq 4$, clearly $s_{\dd}>1$. Recall that $\{\BGSSS_r\}_r$  denotes the $\CECHC$ech s.s. By Lemma \ref{Lsinhass}, we identify the Sinha-s.s. with the spectral sequence $\bar \SINHASS_r$. 
We shall first show the claim that $\BGSSS^{\,-p,\,q}_2=0$ if $q/p < s_{\dd}$.  If a graph $G\in \GG(n+1)$ has $k$ discrete vertices,  $H^*(\Delta_G)$ is isomorphic to $H^*(\TM)^{\otimes k}\otimes H^*(\Delta_{G'})\otimes \{\text{torsions}\}$, where $G'\in \GG(n+1-k)$ is the graph made by removing discrete vertices. With this observation and   simple connectivity of $M$,  we see that  generators of the normalization $NA_n(M)$ are presented as $a_{1}\cdots a_{k} bg_{G}$ where $G$ is a graph in $\GG(n+1)$ with $r$ edges and $k$ discrete vertices except for the vertex $0$, and $a_{t}\in H^{\geq 2}(\TM)$ considered as $t$-th discrete component, and $b\in H^*_{G'}$. We may ignore the torsion part in estimation of degree by the universal coefficient theorem.  The bidegree $(-p,q)$ of this element satisfies conditions $p=n+r$,  and $q\geq 2k+r\dd$.  Clearly we have  $k+2r\geq n+\epsilon$  with $\epsilon=0$ or $1$ according to whether the vertex $0$ has valence $0$ in $G$. With this,  if $\dd\leq 5$, we have the following estimation.
\[
\frac{q}{p}-\frac{\dd}{3}\geq \frac{(6k+(3r-p)\dd }{3(n+r)}\geq \frac{(6-d)k+\dd\epsilon }{3(n+r)}\geq 0
\]
 If $\dd\geq 6$, we have the following estimation.
\[
\frac{q}{p}-2 = \frac{2\epsilon +(\dd -6)r}{n+r}\geq 0
\]
We have shown the claim.  Since the filtration $\{F^{-p}\}$ of $\CECHC$ech s.s. is exhaustive, and the total homology of each $F^{-p}$ is of finite type, the $\CECHC$ech s.s. $\{\BGSSS_r\}_r$ converges to the total homology $H(NA_\bullet(M))$ of the normalized complex. By the same reason, $\{\bar \SINHASS_r\}$ also converges to $H(NA_\bullet(M))$. We shall show $\bar \SINHASS^{\,-p,\,q}_r=0$ if $q/p<s_{\dd} $ for sufficiently large $r$. Suppose there exists a non-zero element $x\in \bar \SINHASS^{\,-p,\,q}_\infty$  with $q/p < s_{\dd}$.  $x$ is considered as an element of 
$(\bar F^{-p}/\bar F^{-p+1})H(NA_\bullet(M))$.  Take a class $x'$ in $\bar F^{-p}H(NA_\bullet(M))$ representing $x$. Take the smallest $p'$ such that $F_{-p'}H(NA_\bullet(M))$ contains $x'$. so $\BGSSS^{\,-p',\,q+p'-p}_\infty $ is not zero and $p'\geq p$ as $\bar F^{-p}\supset F^{-p}$. In the coordinate plane of bidegree,  $x'$ and $x$ are on the same line of the form $-p+q= constant$. This and $p'\geq p$ imply the 'slope' of $x'$ is smaller than $s_{\dd}$, which contradicts to the claim.
This vanishing result on $\bar \SINHASS_r$ and (cohomology version of) Theorem 3.4 in \cite{bousfield} imply the convergence of $\bar \SINHASS_r$ and $\BGSSS_r$ to $H^*(\Emb(S^1,M))$.
\end{proof}
\section{Algebraic presentations of $E_2$-page of $\CECHC$ech spectral sequence}\label{Salgform}
In this section, we assume $M$ is oriented and simply connected and $H^*(M)$ is a free  $\kk$-module.
\begin{defi}\label{Dpoincare}
\begin{itemize}
\item A {\em Poincar\'e algebra of dimension $d$} is a graded commutative algebra $\HH^*$ with a linear isomorphism $\epsilon:\HH^{d}\to \kk$ 
such that the bilinear form  defined as the composition
\[
\HH^*\otimes \HH^* \stackrel{\text{multiplication}}{\longrightarrow} \HH^* \stackrel{\text{projection}}{\longrightarrow}\HH^d\stackrel{\epsilon}{\to} \kk,
\]
induces an isomorphism  $\HH^*\cong (\HH^{d-*})^\vee$.We call $\epsilon$ the {\em orientation of \ $\HH$}. 
\item We denote by $\Delta_{\HH}$ the {\em diagonal class for $\HH^*$} given by
\[
\sum _i (-1)^{|a_i^*|} a_i\otimes a_i^*\in (\HH\otimes \HH)^d
\]
where $\{a_i\}$ and $\{a^*_i\}$ are two basis of $\HH^*$ such that $\epsilon(a_i\cdot a_j^*)=\delta_{ij}$,\ the Kronecker delta. This definition does not depend on a choice of a basis $\{a_i\}$.
\item Let $\HH$ be a Poincar\'e algebra $\HH$ of dimension $d$ with $\HH^1=0$. We set $\HH^{\leq d-2}=\oplus_{p\leq d-2}\HH^p$ and $\HH^{\geq 2}=\oplus_{p\geq 2}\HH^p$ and define a graded $\kk$-module $\HH^{\geq 2}[d-1]$ by $(\HH^{\geq 2}[d-1])^p=X^{p-d+1}$ with $X^*=\HH^{\geq 2}$. We denote by $\bar a $ the element in $(\HH^{\geq 2}[d-1])^p$ corresponding to $a\in \HH^{p-d+1}$. We define a Poincar\'e algebra $S\HH$ of dimension $2{d}-1$ as follows. As graded $\kk$-module, we set
\[
S\HH^*=\HH^{\leq d-2} \oplus \HH^{\geq 2}[d-1].
\]
For $a, b\in \HH^{\leq d-2}$ the multiplication $a\cdot b$ in $S\HH$ is  the one in $\HH$ except for the case $|a|+|b|=d$, in which we set $a\cdot b=0$.  We set 
$ a\cdot \bar b=\overline{ab}$ for  $a\in \HH^{\leq d-1},\ b\in \HH^{\geq 2}$, and  $\bar a\cdot \bar b=0$ for $a,b\in \HH^{\geq 2}$. We give the same orientation on $S\HH$ as on $\HH$ via the identity $S\HH^{2d-1}=\HH^d$. 
\item We  regard $\HH=H^*(M)$ as a Poincar\'e algebra with the orientation 
\[
H^{\dd}(M)\stackrel{ w_M\cap }{\longrightarrow} H_0(M)\cong \kk,
\] where $w_M$ is the fundamental class of $M$ determined by the orientation on $M$ and the latter isomorphism  sends the class represented by a point to $1$.
\end{itemize}
 \end{defi}
 The following lemma is obvious.
 \begin{lem}\label{Lpoincarediagonal}
 With the notations of Definition \ref{Dpoincare}, let $(b_{ij})_{ij}$ denote the inverse  of the matrix $(\eps(a_i\cdot a_j))_{ij}$. Then we have
 \[
 \Delta_\HH=\sum_{i,j}(-1)^{|a_j|}b_{ji}\, a_i\otimes a_j\, .
 \]\hfill \qedsymbol
 \end{lem}
Under some assumption, $S\HH$ is isomorphic to  $H^*(\TM)$ (see the proof of Lemma \ref{Lbhg}), and the algebras $A^*_{\HH,G}$, $B^*_{\HH,G}$ are isomorphic to $H^*(\Delta_G)$.

\begin{defi}\label{Dahgbhg}
For a Poincar\' e algebra $\HH$ of dimension $d$ and graph $G\in \GG(n)$, Define a graded commutative algebra $A_{\HH,G}$
as follows.
\[
A_{\HH,G}= \HH^{\,\otimes\, \pi_0(G)}\otimes \bigwedge \bigl\{\, y_1,\dots, y_n\, \bigr\},\quad \deg y_i=d-1
\]
Here we regard $\pi_0(G)$ as an ordered set by  the minimum in each component, and the tensor product is taken in this order. Furthermore, we also define a graded commutative algebra $B_{\HH,G}$ as follows
\[
B_{\HH,G}= S\HH^{\,\otimes\, n}\otimes \bigwedge \bigl\{\, y_{ij}\mid 1\leq i,j\leq n,\ i\sim_G\, j \, \bigr\}/J_G,\quad \deg y_{ij}=d-1
\]
Here, $i\sim_G j$ means the vertices $i$ and $j$ belong to the same connected component of $G$, and $J_G$ is the ideal generated by the following relation.
\[
\left\{\ \left.
\begin{array}{c}
e_i(a)-e_j(a),\ e_i(\bar a)-e_j(\bar a)-ay_{ij},\vspace{1mm}\\
e_i(\bar b)-e_j(\bar b), \ y_{ii},\ y_{ij}+y_{jk}-y_{ik}
\end{array} \right|
\begin{array}{c}
a\in \HH^{\leq d-2},\ b\in \HH^d, \vspace{1mm}\\
1\leq i,\,j,\,k\leq n,\ i\sim_G j\sim_G k
\end{array} \right\}
\]
Here, $e_j(\bar a)$ is regarded as $0$ if $a\in \HH^0$. \\
\indent For $i<j$, let $f_{ij}:\HH^{\,\otimes 2}\to \HH^{\,\otimes\, n}$ denote the map given by 
\[
f_{ij}(a\otimes b)=1\otimes \cdots \otimes a \otimes \cdots \otimes b \otimes \cdots \otimes 1
\] ($a$ is  the $i$-th factor, $b$ is $j$-th factor and the other factors are $1$ ). We set 
\[
\Delta^{i j}_{\HH}=f_{ij}(\Delta_\HH)\ \in \HH^{\otimes \,n}.
\]
We sometimes regard $\Delta^{ij}_{\HH}$ as an element of $(S\HH)^{\,\otimes\, n}$ via the projection and inclusion $\HH\to \HH^{\leq \dd-1}\subset S\HH$.  We also regard it as an element of $A_{\HH, G}$ for a graph $G$ via the map $\HH^{\,\otimes\, n}\to \HH^{\,\otimes\, \pi_0(G)}$ given by multiplication of factors in the same components with the standard commuting signs. Similarly, $\Delta^{ij}_{\HH}$ and $\Delta^{ij}_{S\HH}$ are regarded as elements of $B_{\HH,G}$.
\end{defi}

\indent As a graded algebra, $B^*_{\HH,G}$ is isomorphic to $(S\HH)^{\, \otimes\, \pi_0(G)}\bigwedge \{y_{ij}\mid \ i\sim_G\, j \}/(y_{ii},\ y_{ij}+y_{jk}-y_{ik}) $ but we need the presentation to describe maps induced by identifying vertices and removing edges.\\
\indent Proof of the following lemma is easy and omitted.
\begin{lem}\label{Lcanonicaliso}
Consider the Serre spectral sequence for a fibration
\[
F\to E\to B
\]
with the base  simply connected  and the cohomology groups of the fiber and base  finitely generated in each degree. If for each $k$, there exists at most single $p$ such that $E^{p,k-p}_\infty\not=0$, the quotient map $F^p\to F^p/F^{p+1}$ has a  unique section which preserves cohomological degree. Gathering these sections for all $p$, one can define an isomorphism of graded algebra $E_\infty\to H^*(E)$. We refer to this isomorphism the {\em canonical isomorphism}. The canonical isomorphisms are natural for maps between fibrations satisfying the above assumption. \hspace{\fill} \qedsymbol
\end{lem}
In the rest of paper, we regard the Euler number $\chi(M)$ as an element of the base ring $\kk$ via the ring map $\ZZ\to \kk$, and $\kk^{\times}\subset \kk$ denotes the subsets of the invertible elements. 
\begin{lem}\label{Lahg}
We use the notations $d_i$, $\Delta_i$,  $s_i$, and $\Delta_{ij}^{!}$ given in Definitions \ref{Diroiro} and \ref{Dcdga}.
Suppose  $\chi (M)=0\in \kk$. Set $\HH^*=H^*(M)$.
 There exists a family of  isomorphisms of graded algebras \\
$\bigl\{\varphi_G:A_{\HH,G}\stackrel{\cong}{\to}  H^*(\Delta_G)\mid n\geq 1,\ G \in \GG(n)\bigr\}$ which satisfies the following conditions.
\begin{enumerate}
\item Let $G\in \GG(n)$ be a tree with $i$ and $i+1$ disconnected. Set $H=d_i(G)$. The following diagram is commutative.
\[
\xymatrix{A^*_{\HH,G}\ar[r]^{\varphi_G}\ar[d]^{\bar \Delta_i} & H^*(\Delta_G)\ar[d]^{ \Delta_i^*}\\
A^*_{\HH,H}\ar[r]^{\varphi_{H}} & H^*(\Delta_{H}).}  
\]
Here, the algebra map $\bar \Delta_i$ is defined as follows.    For $a_1\otimes \cdots \otimes a_p\in \HH^{\otimes \pi_0(G)}$, we set
\[
\bar \Delta_i(a_1 \otimes \cdots \otimes a_p)=\pm a_1\otimes \cdots\otimes a_s\cdot a_t\otimes \cdots a_p, \qquad  \bar \Delta_i(y_j)=y_{j'}\ \text{ with } \ j'=d_i(j).
\]
Here, $s, t\in \pi_0(G)$ are the connected components containing $i$, $i+1$ respectively, and $\pm$ is the standard sign in transposing graded elements. 
\item For a graph $G\in \GG(n)$, set $S=s_i(G)$. The following diagram is commutative. 
\[
\xymatrix{A^*_{\HH,G}\ar[r]^{\varphi_G}\ar[d]^{\bar s_i} & H^*(\Delta_G)\ar[d]^{ s_i}\\
A^*_{\HH,S}\ar[r]^{\varphi_{S}} & H^*(\Delta_{S}).}  
\]
Here, $\bar s_i$ is given by  insertion of the unit\ $1$\ as the factor of $H^{\otimes \pi_0(G)}$ which corresponds to the component containing $i+1$  and  skip of $i+1$ in the subscripts. 
\item For a graph $G\in \GG(n)$ and a permutation $\sigma\in \Sigma_n$, the following diagram is commutative.
\[
\xymatrix{A^*_{\HH,G}\ar[r]^{\varphi_G}\ar[d]^{\bar \sigma} & H^*(\Delta_G)\ar[d]^{ \sigma}\\
A^*_{\HH,\tau(G)}\ar[r]^{\varphi_{\tau(G)}} & H^*(\Delta_{\tau(G)}). }
\]
Here $\tau =\sigma^{-1}$, and the right vertical arrow is induced by the natural permutation of factors of the product $\Delta_{\tau(G)}\to \Delta_G$  and for the left vertical arrow, $\bar \sigma$ is the  algebra map given by the permutation of tensor factors and subscripts. 
\item For  an edge  $(i,j)$ of a tree $G\in\GG(n)$ with $i<j$, we define $K\in \GG(n)$ by $E(K)=E(G)-\{(i,j)\}$. The following diagram is commutative.
\[
\xymatrix{A^*_{\HH,G}\ar[r]^{\varphi_G}\ar[d]^{\bar \Delta_{ij}} & H^*(\Delta_G)\ar[d]^{ \Delta^!_{ij}}\\
A^{*+\dd}_{\HH,K}\ar[r]^{\varphi_{K}} & H^{*+\dd}(\Delta_{K}).}  
\]
Here, $\bar \Delta_{ij}^!$ is the right-$A^*_{\HH,K}$-module homomorphism determined by $\bar \Delta_{ij}^!(1)=\Delta^{ij}_\HH$\, . $A^*_{\HH,G}$ is considered as a $A^*_{\HH,K}$-module via the natural algebra map $A^*_{\HH,K}\to A^*_{\HH,G}$.
\end{enumerate} 
\end{lem}
\begin{proof}
In the following proof, we fix a generator $y$ of $H^{\dd-1}(S^{\dd-1})$. and we denote by $y_i$ ( or $\bar y_i$ ) the image of $y$ by the inclusion to the $i$-th factor  
$H^{\dd-1}(S^{\dd-1})\to H^{\dd-1}(S^{\dd-1})^{\otimes \,n}$. We consider Serre spectral sequence for the fibration
\[
(S^{\dd-1})^{\times n}\longrightarrow  \Delta_G \longrightarrow M^{\times\,\pi_0(G)}, 
\]
where the projection is the restriction of that of sphere tangent bundle. The first possibly non-trivial differential is $d_{\dd}:H^{\dd-1}((S^{\dd-1})^{\times n})=E^{0,\dd-1}_{\dd}\to E^{\dd,0}_{\dd}=H^{\dd}(M)$. This differential takes $y_i$ to the generator of $H^d(M)$ multiplied by $\chi(M)$. As $\chi(M)=0$, we have $d_{\dd}=0$, and latter differentials on $y_i$ is zero by degree reason, we see that $y_i$ survives eternally, which implies  $E_2\cong E_\infty$. Clearly $E_\infty$ satisfies the assumption of Lemma \ref{Lcanonicaliso}. We define $\varphi_G$ as the composition of the following maps
\[
A_{\HH,G}\to E_2=E_\infty\to H^*(\Delta_G )
\]
where the left map is the isomorphism given by identifying $y_i$ in the both sides and $\HH^{\otimes \pi_0(G)}$ with $H^*(M^{\times \pi_0(G)})$ by the Kunneth isomorphism, and the right map is the canonical isomorphism defined in Lemma \ref{Lcanonicaliso}. The parts 1, 2 and 3 obviously follow from naturality of the canonical isomorphisms. For the part 4, $H^*(\Delta_{G})$ is regarded as $H^*(\Delta_{K})$-module via the pullback $\Delta_{ij}^*:H^*(\Delta_{K})\to H^*(\Delta_G)$ by the inclusion $\Delta_G\to \Delta_K$. This module structure is compatible with the $A^*_{\HH,K}$-module structure on $A^*_{\HH,G}$ via $\varphi_G$ and $\varphi_{K}$ by  naturality of canonical isomorphism. By a general property of a  shriek map, the  map $\Delta_{ij}^!$ is $H^*(\Delta_{K})$-module homomorphism. So to prove the compatibility, we have only to check the image of $1$. For simplicity, we may assume $n=2$ and $(i,j)=(1,2)$. We may write $\Delta_G$ as $\TM\times_M\TM$.  The following diagram is  commutative.
\[
\xymatrix{H_{d-*}(M)\ar[r]^{P.D.} \ar[d]^{\Delta_*} & H^*(M) \ar[r]^{proj^*} \ar[d]^{\Delta^!} & H^*(\TM\times _M\TM)\ar[d]^{\Delta_{12}^!} \\
H_{d-*}(M\times M)\ar[r]^{P.D.} & H^{*+d}(M\times M)\ar[r]^{proj^*} & H^{*+d}(\TM \times \TM)} 
\] 
Here, P.D. denotes the cap product with the fundamental class. By the commutativity of the left square, we see $\Delta^!(1)$ is the   Poincar\'e dual class in $H^*(M\times M)$ of the diagonal $\Delta(M)$, which corresponds to $\Delta_\HH$ by Kunneth isomorphism. By the commutativity of the right square, we see $\Delta_{12}^!(1)$ corresponds  to $f_{ij}\Delta_\HH$. This completes the proof.

\end{proof}


\begin{lem}\label{Lbhg}
We use the notations $d_i$, $\Delta_i$,  $s_i$, and $\Delta_{ij}^{!}$ given in Definitions \ref{Diroiro} and \ref{Dcdga}.
Suppose   $\chi(M)\in \kk^{\times}$. Set $\HH=H^*(M)$. 
 There exists a family of  isomorphisms of graded algebras \\
$\bigl\{\varphi_G:B_{\HH,G}\stackrel{\cong}{\to}  H^*(\Delta_G)\mid n\geq 1,\ G \in \GG(n)\bigr\}$ which satisfies the following conditions.
\begin{enumerate}

\item Let $G$ and $H$ be trees given in Lemma \ref{Lcomposition} (1). The following diagram is commutative.
\[
\xymatrix{B_{\HH,G}\ar[r]^{\varphi_G}\ar[d]^{\bar \Delta_i} & H^*(\Delta_G)\ar[d]^{ \Delta_i^*}\\
B_{\HH,H}\ar[r]^{\varphi_{H}} & H^*(\Delta_{H}).}  
\]
Here,  $\bar \Delta_i$ is defined by 
\[
\bar \Delta_i(e_j(x))=
e_{j'}(x) \ 
\text{ for }\ x\in S\HH,\quad \text{and} \quad 
\bar\Delta_i(y_{jk})=y_{j'k'}
\] where we set $j'=d_i(j)$ and $k'=d_i(k)$.

\item For a graph $G\in \GG(n)$, set $S=s_i(G)$. The following diagram is commutative.
\[
\xymatrix{B^*_{\HH,G}\ar[r]^{\varphi_G}\ar[d]^{\bar s_i} & H^*(\Delta_G)\ar[d]^{ s_i}\\
B^*_{\HH,S}\ar[r]^{\varphi_{S}} & H^*(\Delta_{S}).}  
\]
Here, $\bar s_i$ is given by the insertion of $1$ to the $i+1$-th factor of $S\HH^{\otimes n}$ and  skip of $i+1$ in the subscripts.

\item For a graph $G\in \GG(n)$ and a permutation $\sigma\in \Sigma_n$, the following diagram is commutative.
\[
\xymatrix{B^*_{\HH,G}\ar[r]^{\varphi_G}\ar[d]^{\bar \sigma} & H^*(\Delta_G)\ar[d]^{ \sigma}\\
B^*_{\HH,\tau(G)}\ar[r]^{\varphi_{\tau(G)}} & H^*(\Delta_{\tau(G)}) .}
\]
Here $\tau $ and the right vertical arrow are defined as in Lemma \ref{Lahg}, $\bar \sigma$ is the  algebra homomorphism defined by the permutation of the tensors and subscripts.

\item 
For  an edge  $(i,j)\in E(G)$ of a tree $G\in\GG(n)$ with $i<j$, we define  $K\in \GG(n)$ by $E(K)=E(G)-\{(i,j)\}$. The following square is commutative.
\[
\xymatrix{B^*_{\HH,G}\ar[r]^{\varphi_G}\ar[d]^{\bar \Delta_{ij}} & H^*(\Delta_G)\ar[d]^{ \Delta^!_{ij}}\\
B^{*+\dd}_{\HH,K}\ar[r]^{\varphi_{K}} & H^{*+\dd}(\Delta_{K}).}  
\]
Here,  $\bar \Delta_{ij}$ is the right-$B^*_{\HH,K}$-module homomorphism determined by 
$\bar \Delta_{ij}(1)=\Delta_\HH^{ij}$ and $\bar \Delta_{ij}(y_{ij})=\Delta_{S\HH}^{ij}$\, . $B^*_{\HH,G}$ is considered as a $B^*_{\HH,K}$-module via the algebra map $f_K^G:B^*_{\HH,K}\to B^*_{\HH,G}$ given by
\[
f_K^G(e_k(x))=e_k(x)\quad \text{for}\ x\in S\HH, \quad \text{and}\quad  f_K^G(y_{kl})=\left\{
\begin{array}{cc}
0 & ((k,l) =(i,j))\\
y_{kl}  & (\text{otherwise}).
\end{array}\right.
\]
\end{enumerate} 
\end{lem}
\begin{proof}
As in the proof of Lemma \ref{Lahg}, we fix a generator $y\in H^{\dd-1}(S^{\dd-1})$. $d$ is automatically even as $\chi(M)\not =0$.  We first show an isomorphism of algebras $S\HH^*\cong H^*(\TM)$. Consider the Serre spectral sequence for the sphere tangent fibration
\[
S^{\dd-1}\to \TM\to M
\] 
The only non trivial differential is $d_{\dd}:E^{0,\dd-1}=H^{\dd-1}(S^{\dd-1})\to H^{\dd}(M)$. As $\chi(M)$ is invertible, $d_{\dd}$ is an isomorphism. Since all other differentials vanish by degree reason, we have $E_\infty\cong E_{\dd+1}\cong S\HH$, where the second isomorphism is given by $E^{p,0}_{\dd+1}= H^p(M) \subset \HH^{\leq \dd-2}\ \subset S\HH$ for $p\leq \dd-2$ and  $E^{p,\dd-1}=H^{\dd-1}(S^{\dd-1})\otimes H^p(M)\ni y\otimes a\mapsto \bar a \in S\HH$ for $p\geq 2$. Since $H^1(M)=0$ and $H^*(M)$ is free, $H^{\dd-1}(M)=0$ which implies the fibration satifies the condition of Lemma \ref{Lcanonicaliso}. Composing this isomorphism with the canonical isomophism $E_\infty\to H^*(\TM)$, we have an isomorphism 
\[
S\HH^*\cong H^*(\TM) \cdots\cdots (*)
\]  If necessary, we modify $y$ so that the composition $S\HH^{2\dd-1}\to H^{2\dd-1}(\TM)\to \kk$ of the map (*) and the cap product with the fundamental class $\tw$ in Definition \ref{Diroiro} coincides with the orientation given in Definition \ref{Dpoincare} by multiplying a scalar. 
We shall define the isomorphism $\varphi_G$. We may assume that $G\in \GG(n)$ is connected as in disconnected case, everything involved is a tensor product of the objects corresponding to connected subgraphs. Consider the Serre spectral sequence for the fibration
\[
(S^{\dd-1})^{\times n-1}\to \Delta_G\to \TM
\] 
 given by the projection to the first component. As $E^{\dd,0}_2=S\HH^{\dd}=0$, elements in $E^{0,\dd-1}_2\cong H^{\dd-1}(S^{\dd-1})^{\otimes n-1}$ survive eternally. As in the proof of Lemma \ref{Lahg}, $y_j$ denotes the copy of $y$ living in the $j$-th factor of $H^*(S^{\dd-1})^{\otimes n-1}$, which is also regarded as a  generator of $E^{0,\dd-1}_2$.

We construct an isomorphism  $\psi_G:S\HH^*\otimes \bigwedge (y_1,\dots y_{n-1})\cong E_\infty\cong H^*(\Delta_G)$ using the isomorphism (*), similarly to the construction of the isomorphism (*). Consider the Serre spectral sequence $\{\bar E^{p,q}_r\}$ for the fibration
\[
(S^{\dd-1})^{\times n}\to \Delta_G\to M
\]
given by the projection of the sphere bundle. Let $\bar y_j$ be the copy of $y$ in the $j$-th factor of  $\bar E^{0,\dd-1}_2 \cong (H^*(S^{\dd-1})^{\otimes n})^{*=\dd-1}$. For any $i,j$, since $d_{\dd}(\bar y_i)=d_{\dd}(\bar y_j)=\text{( a multiple of )}\chi(M)w_M$, $\bar y_i-\bar y_j$ survives eternally by degree reason. Clearly $\bar E_\infty$ satisfies the assumption of Lemma \ref{Lcanonicaliso}, we can take the canonical isomorphism $\bar E^{*,*}_\infty\to H^*(\Delta_G)$. We define an algebra map
\[
\varphi'_G:(S\HH)^{\otimes n}\otimes \bigwedge \{y_{ij}\mid 1\leq i,j\leq n\}\to \bar E^{*,*}_\infty
\] 
by $e_i(a)\mapsto a\in E^{*,0}_\infty$ for $a\in \HH^{\leq \dd-2}$,\ $e_i(\bar b)\mapsto b\bar y_i\in E^{*,\dd-1}_\infty$, for $b\in \HH^{\geq 2}$, and $y_{ij} \mapsto \bar y_i-\bar y_j$. \ We see $\varphi'_G(J_G)=0$ where $J_G$ is the ideal in Definition \ref{Dahgbhg}. For example, since $d_{\dd}(\bar y_i\bar y_j)=\chi (M)(\bar y_j-\bar y_i)w_M$ (up to $\kk^\times$) and $\chi(M)$ is invertible, $(\bar y_i-\bar y_j)w_M=0$ in $\bar E^{\dd,\dd-1}_{\dd+1}$ , which implies $\varphi'_G(e_i(\bar b)-e_j(\bar b))=0$ for $b\in \HH^{d}$. Annihilation of other elements in $J_G$ is obvious. We define $\varphi_G$ to be the unique map which makes the following diagram commutative.
\[
\xymatrix{ (S\HH )^{\otimes n}\otimes \bigwedge \{y_{ij} \} \ar[r] \ar[d]^{\varphi'_G} &   (S\HH )^{\otimes n}\otimes \bigwedge \{y_{ij} \}/J_G \ar@{=}[r]& B^*_{\HH,G} \ar[d]^{\varphi_G} \\
\bar E^{*,*}_\infty \ar[rr]^{can.iso} && H^*(\Delta_G)} 
\]
Since $G$ is connected, $e_1:S\HH\to S\HH^{\otimes n}$  induces an  isomorphism $\alpha_G:S\HH\otimes \bigwedge\{y_{12},\dots,y_{1n}\}\cong B^*_{\HH,G}$. It is easy to see the composition
\[
S\HH\otimes \bigwedge\{y_{12},\dots,y_{1n}\}\stackrel{\alpha_G}{\cong} B^*_{\HH,G}\stackrel{\varphi_G}{\to}  H^*(\Delta_G) \stackrel{\psi_G^{-1}}{\cong} S\HH\otimes \bigwedge\{y_1,\dots y_n \}
\]
identifies the subalgebra $S\HH$ in the both side and the sub $\kk$-module $\kk\langle y_{12},\dots, y_{1n}\rangle$ with  $\kk\langle y_1,\dots y_n\rangle$ (since these are both isomorphic to $H^{\dd-1}(\Delta_G)$), which implies the composition is isomorphism and we conclude $\varphi_G$ is an isomorphism. \\
\indent The parts 1, 2 and 3 obviously follow from naturality of the canonical isomorphism. We shall show the part 4. Since $\varphi_G$ is an isomorphism, we may define $\bar \Delta_{ij}$ to be the map which makes the square in the part 4 commutes. As in the proof of Lemma \ref{Lahg}, $\bar \Delta_{ij}$ is $B^*_{\HH,K}$-module homomorphism and we have $\bar \Delta_{ij}(1)=f_{ij}(\Delta_\HH)$. We shall show the equality $\bar \Delta_{ij}(y_{ij})=f_{ij}(\Delta_{S\HH})$. We may assume $n=2$ and $G={(1,2)}$. In this case clearly $\Delta_G=\TM\times_M\TM$. We consider the following commutative diagram.
\[
\xymatrix{H^0(S^{\dd-1})\ar[d]^{\Delta_1^!}& H^0(\TM)\ar[r]^{P.D.}\ar[l]\ar[d]^{\Delta_2^!}  & H_{2\dd-1}(\TM)\ar[d]^{(\Delta_2)_*} \\
H^{\dd-1}(S^{\dd-1}\times S^{\dd-1}) & H^{\dd-1}(\TM\times _M\TM )\ar[d]^{\Delta_{12}^!}\ar[r]^{P.D.} \ar[l] & H_{2\dd-1}(\TM\times_M\TM)\ar[d]^{(\Delta_{12})_*} \\
& H^{2\dd-1}(\TM\times \TM)\ar[r]^{P.D.} & H_{2\dd-1}(\TM\times \TM)}
\]
where the left horizontal arrows are induced by the fiber restriction, and the right ones are capping with the fixed fundamental classes, and $\Delta^!_1$ and $\Delta^!_2$ are the shriek maps induced by the diagonals. As $\dd$ is even, we have $\Delta_1^!(1)=\bar y_1-\bar y_2$.   As $\bar y_1-\bar y_2$ conincides with the image of $\varphi_G(y_{12})$ by the fiber restriction which induces an isomorphism in degree $\dd-1$, we have $\Delta_2^!(1)=\varphi_G(y_{12})$. So we have 
$\Delta_{12}^!(\varphi_G(y_{12}))=(\Delta_{12}\circ\Delta_2)^!(1)$. By the commutativity of the right hand side square, $ (\Delta_{12}\circ\Delta_2)^!(1)$ is the diagonal class for $\TM$. Thanks to the modification of $y$ after the definition of the isomorphism (*), the diagonal class corresponds to $\Delta_{S\HH}$ by $\varphi_G$. This implies $\bar \Delta_{12}(y_{12})=\Delta_{S\HH}$.
\end{proof}

\begin{defi}\label{Dahnbhn}
Let $\HH$ be a Poincar\'e algebra of dimension $\dd$.
\begin{itemize}
\item 
 We define a  CDBA  $A_\HH^{\star\,*}(n)$ by the equality
\[
A_\HH^{\star\, *} (n)=\HH^{\otimes n}\otimes \bigwedge \{ \, y_i, \ g_{ij}\mid 1\leq i,j\leq n\}/\mathcal{I}.
\]
Here, for the bidegree, we set  $|a|=(0, l)$ for $a\in (\HH^{\otimes n})^{*=l}$ and  $|y_i|=(0,\dd-1)$, and $|g_{ij}|=(-1,\dd)$. The ideal $\mathcal{I}$ is generated by the elements
\[
\begin{split}
g_{ij}-(-1)^{\dd}g_{ji},\quad (g_{ij})^2,\quad g_{ii},&\quad (e_i(a)-e_j(a))g_{ij},\vspace{\baselineskip}\\
 g_{ij}g_{jk}+g_{jk}g_{ki}+g_{ki}g_{ij} &\quad (1\leq i,j,k\leq n,\ a\in \HH ). 
\end{split}
\]
We call the last relation the {\em 3-term relation for $g_{ij}$}.
The differential is given by $\partial(a)=0$ for $a \in \HH^{\otimes n}$ and $\partial (g_{ij})=\Delta_\HH^{ij}$\, , see Definition \ref{Dahgbhg}.
\item Suppose $\dd$ is even. We define a CDBA $B_\HH^{\star\,*}(n)$ by the equality
\[
B_\HH^{\star\,*}(n)=(S\HH )^{\otimes n}\otimes \bigwedge \{\, g_{ij},\ h_{ij},\ \mid 1\leq i,j\leq n\}/\mathcal{J}
\] 
Here, for the bidegree, we set  $|a|=(0, l)$ for $a\in (\HH^{\otimes n})^{*=l}$ and $|g_{ij}|=(-1,\dd)$ and $|h_{ij}|=(-1,2\dd-1)$. The ideal $\mathcal{J}$ is generated by the following elements.
\[
\begin{split}
g_{ij}-g_{ji},\quad (g_{ij})^2,\quad g_{ii},&\quad ,h_{ij}+h_{ji},\quad (h_{ij})^2,\quad h_{ii}, \quad \vspace{1\baselineskip}\\
e_{ij}({a})g_{ij},\quad e_{ij}({a})h_{ij},\quad & e_{ij}(\bar b)g_{ij}-e_i({b})h_{ij},\quad e_{ij}(\bar b)h_{ij}, \vspace{\baselineskip}\\
 g_{ij}g_{jk}+g_{jk}g_{ki}+g_{ki}g_{ij},\quad  &  h_{ij}h_{jk}+h_{jk}h_{ki}+h_{ki}h_{ij}  \vspace{\baselineskip}\\
(h_{ij}+h_{ki})g_{jk}-(h_{ij}+h_{jk})g_{ki}
&\qquad (1\leq i,j,k\leq n,\ a\in \HH^{\leq \dd-2}, b\in \HH^{\geq 2} ) \\ 
\end{split}
\]
where  we regard $e_i(b)$ as $0$ for $b\in \HH^{\dd}$, and $e_{ij}:S\HH\to (S\HH )^{\otimes n}$ is the map given by $e_{ij}=e_i-e_j$.  The differential is given by $\partial (x)=0$\ for $x \in S\HH^{\otimes n}$ and $\partial (g_{ij})=\Delta_\HH^{ij}$ and $\partial (h_{ij})=\Delta_{S\HH}^{ij}$\, , see Definition \ref{Dahgbhg}. 
\item  We equip the sequences $A_\HH=\{A_\HH(n)\}_n$ and $B_\HH=\{B_\HH(n)\}_n$ with structures of     $\ASS$-comodules of CDBA as follows. For $B_\HH$, we define a partial composition and an action of $\Sigma_n$ by the equalities
\[
\begin{split}
\mu\circ_i e_j(x)=e_{j'}(x), \quad \mu\circ_i(h_{j\,k})=h_{j'k'}, &       \quad \mu\circ_i(g_{j\, k})=g_{j'k'}, \quad e_j(x)^\sigma =e_{\tau(j)}(x),\\
h_{j\,k}^\sigma=h_{\tau(j),\tau(k)}, \quad g_{j\,k}^\sigma=g_{\tau(j),\tau(k)} &      \quad (x\in S\HH,\ \sigma\in \Sigma_n) 
\end{split}, 
\]
where $j'$ and $k'$ are numbers given by $j'=d_i(j)$ and $k'=d_i(k)$ and $\tau=\sigma^{-1}$ (see Definition \ref{Diroiro} for $d_i$ and $\mu$). The definition of $A_\HH$ is similar.
\item We define simplicial CDBA's $A^{\star\,*}_\bullet (\HH)$ and $B^{\star\,*}_\bullet (\HH)$ as follows. For $B^{\star\,*}_\bullet(\HH)$, we set
\[
B^{\star\,*}_n(\HH)= B^{\star\,*}_\HH (n+1).
\]
As in Definition \ref{Dcdga}, we relabel the involved subscripts  with $0,\dots, n$. 
The face map $d_i:B^{\star\,*}_n(\HH)\to B^{\star\,*}_{n-1}(\HH)$ is given by $d_i=\mu\circ_i(-)$ ($i<n$) and $d_n=\mu\circ_0(-)^\sigma$ where $\sigma=(n,0,1,\dots, n-1)$. The degeneracy map $s_i:B^{\star\,*}_n(\HH)\to B^{\star\,*}_{n+1}(\HH)$ is given by insertion of  $1$ as the $i+1$-th factor of $S\HH^{\otimes n+1}$ and skip of the subscript $i+1$. $A^{\star\, *}_\bullet(\HH)$ is defined similarly using $A^{\star\,*}_\HH$.

 \end{itemize}
 \end{defi}
In the rest of this section, we prove that $A_\HH$ and $B_\HH$  are isomorphic to $A_M$ as a $\ASS$-comodule of CDBA under different assumptions, and also prove similar statements for the simiplicial CDGA's. We mainly deal with the case of $B_\HH$. The case of $A_\HH$ is similar.
\begin{lem}\label{Lgraphincl}
 The map 
\[
\bigoplus_{G\in \GG(n)^{dis} }H^*_Gg_G\to A_M
\]
defined by the composition of the inclusion and  quotient map is an isomorphism of $\kk$-modules (see Definition \ref{Diroiro} for $\GG(n)^{dis}$).
\end{lem}
\begin{proof}
Let $\Pi$ be the set of partitions of $\underline{n}$. The ideal $J(n)$ in Definition \ref{Dcdga} has a decomposition $J(n)=\oplus_{\pi\in \Pi}J(n)_\pi$ such that $J(n)_\pi\subset \oplus_{\pi_0(G)=\pi}H_G$ since generaters of $J(n)$ are sums of monomials which have the same connected components. If $\pi_0(G)=\pi_0(H)=\pi$, clearly $H^*_G=H^*_H$. We denote this module by $H^*_\pi$. We have $\oplus_{\pi_0(G)=\pi}H_Gg_G=H_{\pi}\otimes (\oplus_{\pi_0(G)=\pi}\kk g_G)$. Similarly We have $J(n)_\pi=H_\pi\otimes J(n)'_\pi$ where $J(n)'_\pi$ is the sub $\kk $-module of $\oplus_{\pi_0(G)=\pi}\kk g_G$ generated by multiples of 3-term relations, $g_{ij}^2$ and $g_{ij}-(-1)^dg_{ji}$. We have
\[
A_M^*=\oplus_{\pi\in \Pi}\{ (\oplus_{\pi_0(G)=\pi}H_Gg_G)/J(n)_\pi\}=\oplus_{\pi\in \Pi}H_\pi\otimes \{(\oplus_{\pi_0(G)=\pi}\kk g_G )/J(n)'_\pi \}
\] 
Note that $\oplus_{\pi\in \Pi}\{(\oplus_{\pi_0(G)=\pi}\kk g_G )/J(n)_\pi' \}$ is isomorphic to $H^*(C_n(\RR^{\dd}))$, whose basis is  $\{g_G\mid G\in \GG(n)^{dis}\}$. So $(\oplus_{\pi_0(G)=\pi}\kk g_G )/J(n)'_\pi$ has a basis $\{g_G\mid G\in \GG(n)^{dis},\quad  \pi_0(G)=\pi\}$, which implies  the lemma.
\end{proof}
Under the assumptions and notations of Lemma \ref{Lbhg}, we identify $H^*_G$ with $B_{\HH,G}$ by the isomorphism $\varphi_G$ so $A^*_M(n)$ is regarded as a quotient of $\oplus_{G\in \GG(n)}B^*_{\HH,G}g_G$. With this identification, we set $\bar h_{ij}=y_{ij}g_{ij}\in A_M(n)$. $A_M(n)$ contains $\SHH^{\otimes n}$ as the subalgebra  $H_{\emptyset}g_{\emptyset}$, the summand corresponding to the graph $\emptyset\in \GG(n)$. We regard $A_M(n)$ as a left $S\HH^{\otimes n}$-module via the multiplication by $H_{\emptyset}g_\emptyset$. In the  following lemma and its proof, $h_G$, $\bar h_G$ and $y_G$ are notations similar to $g_G$. For example, $h_G=h_{i_1,j_1}\cdots h_{i_r,j_r}$ for $E(G)=\{(i_1,j_1)<\cdots <(i_r,j_r)\}$.
\begin{lem}\label{Ladmissiblegenerator}
Under the assumptions of Lemma \ref{Lbhg}, and the above notations, 
as $S\HH^{\otimes n}$-module, $A_M(n)$ is generated by the set $S=\{g_G\bar h_H \mid G, H\in \GG(n), \ E(G)\cap E(H)=\emptyset,\ GH\in \GG(n)^{dis} \}$, and $B_\HH(n)$ is generated by the set $S'=\{g_Gh_H \mid  G, H\in \GG(n),\ E(G)\cap E(H)=\emptyset,\ GH\in \GG(n)^{dis}\}$.
\end{lem} 
\begin{proof}
$A_M(n)$ is generated by elements $y_Hg_G$ for  graphs $G$ an $H$ such that each connected component of $H$ is contained in some connected component of $G$. We can express $g_G$ as a sum of monomials $g_{G_1}$ with $G_1\in\GG(n)^{dis}$  and $\pi_0(G)=\pi_0(G_1)$ using the 3-term relation and the relation $g_{ij}=g_{ji}$ (This is standard procedure in the computation of $H^*(C_n(\RR^{\dd})$). So we may assume $G$ is distinguished.  For a sequence of  edges   $(i,k_1),(k_1,k_2),\dots, (k_s,j)$ in $G$, we have $y_{ij}=y_{i,k_1}+\cdots y_{k_s,j}$. By successive application of this equality, $y_H$ is expressed as a sum of monomials $y_{H_1}$ with $H_1$ being a subgraph of $G$. Thus we see any element of $A_M(n)$ is expressed as a $S\HH^{\otimes n}$-linear combination of monomials $y_Hg_G$ with $G\in \GG(n)^{dis}$ and $E(H)\subset E(G)$. Clearly, we see $y_Hg_G=\pm g_{G-H}\bar h_{H}$. Thus we have seen the set $S$ generates $A_M(n)$.  Proof for the assertion for $B_\HH (n)$ is similar when one use  3-term relations for $g_{ij}$ and $h_{ij}$, and the last relation for $g_{ij}$ and $h_{ij}$ in the ideal $\mathcal{J}$ in Definition \ref{Dahnbhn}. 
\end{proof}
To prove $B_{\HH}(n)$ and $A_M(n)$ are isomorphic, we define a structure of $B_{\HH,G}$-module on $B_\HH(n)$ as follows. We first define two graded algebras $\TB_{\HH,G}$ and $\TB_{\HH}(n)$. For a graph $G\in \GG(n)$, we set 
\[
\TB_{\HH,G}=S\HH^{\otimes n}\otimes\, T\{ y_{ij} \mid i<j,\  i\sim_G j\},\quad\quad \TB_{\HH}(n)=\SHH^{\otimes n}\otimes \bigwedge \{ g_{ij},\ h_{ij}, \mid 1\leq i<j\leq n\}
\]
where $T\{y_{ij}\}$ denotes the tensor algebra generated by $y_{ij}$'s. For convenience, we set $y_{ij}=-y_{ji}$, $g_{ij}=g_{ji}$, and $h_{ij}=-h_{ji}$ for $i>j$.  The degrees are the same as the elements of the same symbols in $B_{\HH,G}$ and $B_{\HH}(n)$. We shall define a map of graded $\kk $-modules
\[
(-\cdot -):\TB_{\HH,G}\otimes _{\kk }\TB_\HH (n)\to B_\HH (n).
\]
We define $y_{ij}\cdot xg_Gh_H$ ($x\in \SHH^{\otimes n},\ G, H\in \GG(n)$) as follows. If $E(G)\cap E(H)\not=\emptyset$, we set $y_{ij}\cdot xg_Gh_H=0$. Suppose $E(G)\cap E(H)=\emptyset $. If $(i,j)\in E(G)$ is the $t$-th edge (in the lexicographical order), we set $y_{ij}\cdot xg_Gh_H=(-1)^{t+1+|x|} h_{ij}xg_Kh_H$ with $E(K)=E(G)-\{(i,j)\}$. If $(i,j)\in E(H)$ is an edge, we set $y_{ij}\cdot xg_Gh_H=0$. If $i\sim_{GH} j$, we take a sequence of edges $(k_0,k_1),\dots, (k_s,k_{s+1})$ of $GH$ with $k_0=i$ and $k_{s+1}=j$ and set 
$y_{ij}\cdot xg_Gh_H=\sum_{l=0}^sy_{k_l,k_{l+1}}\cdot xg_Gh_H$. This does not depend on the choice of the sequence because $g_Gh_H=0$ if $GH$ is not a tree, which  is proved by using the last three relations in the definition of $\mathcal{J}$ in Definition \ref{Dahnbhn}. If $i$ and $j$ are disconnected in $GH$, we set $y_{ij}\cdot xg_Gh_H=0$. For $z \in \SHH^{\otimes n}$, we set $z\cdot xg_Gh_H=z xg_Gh_H$, the multiplication in $B_\HH(n)$. We shall show the map $(-\cdot -)$ annihilates the elements of $\mathcal{J}$ (We regard $\mathcal{J}$ as an ideal in $\TB_\HH(n)$). Direct computation shows the generators of $\mathcal{J}$ are anihilated by any elements of $\TB_{\HH,G}$. For example, $y_{ij}\cdot(g_{ij}g_{jk}+g_{jk}g_{ki}+g_{ki}g_{ij})=(h_{ij}+h_{ik})g_{jk}-(h_{ij}+h_{jk})g_{ki}=0$ and $y_{jk}\cdot \{(h_{ij}+h_{ki})g_{jk}-(h_{ij}+h_{jk})g_{ki}\}=h_{ij}h_{jk}+h_{jk}h_{ki}+h_{ki}h_{ij}=0$. We also easily see $y_{ij}\cdot xg_Gh_H=\pm (y_{ij}\cdot xg_{G'}h_{H'})g_{G-G'}h_{H-H'}$ for subgraphs $G'\subset G$ and $H'\subset H$ such that $i\sim _{G'H'} j$. These observations imply the assertion and we see the map $(-\cdot-)$ factors through a map  $\TB_{\HH,G}\otimes _{\kk}  B_\HH(n)\to B_\HH(n)$ which is also denoted by $(-\cdot -)$. Clearly, the map $(-\cdot -)$ annihilates $J_G$ in the definition of $B_{\HH,G}$. It also  annihilates the commutativity relation $y_{ij}y_{kl}+y_{kl}y_{ij}$.  If two paths connecting $i$ and $j$ or $k$ and $l$ have a common edge, Both of the action of $y_{ij}y_{kl}$ and $y_{kl}y_{ij}$ are zero and otherwise the commutativity in $B_\HH(n)$ implies the annihilation. Annihilation of these relations  implies the map $(-\cdot -)$ factors through a map $(-\cdot -):B_{\HH,G}\otimes B_{\HH}(n)\to B_{\HH}(n)$ which defines a structure of $B_{\HH,G}$-module on $B_{\HH}(n)$.

\begin{thm} \label{Talgebraicss}
Suppose $M$ is simply connected and oriented, and $H^*(M)$ is a free  $\kk $-module. Set $\HH=H^*(M)$.
\begin{enumerate}
\item Suppose $\chi(M)=0\in \kk $.  Two $\ASS$-comodules of CDBA $A^{\star\,*}_M$ and $A^{\star\,*}_\HH$ are isomorphic, and  two simplicial CDBA $A^{\star\,*}_\bullet(M)$ and $A^{\star\,*}_\bullet(\HH)$ are isomorphic. In particular, the $E_2$-page of $\CECHC$ech s.s. is isomorphic to the total homology of the normalization $NA^{\star\,*}_\bullet(\HH)$ as a bigraded $\kk $-module. The bigrading is given by $(\star-\bullet, *)$
\item Suppose  $\chi(M)\in \kk^{\times}$. Two $\ASS$-comodules of CDBA $A^{\star\,*}_M$ and $B^{\star\,*}_\HH$ are isomorphic, and  two simplicial CDBA $A^{\star\,*}_\bullet(M)$ and $B^{\star\,*}_\bullet(\HH)$ are isomorphic. In particular, the $E_2$-page of $\CECHC$ech s.s. is isomorphic to the total homology of the normalization $NB^{\star\,*}_\bullet(\HH)$ as a bigraded $\kk $-module. The bigrading is given by $(\star-\bullet, *)$
\end{enumerate}

\end{thm}
\begin{proof}
The part 1 obviously follows from Theorem \ref{Tam} and Lemma \ref{Lahg}. We shall prove the the part 2.  We define a map $\Phi_n:B_\HH(n)\to A_M(n)$ of algebras by identifying the subalgebra $\SHH^{\otimes n}$ and elements $g_{ij}$ in the both sides and taking $h_{ij}$ to $\bar h_{ij}$ (see the paragraph above Lemma \ref{Ladmissiblegenerator}). We easily verify $\Phi_n$ is well-defined. $\Phi_n$ fits into the following commutative diagram. 
\[
\xymatrix{
\underset{G\in \GG(n)^{dis}}{\bigoplus} H_Gg_G \ar[dr]\ar[d]& \\
B_\HH(n)\ar[r]^{\Phi_n} & A_M(n)}
\]
Here the vertical arrow is induced by the inclusion of a submodule $H_Gg_G=B_{\HH,G}g_G\subset B_{\HH}(n)$ given by  the isomorphism $\varphi_G$ in Lemma \ref{Lbhg} and the module structure defined above, and the slanting arrow is given in Lemma \ref{Lgraphincl}. The vertical arrow and $\Phi_n$ are  epimorphisms by Lemma \ref{Ladmissiblegenerator} and the slanting arrow is an isomorphism by Lemma \ref{Lgraphincl} so $\Phi_n$ is an isomorphism. By the definition of $\Phi_n$ and Lemma \ref{Lbhg}, the collection $\{\Phi_n\}_n$ commutes with the structures of $\ASS$-comodule and degeneracy maps. The assertion for the $E_2$-page immediately follows from the isomorphism of simplicial objects.
\end{proof}
\begin{rem}\label{Ralgebraicss}
The Euler number $\chi(M)$ can be recovered from the Poincar\'e algebra $\HH^*=H^*(M)$. It is the image of $\Delta_\HH$ by the composition 
\[
(\HH^{\otimes 2})^{*=\dd} \stackrel{\text{multiplication}}{\longrightarrow} \HH^{\dd}\stackrel{\eps}{\longrightarrow} \kk
\]
So under the assumption of Theorem \ref{Talgebraicss}, the $E_2$-page of $\CECHC$ech s.s. is determined by the cohomology algebra $H^*(M)$. (Different orientations give apparantly different presentations but they are isomorphic.)
\end{rem}
\section{Examples}\label{Sexample}
In this section, we compute some part of $E_2$-page of $\CECHC$ech s.s. for spheres and products of two spheres $S^k\times S^l$ for $(k,l)=(\text{odd},\text{even})$ or $(\text{even},\text{even})$ and deduce some results on cohomology groups for the products of spheres. We also prove Corollary \ref{Tmain4dim}. Our computation is restricted to low degrees and consists of only  elementary linear algebra on differentials and degree argument based on Theorem \ref{Talgebraicss}. We briefly state the results for the case of spheres since in these cases, the $\CECHC$ech s.s. only gives less information than the combination of Vassiliev's (or Sinha's) spectral sequence for long knots and the Serre spectral sequence for a fibration $\Emb(S^1,S^{\dd})\to STS^{\dd}$ (see the proof of Proposition \ref{Pnotcollapse}) gives, at least in the  degrees which we have computed. We give  concrete descriptions of the differentials in the case of $M=S^k\times S^l$ with $k$ odd and $l$ even. In the rest of this section, we set $\HH=H^*(M)$ for a fixed orientation.  
\subsection{The case of $M=S^{\dd}$ with  $\dd$ odd}
In this case $A^{\star\,*}_\bullet(\HH)$ is described as follows.
\[
A_n^{\star\,*}(\HH)=\bigwedge\{x_i,\  y_i,\ g_{ij}\ \mid 0\leq i,j\leq n \}/\mathcal{I}
\]
where  $|x_i|=(0,\dd),\ |y_i|=(0,\dd-1),\ |g_{ij}|=(-1,\dd)$, and $\mathcal{I}$ is the ideal generated by 
\[
\begin{split}
(x_i)^2,\ (y_i)^2, &\  (g_{ij})^2, \ g_{ii},  \ g_{ij}+g_{ji},\ \\
(x_i-x_j)g_{ij},\ & \text{ and the 3-term relation for }g_{ij}. 
\end{split}
\] 
The diagonal class is given by $\Delta_\HH=x_0-x_1\in \HH\otimes \HH$. 

\begin{prop}\label{Toddsphere}
Consider the $\check{C}$ech s.s. $\BGSSS^{\,p\,q}_r$ for the sphere $S^{\dd}$ with odd $\dd\geq 5$. We abbreviate $\BGSSS^{\,p\,q}_2$ as $(p,q)$. The following equalities hold.
\[
\begin{split}
(-3,\dd)& =\kk\langle \ g_{12}\ \rangle, \quad 
(-1,\dd-1)=\kk\langle \ y_1\ \rangle,     \\
(0,\dd-1) & =\kk\langle \ y_0\ \rangle, \quad      (0,\dd)=\kk\langle \ x_0\ \rangle, \\ 
(-6,2\dd) & =\kk\langle g_{13}g_{24},\ -g_{12}g_{34}+g_{14}g_{23}\rangle, \quad 
 (-4,2\dd-1)=\kk\langle \ y_1g_{23}-y_2g_{13}+y_3g_{12} \rangle, \\
  (-5,2\dd) & =\kk\langle g_{01}g_{23}+g_{02}g_{13}+g_{13}g_{23}\rangle, \quad 
 (-3,2\dd-1)=\kk\ \langle \ y_0g_{12}\rangle \\
(-3,2\dd) &=\kk\langle x_0g_{12}\rangle,  \quad    
(-1,2\dd-1) = \kk\langle \  x_0y_1,\ x_1y_0,\ x_1y_1\ \rangle, \\
(0,2\dd-1) &  =\kk\ \langle \ x_0y_0\ \rangle 
\end{split}
\]
For other $(p,q)$ with $p+q\leq 2\dd -1$, we have $(p,q)=0$. {\hspace{\fill}  \qedsymbol}
\end{prop}

\begin{prop}\label{Pnotcollapse}
Let $\dd$ be an odd number with $\dd\geq 5$.\\
(1) $\Emb (S^1,S^{\dd})$ is $\dd-2$-connected.\\
(2) The $\check{C}$ech s.s. for $S^{\dd}$ does not collapse at $E_2$-page in any coefficient ring.
\end{prop}
\begin{proof}
For the part 1, consider the fiber sequence
\[
\Emb_c(\RR,\RR^{\dd})\to \Emb (S^1,S^{\dd})\to STS^{\dd}
\]
where the left hand side map is given by taking the tangent vector at a fixed point and the right  space is the space of long knots. As is well known, $STS^{\dd}$ is $\dd-2$-connected and $\Emb_c(\RR,\RR^{\dd})$ is $2\dd -7$-connected. As $\dd\geq 5$, we have the claim. The part 2 follows from the part 1 and Proposition \ref{Toddsphere}. (There is non-zero elements in total degrees $\dd-3$ and $\dd-2$.) 
\end{proof}
\begin{rem}\label{Rbudney}
The reader may find inconsistency between \cite[Proposition 3.9 (3)]{budney} and Proposition \ref{Pnotcollapse} (1). This is just a notational matter. $n-j-2$ should be replaced with $n-j-1$ (and $n-j-1$ with $n-j$) in the proposition, see its proof.  
\end{rem}

\subsection{The case of $M=S^{\dd}$ with $\dd$ even}
In this subsection, we assume $2\in \kk^{\times }$.    $B^{\star\,*}_\bullet(\HH)$ is described as follows.
\[
B_n^{\star\,*}(\HH)=\bigwedge\{z_i,\ g_{ij},\ h_{ij}\ \mid 0\leq i,j\leq n \}/\mathcal{J}
\]
where  $|z_i|=(0,2\dd-1),\ |g_{ij}|=(-1,d),\ |h_{ij}|=(-1,2\dd-1)$, and $\mathcal{J}$ is the ideal generated by 
\[
\begin{split}
 (z_i)^2,   (g_{ij})^2,\ (h_{ij})^2,\ g_{ii},\ & h_{ii}, \ 
 g_{ij}-g_{ji},\ h_{ij}+h_{ji},  \\
 (z_i-z_j)g_{ij},\  (z_i-z_j)h_{ij},\ & (h_{ij}+h_{ki})g_{jk}-(h_{ij}+h_{jk})g_{ki},\\ 
 \text{ and the 3-term relation for }&  g_{ij}\text{ and }h_{ij}. 
\end{split}
\] 
The diagonal classes are given by $\Delta_\HH=0\in S\HH\otimes S\HH$ and $\Delta_{S\HH}=z_0-z_1\in S\HH\otimes S\HH$. 
\begin{prop}\label{Tevensphere}
Suppose $2\in \kk^\times$ . Consider the $\check{C}$ech s.s. $\BGSSS^{\,p\,q}_r$ for  $S^{\dd}$ with even $\dd\geq 4$. We abbreviate $\BGSSS^{\,p\, q}_2$ as $(p,q)$. 
 The following equalities hold.
\[
\begin{split}
(-6,2\dd)=\kk \langle g_{13}g_{24}\rangle, & \quad (-5,2\dd)=\kk\langle g_{01}g_{23}+3g_{02}g_{13}+g_{03}g_{12}\rangle , \\
(-3,2\dd-1)=\kk\langle h_{12}\rangle, &\quad  (0,2\dd-1)=\kk\langle z_0\rangle \,.
\end{split}
\]
For other $(p,q)$ with $p+q\leq 2\dd-1$, we have $(p,q)=0$.\hspace{\fill} \qedsymbol
\end{prop}
For the case of  $\kk=\FF_2$, the same statement as Proposition \ref{Toddsphere} holds, except that " odd $\dd\geq 5$" is replaced with " even $\dd\geq 4$".
\subsection{The case of $M=S^k\times S^l$ with $k$ odd and $l$ even}
 We fix  generators $a\in H^k(S^k)$ and $b\in H^l(S^l)$. $\HH$ is presented as $\wedge \{a,\ b\}$. We fix an orientation $\epsilon$ on $\HH$ by $\epsilon (ab)=1$. We write $a_i$ for $e_i(a)$ (similarly for $e_i(b)$) and  $A_n(\HH)$ is presented as
\[
A_n(\HH)=\bigwedge\{a_i,\ b_i,\ y_i,\ g_{ij}\ \mid 0\leq i,j\leq n \}/\mathcal{I}
\]
where  $|y_i|=(0,k+l-1),\ |g_{ij}|=(-1,k+l)$, and $\mathcal{I}$ is the ideal generated by 
\[
\begin{split}
(a_i)^2,\ (b_i)^2,\ (y_i)^2,\  & (g_{ij})^2, \ g_{ii},  \ g_{ij}+g_{ji},\ \\
(a_i-a_j)g_{ij},\ (b_i &-b_j)g_{ij}\text{ and the 3-term relation for }g_{ij}. 
\end{split}
\] 
The diagonal class is given by $\Delta_\HH=a_0b_0-a_1b_0+a_0b_1-a_1b_1\in \HH\otimes \HH$. 
The module $NA_n(\HH)$ is generated by the monomials of the form $
a_{p_1}\cdots a_{p_s}b_{q_1}\cdots b_{q_s}g_{i_1j_1}\cdots g_{i_rj_r}$ such that the set of subscripts $\{p_1,\dots p_s,q_1,\dots,q_t,i_1,\dots, i_r,j_1,\dots, j_r\}$ contains the set $\{1,\dots, n\}$. \\

\indent We shall present the total differential  $\td$ on 
\[
\BGSSS^{\,p\,q}_1=\underset{\star-\bullet=p}\bigoplus NA_\bullet^{\star, q}(\HH)
\]  up to $p+q\leq \max\{2k+l,k+2l\}$.
For $(p,q)=(-1,k), \ (-1,l),\ (-1,k+l-1),\ (-1,k+l), \ (-1,2k),\ (-1,2l),\ (-1,2k+l),\ (-1,k+2l),\ (-1,2k+l-1),\ (-1,k+2l-1),\ (-2,2k),\ (-2,2l),\ (-2,3k),\ (-2, 3l)$, $\td$ is zero.  \\

For $(p,q)=(-3, k+l)$, $\td$ is presented by the following matrix
\[
\begin{array}{c|c}
     &    g_{12}    \\
\hline
g_{01}  & 0   \\
a_1b_2  &  1    \\
a_2b_1  & -1     
\end{array}.
\]
This is read as $\td (g_{12})=a_1b_2-a_2b_1$. For $(p,q)=(-2, k+l)$, 
\[
\begin{array}{c|ccc}
  & g_{01}  & a_1b_2  & a_2b_1 \\
  \hline
 a_0b_1 &  1 &  1 &   1  \\
 a_1b_0 & -1 &   1 &   1  \\
 a_1b_1 & -1 &  -1 &  -1
 \end{array}.
\]
For $(p,q)=(-4,2k+l)$, 
\[
\begin{array}{c|ccc}
& a_1g_{23} &    a_2g_{13}   &  a_3g_{12}   \\
\hline
a_0g_{12}  &  1 &  0     & -1     \\
a_1g_{02}   & 1  & 1  &    0      \\
a_1g_{12}   & -1 & 0  & 1     \\
a_2g_{01} & 0  & 1  & 1   \\
a_1a_2b_3  &  -1  &  1 & 0  \\
a_1a_3b_2   & 1  &  0  &  1  \\
a_2a_3b_1  & 0 & 1  &  -1  
\end{array}.
\]
For $(p,q)=(-3,2k+l)$, 
\[
\begin{array}{c|ccccccc}
 &    a_0g_{12}     &     a_1g_{02}   &     a_1g_{12}  &   a_2g_{01}  &     a_1a_2b_3   &    a_1a_3b_2   &     a_2a_3b_1   \\
\hline
a_0g_{01} & 0   &0      &0     &0    &0   &0   & 0  \\
 a_0a_1b_2   &-1  & 1 &   0  &   0  &  -1 & -1 &  0  \\
 a_0a_2b_1  & 1 & 0  & 0  & 1  & 0  & -1 & -1  \\
 a_1a_2b_0  &  0&   1& 0  &-1 & 1 &0  &-1 \\
a_1a_2b_1   & 0 & 0  &  1 &  -1 &  0 & 1  & 1  \\
a_1a_2b_2   & 0   & 1 & 1 & 0 & -1 & -1 & 0  
\end{array}.
\]

For $(p,q)=(-2,2k+l)$, 
\[
\begin{array}{c|cccccc}
           &  a_0g_{01}    &   a_0a_1b_2  &   a_0a_2b_1  &   a_1a_2b_0  &    a_1a_2b_1 &  a_1a_2b_2   \\
\hline
a_0a_1b_0 & 1   &  -1   & -1   &   0    &   -1   & 1 \\
a_0a_1b_1  &  1  &   1 &   1  &  0  &   1&    -1
\end{array}.
\]
For $(p,q)=(-2,2k+l-1)$, 
\[
\begin{array}{c|cc}
 & a_1y_2 &   a_2y_1  \\
 \hline
a_0y_1  & 1   &   1 \\
a_1y_0  & 1   &   1 \\
a_1y_1  & -1  & -1  
\end{array}.
\]
For $(p,q)=(-4,k+2l)$, 
\[
\begin{array}{c|ccc}
  &  b_1g_{23}  &    b_2g_{13}  &    b_3g_{12}      \\
  \hline
b_0g_{12} & -1 & 0 &  1   \\
b_1g_{02}&-1& -1& 0 \\
b_1g_{12}& 1 & 0 &  -1  \\
b_2g_{01}& 0 & -1 &  -1 \\
a_1b_2b_3 & 0  & 1 &  1  \\ 
a_2b_1b_3& 1 & 0 & -1   \\
a_3b_1b_2& -1 &-1&  0
\end{array}.
\]

For $(p,q)=(-3,k+2l)$, 
\[
\begin{array}{c|ccccccc}
      &   b_0g_{12}  & b_1g_{02}   &  b_1g_{12}  &   b_2g_{01}  &   a_1b_2b_3 &    a_2b_1b_3  &  a_3b_1b_2    \\
\hline 

b_0g_{01}  & 0  & 0 & 0 & 0  & 0&  0&  0   \\
a_0b_1b_2 &0  &1   &0  &1  &1  &0   &-1  \\
a_1b_0b_2   & 1&  0   &0&-1  &-1 &1 &0   \\

a_1b_1b_2  &  0& 0  &1  &-1&  -1&  -1&     0  \\
a_2b_0b_1 & -1 &  -1 &  0 &  0  & 0  &  -1&   1\\
a_2b_1b_2   & 0&  -1&  -1&  0&   0&  1&   1   
\end{array} .  
\]
For $(p,q)=(-2,k+2l)$, 
\[
\begin{array}{c|cccccc}
    &   b_0g_{01}  &    a_0b_1b_2     &    a_1b_0b_2   &  a_1b_1b_2    &    a_2b_0b_1  &  a_2b_1b_2   \\
 \hline
a_0b_0b_1   &    1 &  2 &  1   &  1  &   1   &  1  \\
a_1b_0b_1  &  -1 &  0    &-1     & 1     & -1 &  1
\end{array}.
\]
For $(p,q)=(-2,k+2l-1)$, 
\[
\begin{array}{c|cc}
  & b_1y_2  &  b_2y_1  \\
  \hline
b_0y_1 & -1 &-1 \\
b_1y_0 & 1 & 1 \\
b_1y_1 & 1 & 1    
\end{array}.
\]
By direct computation based on the above presentation, we obtain the following result. Let $\kk_2$ (resp. $\kk^2$) denote the module $\kk/2\kk$ (resp. $\kk\oplus \kk$).
\begin{prop}\label{Pbgsss}
Suppose $\kk$ is either of $\ZZ$ or $\FF_{\pp}$ with $\pp$ prime. Let $k$ be an odd number and $l$ be an even numbers with $k+5\leq l\leq 2k-3$ and $|3k-2l|\geq 2$,  or $l+5\leq k \leq 2l-3$ and $|3l-2k|\geq 2$.  We abbreviate   $\BGSSS^{\,p\,q}_2$  for $S^k\times S^l$ as $(p,q)$.
We have the following isomorphisms.
\[
\begin{split}
(0,k) &  = (-1,k)= (0,l)= (-1,l)= (-1,2k)= (-2,2k) = (-1,2l)= (-2,2l)=\kk \\   
        (-2,3k) &  = (-3,3k)= (-2,3l)= (-3,3l)= (0,k+l-1)= (-1,k+l-1)= \kk , \\
     (0,    &  \, k+l)   = \kk ,     \quad (-1,k+l)= \kk \oplus \kk _2 \text{ or } \kk ^2,\quad (-2, k+l)= 0\text{ or } \kk , \\
(0,   & \,  2k+l-1)      = \kk ,  \quad (-1,2k+l-1)= \kk ^2, \quad (-2,2k+l-1)= \kk ,   \\
(-1, & \, 2k+l)= \kk _2 \text{ or } \kk ,  \quad (-2, \,2k+l)= \kk _2 \text{ or } \kk ^2, \quad (-3,2k+l)= \kk _2\text{ or }\kk ^2,\\
 (-4, &  \,2k+l)= 0 \text{ or }\kk  \\
(0,   & \,  k+2l-1)      = \kk ,  \quad (-1,k+2l-1)= \kk ^2, \quad (-2,k+2l-1)= \kk ,   \\
(-1, & \,k+2l)= \kk _2 \text{ or } \kk ,  \quad (-2, k+2l)= \kk  \text{ or } \kk ^2, \quad (-3,k+2l)= \kk ^2, \quad (-4 , 2k+l)= \kk  \\
\end{split}
\] 
Here "$(p,q)= A\text{ or }B$" means $(p,q)= A$ if $\kk =\ZZ\  or\  \FF_{\pp}$ with $\pp \not=2$, $(p,q)= B$ if $\kk =\FF_2$. For other $(p,q)$ with $p+q\leq  \max \{k+2l, 2k+l\}$,  $(p,q)=0$.\hspace{\fill} \qedsymbol
\end{prop}
The isomorphisms of proposition \ref{Pbgsss} holds under  milder conditions on $k$ and $l$. It suffices to ensure the bidegrees presented above are pairwise distinct. 
By degree argument, we obtain the following corollary.
\begin{cor}\label{Coddevensphere}
Suppose $\kk$ is either of $\ZZ$ or $\FF_{\pp}$ with $\pp$ prime. Let $k$ be an odd number and $l$ be an even number with $k+5\leq l\leq 2k-3$ and $|3k-2l|\geq 2$,  or $l+5\leq k \leq 2l-3$ and $|3l-2k|\geq 2$. 
We set $H^*= H^*(\Emb (S^1,S^k\times S^l))$.
\begin{enumerate}
\item We have   isomorphisms
\[
H^{i} = \kk  \quad (i=k-1,k,\ 2k-2,\ 2k-1,\ k+l). 
\]
\item If $\kk =\FF_{\pp}$ with $\pp\not =2$, we have isomorphisms
\[ 
H^i= \kk ^2 \ (i=k+l-2,\ k+l-1,\ 2k+l-3,\ 2k+l-2,\ 2k+l-1).
\]
\end{enumerate}
\end{cor}
\begin{proof}
By an argument similar to the proof of Theorem \ref{Tconvergence}, we see $\BGSSS_2^{-p,q}=0$ if $q/p<(k+l)/3$. We shall show any differential $d_r:\BGSSS_r^{(-p-r,q+r-1)}\to \BGSSS_r^{-p,q}$ going into the term contained in the cohomology of the claim is zero. It is enough to the case of $(-p,q)=(0,2k+l-1)$ and $q+r-1\geq k+2l-1$ since other cases are obvious or follow from this case. We see 
\[
\frac{q+r-1}{p+r}=\frac{q-1}{r}+1\leq \frac{2k+l-2}{l-k+1}+1=\frac{k+2l-1}{l-k+1}<\frac{k+l}{3}.
\]  
So $\EE^{(-p-r,q+r-1)}_r=0$ and $d_r=0$.
\end{proof}
\subsection{The case of $M=S^k\times S^l$ with $k,l$ even}
 We fix  generators $a\in H^k(S^k)$ and $b\in H^l(S^l)$. $\HH$ is presented as $\wedge \{a,\ b\}$. We fix an orientation $\epsilon$ on $\HH$ by $\epsilon (ab)=1$. We set $c=\bar a \in S\HH$, $d=\bar b \in S\HH$. We write $a_i$ for $e_i(a)$ (similarly for $e_i(b)$ etc.) and  $B_n(\HH)$ is presented as
\[
B_n(\HH)=\bigwedge\{a_i,\ b_i,\ c_i,\ d_i,\  g_{ij},\ h_{ij} \ \mid 0\leq i,j\leq n \}/\mathcal{J}
\]
where  $|g_{ij}|=(-1,k+l),\ |h_{ij}|=(-1,2(k+l)-1)$, and $\mathcal{J}$ is the ideal generated by 
\[
\begin{split}
(a_i)^2,\ (b_i)^2,\ (c_i)^2,\ (d_i)^2,\     a_ib_i,\  &   a_ic_i,\ b_id_i,\ c_id_i,  \ a_id_i-b_ic_i  \vspace{1mm}\\
       (g_{ij})^2, \ (h_{ij})^2,\ g_{ii},\ h_{ii},       &        \ g_{ij}-g_{ji},\ h_{ij}+h_{ji},\\
(a_i-a_j)g_{ij},\ (b_i       -b_j)g_{ij},\ (c_i-c_j)g_{ij}    &      -a_ih_{ij},\  (d_i-d_j)g_{ij}-b_ih_{ij} \\
(a_i-a_j)h_{ij},\ (b_i-b_j)h_{ij}, \      &         (c_i-c_j)h_{ij},\ (d_i-d_j)h_{ij}  \\
(h_{ij}+h_{ik})g_{jk}-(h_{ij}+h_{jk})g_{ki},\      &     \text{ and the 3-term relations for }g_{ij}\  \text{and}\  h_{ij}. 
\end{split}
\] 
The diagonal classes are  given by $\Delta_\HH=a_0b_1+a_1b_0\in S\HH\otimes S\HH$ and $\Delta_{S\HH}=a_0d_0+a_1d_0+b_1c_0-b_0c_1-a_0d_1-a_1d_1$\\

By an argument similar to the proof of Corollary. \ref{Coddevensphere}, we obtain the following corollary.
\begin{cor}\label{Ceveneven}
Suppose $2\in \kk^{\times} $. Let $k$ and  $l$ be two  even numbers with $k+2\leq l\leq 2k-2$ and $|3k-2l|\geq 2$. 
We set $H^*= H^*(\Emb (S^1,S^k\times S^l))$.
 We have isomorphisms
\[
H^{i} =
\kk  \quad (i=k-1,k,l-1,l,k+l-3,k+l-2,k+l-1, 3k).  
\]
For any other degree $i\leq 2k+l$, $H^i=0$. \hfill \qedsymbol
\end{cor}

\subsection{The case of $4$-dimensional manifolds}\label{SS4dim}
In this subsection, we prove Corollary \ref{Tmain4dim}. We  assume that $M$ is a simply connected  $4$-dimensional manifold. So, as is easily observed,  $\HH$ is a free $\kk$-module for any $\kk$. 
\begin{defi}\label{D4dim}
Set $\chi=\chi(M)$.
\begin{itemize}

\item 
We define a map $\alpha:(\HH^2)^{\otimes 2}\oplus \kk g_{01} \to (\HH^2)^{\otimes 2}\oplus \HH^4/\chi\HH^4$ by
\[
\alpha (a\otimes b)=(-a\otimes b-b\otimes a )+ ab, \quad \alpha( g_{01})=pr_1(\Delta_\HH)
\]
Here, $g_{01}$ is a formal free generator (which will correspond to the element of the same symbol in $\BGSSS^{-2,4}_1$) and $pr_1$ is the  projection 
\[
(\HH^{\otimes 2})^{*=4}\to (\HH^2)^{\otimes 2}\oplus (1\otimes \HH^4)\to (\HH^2)^{\otimes 2}\oplus  \HH^4/\chi \HH^4.
\]

\end{itemize}

\end{defi}
 The following proposition follows from direct computation and degree argument based on Theorem \ref{Talgebraicss}.
 
\begin{lem}\label{L4dim}
We use the notations in Definition \ref{D4dim}. Suppose $\kk$ is a field  and $\HH^2$ is not zero.
\begin{enumerate}
\item When $p+q=1$, $\BGSSS^{p,q}_r$ is stationary after $E_2$. In particular, $\BGSSS_2^{p,q}\cong \BGSSS_\infty^{p,q}$. We have isomorphisms
\[\BGSSS^{p,q}_2\cong \left\{
\begin{array}{ll}
\HH^2 & ((p,q)=(-1,2))\\
0 & (\text{otherwise})
\end{array}\right. .\]
\item  There exists an isomorphism
\[
\BGSSS_2^{\,-2,\,4}\cong \mathrm{Ker}(\alpha)/\kk( pr_2(\Delta_\HH)+ 2g_{01}) .
\]
Here $pr_2$ is the projection 
$
(\HH^{\otimes 2})^{*=4}\to (\HH^2)^{\otimes 2}
$ . The differential $d_r$ coming into this term is zero for $r\geq 2$. \hfill \qedsymbol
\end{enumerate} 
\end{lem}
\begin{rem}\label{R4dimE2}
Actually, Lemma \ref{L4dim} holds even when $\kk$ is a not a field since the torsions for K\"unneth theorem does not affect the range.  
\end{rem}
\begin{proof}[Proof of Corollary \ref{Tmain4dim}]
In this proof, we suppose $\kk$ is a field.  Set $H_2^\ZZ=H_2(M;\ZZ)$.
As is well-known, there is a weak homotopy equivalence between $\Imm (S^1,M)$ and  the free loop space $L\TM$, and there is an isomorphism $\pi_1(L\TM)\cong \pi_1(\TM)\oplus \pi_2(\TM)$. As $M$ is simply connected, we have $\pi_1\Imm (S^1,M)\cong \pi_2(\TM)\cong \pi_2(M)\cong H_2^\ZZ$. \\
\indent  By Goodwillie-Klein-Weiss convergence theorem, the connectivity of the standard projection $\underset{\BDelta}{\holim}\  \CC^\bullet\CPTM\to \underset{\BDelta_n}{\holim}\ \CC^\bullet\CPTM$ increases as $n$ increases. Since $\Delta_n$ is a compact category in the sense of \cite{farjoun}, and $\CC^n\CPTM$ is simply connected for any $n$, by  \cite[Theorem 2.2]{farjoun}, we see $\Emb(S^1,M)$ is $\ZZ$-complete. In particular $\pi_1(\Emb(S^1,M))$ is a pro-nilpotent group. So, by a theorem of Stallings \cite{stalling}, we only have to prove the composition
\[
\Emb(S^1,M)\stackrel{i_M}{\longrightarrow }\Imm(S^1,M)\stackrel{\simeq}{\longrightarrow} L\TM\stackrel{cl_1}{\longrightarrow} K(H_2^\ZZ, 1)
\]
induces an isomorphism on $H_1(-;\ZZ)$ and a surjection on $H_2(-;\ZZ)$. Here the rightmost map $cl_1$ is the classifying map. See \cite{gersten}.\\
\indent Consider the spectral sequence $E^{p,q}_r$ associated to the Hochschild complex of $C^S_*(\TCCM)$. This spectral sequence is isomorphic to the Bousfield-Kan type cohomology spectral sequence associated to the well-known cosimplicial model for $L\TM$ given by $[n]\mapsto \TM^{\times n+1}$. The quotient map $\TCCM\to \CCM$ induces a map $f_r: E^{p\, q}_r\to \BGSSS^{p\, q}_r$ of spectral sequences. For $r=\infty$, this map is identified with the map on the associated graded induced by the inclusion $i_M$. For $p+q=1$, by Lemma \ref{L4dim} (and similar computation for $E^{p\, q}_r$), $f_2$ is an isomorphism for  any field $\kk$. Since $\pi_1(\Emb(S^1,M))$ is the same as $\pi_1$ of a finite stage of Taylor tower which is finite homotopy limit of simply connected finite cell complex,  it is finitely genrated, and so is $H_1$. By the universal coefficient theorem, we see $i_M$ induces an isomorphism on $H_1(-;\ZZ)$.  For the part of   $p+q=2$,  we see $E^{p\,q}_2=0$ for $p<-2$ and $E^{\,-2,\,4}\cong \mathrm{Ker}(\alpha)\cap(\HH^2)^{\otimes 2}$.
Consider the following zigzag
\[
L\TM\stackrel{L(cl_2)}{\longrightarrow} LK(H_2^\ZZ,2)\stackrel{i_K}{\longleftarrow} \Omega K(H_2^\ZZ,2),
\] where the left map is induced by the classifying map $cl_2:\TM\to K(H_2^\ZZ,2)$ and the right one is the inclusion from the based loop space. Clearly, the composition $cl_1\circ i_K:\Omega K(H_2^\ZZ, 2)\to K(H_2^\ZZ, 1)$ is a weak homotopy equivalence. Observe spectral sequences associated to the standard cosimplicial models of the above three spaces. Since the maps $L(cl_2)$ and $i_K$ are induced by  cosimplicial maps, they induce maps on spectral sequences. In the part of total degree $2$, we see the filter $F^{-2}$ for each of three spectral sequences is the entire cohomology group, and the filter $F^{-1}$ for the one for $\Omega K(H_2^{\ZZ}, 2)$ is zero. With these  observations,  we see that the image of $H^2(K(H_2^\ZZ,1))$ in $H^2(L\TM)$ by the  map $cl_1$ is sent to a subspace $V$ of $F^{-2}/F^{-1}\cong E_\infty^{\,-2,\,4}\subset E_2^{\,-2,\,4}$ isomorphically, and  a basis of $V$  is given by $\{\,a_i\otimes a_j-a_j\otimes a_i \mid i<j\,\}$ as elements of $E_2^{-2,4}$, where $\{a_i\}_i$ denotes a basis of $\HH^2$. (We also see these elements must be stationary.) If  $\kk\not=\FF_2$, or if $\kk=\FF_2$ and the inverse of the intersection matrix has at least one non-zero diagonal component, we see the restriction of  $f_2$ to $V$ is a monomorphism by Lemmas \ref{Lpoincarediagonal} \ref{L4dim}. (Otherwise, the elements of the basis of $V$ have the relation $pr_0(\Delta_\HH)=0$.) This implies $i_M$ induces a surjection on $H_2$ for any field $\kk$ under the assumption of the theorem.    By the universal coefficient theorem, we obtain the desired assertion on $H_2(-;\ZZ)$.  
\end{proof}
\begin{rem}\label{R4dim}
If all of the diagonal components of the inverse of intersection matrix on $H_2(M;\FF_2)$ is zero, the map $f_2:V\to \BGSSS_2^{-2,4}$ in the proof is not a monomorphism for $\kk=\FF_2$ but this does not necessarily imply the original (non-associated graded) map is not a monomorphism. So in this case, it is still unclear whether $i_M$ is an isomorphism on $\pi_1$. 
\end{rem}
\section{Precise statement and proof of Theorem \ref{Tmainhomotopycolimit}}\label{Shomotopycolim}
\begin{defi}\label{Dfunctor}
\begin{itemize}
\item Fix a coordinate plane with coordinate $(x,y)$. A {\em planer rooted $n$-tree} $(T,\mathfrak{e})$ consists a $1$-dimensional finite cell complex $T$ and a continuous monomorphism  $\mathfrak{e}$ from its realization $|T|$  to  the half plane $y\geq 0$ such that 
\begin{itemize}
\item $T$ is connected and $\pi_1(T)$ is trivial.
\item the intersection of the image of $\mathfrak{e}$ and the $x$-axis consisting of the image of  $n$ univalent vertices called {\em leaves} and labeled by $1,\dots, n$ in the manner consistent with the standard order on the axis,
\item $T$ has a unique distinguished vertex called the {\em  root} which is at least bivalent, and 
\item any  vertex except for the leaves and root is at least trivalent. 
\end{itemize}
An {\em isotopy} between $n$-trees $(T_1,\mathfrak{e}_1) \to (T_2,\mathfrak{e}_2)$ is an isotopy  of the half plane onto itself which maps $\mathfrak{e}_1(|T_1|)$  onto $\mathfrak{e}_2(|T_2|)$ and the root to the root. (So an isotopy preserves the leaves including the labels.)  We  will denote an isotopy class of planer rooted $n$-trees simply by $T$.  The root vertex of a tree is usually denoted by $v_r$. For a vertex $v$ of a tree, $|v|$ denotes the number which is the valence minus $1$ if  $v\not=v_r$, and equal to the valence if $v=v_r$.   
\item Let $\Psi_n^o$ be a category defined as follows. An object of $\Psi_n^o$ is an isotopy class of the  planer  rooted $n$-trees.  There is a unique morphism $T\to T'$ if $T'$ is obtained from $T$ by a successive contraction of internal edges (i.e. edges not adjacent to leaves). 
\item Let $\CAT$ be the category of small categories and functors.  Let $i_n:\Psi_n^o\to \Psi_{n+1}^o$ be a functor which sends $T$ to the tree made from $T$ by attaching  two edges to  the $n$-leaf of $T$ and labeling the new leaves with $n$ and $n+1$.  We define a category $\Psi^o$ to be the colimit of the sequence $\Psi_1^o\stackrel{i_1}{\to} \Psi^o_{2}\stackrel{i_2}{\to} \cdots$ taken in $\CAT$.  $\mathcal{F}_n:\Psi^o_{n+1}\to P_\nu(\underline{n})$ denotes the functor given in Definition 4.14 of \cite{sinha}, which sends a tree $T\in \Psi_{n+1}^o$  to the set of the numbers $i$ such that the shortest paths from  $i$ and $i+1$ to the root in $T$ intersects only at the root.   For the functor $\mathcal{G}_{n}:P_{\nu}(\underline{n+1})\to \BDelta_n$, see subsection \ref{SSnt}. The following square is clearly commutative.
\[\xymatrix{\Psi^o_{n+2}\ar[d]^{i_n}\ar[r]^{\mathcal{F}_n\circ \mathcal{G}_{n+1}}  & \BDelta_n\ar[d]^{i_n} \\
\Psi^o_{n+3}\ar[r]^{\mathcal{F}_{n+1}\circ \mathcal{G}_{n+2}}  & \BDelta_{n+1},}
\]
where the right vertical arrow is the natural inclusion,  so we have the induced functor $\mathcal{F}\circ\mathcal{G}:\Psi^o\to \Delta$. 
\item In the rest of the paper, for a symmetric sequence $X$ and a vertex $v$ of a tree in $\Psi^o$, $X(v)$, $X(v-1)$ and $\underline{v-1}$ denote $X(\,|v|\,)$, $X(\,|v|-1)$  and $\underline{|v|-1}$ respectively. ($|v|$ is the number of the 'out going edges')
\item For a  $\KK$-comodule $X$ in $\SP$, We shall define a functor $\FPsi^n X:(\Psi_{n+2}^o)^{op}\to \SP$. The definition is similar to (a dual of) the construction of $\DD_n[M]$ in Definition 5.6 of \cite{sinha}. For a tree $T\in \Psi_{n+2}$, define a space $\KK_T^{nr}$ by
\[
\KK_T^{nr}=\underset{v}{\prod}\,\KK(v)
\]
Here, $v$ runs through all the non-root and non-leaf vertices of $T$. This is denoted by $K^{nr}_T$ in \cite{sinha}.
We set 
\[
\FPsi^n X(T)= \Map(\KK_T^{nr}, X(v_r-1)).
\]  For a morphism $T\to T'$ given by the contraction of a non-root edge $e$ (an edge not adjacent to the root), the map $e^*:\FPsi^n X(T')\to \FPsi^n X(T)$ is pull-back by the inclusion $\KK_T^{nr}\to \KK_{T'}^{nr}$ to a face corresponding to the edge contraction (see Definition 4.26 of \cite{sinha}). For the $i$-th root edge $e$, the corresponding map is given by the following composition.
\[
\begin{split}
\Map \Biggl(\,\prod_{\tiny  v\in T'   } \KK(v),\ X(v_r'-1)\Biggr) &  
 = \Map\Biggl(\prod_{\tiny \begin{array}{c} v\in T \\  v\not= v_t\end{array}}\!\!\!\! \KK(v ),\  X(v_r'-1)\Biggr) \\
    &  \to \Map\Biggl(\prod_{\tiny \begin{array}{c} v\in T\\  v\not= v_t\end{array}}\!\!\!\! \KK (v),\  \Map \Bigl( \KK(v_t ),\ X(v_r-1)\Bigr)\Biggr)\\
    &  \cong \Map \Biggl(\Bigl(\prod_{\tiny \begin{array}{c} v\in T\\ v\not= v_t\end{array}}\!\!\!\! \KK(v )\Bigr) \times \KK(v_t ),\ X(v_r-1)\Biggr)\\
    &  = \Map \Biggl(\,\prod_{\tiny  v\in T} \KK(v ),\ X(v_r-1)\Biggr).
\end{split}
\] 
Here $v_t$ is the vertex of $e$ which is not the root. For $1\leq i\leq |v_r|-1$, the arrow in the second line  is the pushforward by the adjoint of the partial composition $(-\circ_i-): \KK(v_t )\hotimes X(v'_r-1) \to X(v_r-1)$,  and for $i=|v_r|$ it is the pushforward by the adjoint of the composition 
\[
 \KK(v_t )\hotimes X(v'_r-1)\stackrel{id\otimes (-)^{\sigma}}{\to } \KK(v_t )\hotimes X(v'_r-1)\stackrel{(-\circ_1-)}{\to} X(v_r-1)
\] where $\sigma$ is the transposition of the first $|v'_r|-|v_t|$ and the last $|v_t|-1$ letters. The functors $\{\FPsi^n\}_n$ is compatible with the inclusion $i_n:\Psi^o_{n+2}\to \Psi^o_{n+3}$. Precisely speaking, There exists an obviously defined natural isomorphism $j_n:{\FPsi}^nX\cong {\FPsi}^{n+1}X|_{\Psi_{n+2}^o}$ because the inclusion does not change $|v_r|$. We define a functor $\FPsi X:\Psi^o\to \SP$ by $\FPsi X(T)$ being the colimit of the sequence $\FPsi^nX(T)\stackrel{\cong}{\to} \FPsi^{n+1}X(T)\stackrel{\cong}{\to}\FPsi^{n+2}X(T)\stackrel{\cong}{\to}\cdots$.
\item We shall define a functor $\omega:(\Psi_{n+2}^o)^{op}\to \CAT$. We set $ \omega (T)=\GG(\,|v_r|-1)^+$. For  the contraction $T\to T'$ of an edge $e$, we define a map $e^*:\underline{v'_r-1}\to \underline{v_r-1}$ as follows. If $e$ is a non-root edge, $e^*$ is the identity. If $e$ is the $i$-th root edge, for $1\leq i\leq |v_r|-1$, $e^*$ is the order-preserving surjection with $e^*(j)=i$ ($i\leq j\leq i+|v_t|-1$).
 For $i=|v_r|$, $e^*$ is the composition 
\[
\underline{v'_r-1}\stackrel{(-)^\sigma}{\to}\underline{v'_r-1}\stackrel{(e')^*}{\to}\underline{v_r-1} ,\quad \text{where}\quad (e')^*(j)=\left\{
\begin{array}{cc}
1 & (1\leq j\leq |v_t|\, ) \\
j-|v_t|+1 & (\,|v_t|+1\leq j\leq |v_r'|-1), 
\end{array}\right.
\]
and $\sigma $ is given in the previous item. For $G\in \GG(\,|v_r'|-1)^+$, we define an object $e^*(G)\in \GG(\,|v_r|-1)^{+}$ by
\[
e^*(G)=\left\{
\begin{split}
\ \bigl\{\,(e^*(s),e^*(t))\, \mid \,  &   (s,t)\in E(G),\ e^*(s)\not=e^*(t)\,\bigr\} \in \GG(\,|v_r|-1) \\
& \quad \text{if $G\not=*$ and  $(j,k)\not \in E(G)$ for any pair   $j,k\in\{i,\dots, i+|v_t|-1\}$,} \\
* \qquad \qquad &\quad  \text{otherwise}
\end{split}
\right.
\]
Here,   the set $\{i,\dots, i+|v_t|-1\}$ is considered modulo $|v_r'|-1$ if $i=|v_r|$. 
\item We define a category $\TPsi_{n+2}$ as the Grothendieck construction for the (non-lax) functor $\omega$
\[
\TPsi_{n+2}=\int_{\Psi_{n+2}^o}\omega .
\]
An object of $\TPsi_{n+2}$ is a pair $(T,G)$ with $T\in \Psi^o_{n+2}$ and $G\in \omega(T)$. A map $(T,G)\to (T',G')$ is a pair of  a map $e:T\to T'\in \Psi^o_{n+2}$ and a map $G\to e^*(G')\in \omega (T)$. $i_n:\Psi_{n+2}^o\to \Psi_{n+3}^o$ and the identity $\omega(T)=\omega(i_n(T))$ naturally induce a functor $i_n:\TPsi_{n+2}\to \TPsi_{n+3}$. We denote by $\TPsi$ the colimit of the sequence $\{\TPsi_{n+2};i_n\}$.
\item We fix a map $\KK\to \DD_1$ of operads and regard $\TCCM$ as a $\KK$-comodule via this map.
\item We shall define a functor $\THOM_n :(\TPsi_{n+2})^{op}\to \SP$. We set 
\[
\THOM_n(T,G)=\left\{
\begin{array}{cc}
\Map(\KK_T^{nr}, \DeltaT_G) & (G\in \GG(|v_r|-1)\  ) \vspace{1mm}\\
*   & (G=*)
\end{array}\right.
\] For a map $(T,G)\to (T',G')$, we set
\[ 
\begin{split}
\Map \Biggl(\,\prod_{\tiny  v\in T'   } \KK(v ),\ \DeltaT_{G'}\Biggr)   
    &  \to \Map\Biggl(\prod_{\tiny \begin{array}{c} v\in T\\  v\not= v_t\end{array}}\!\!\!\! \KK(v ),\  \Map \Bigl( \KK(v_t ),\ \DeltaT_G \Bigr) \Biggr)\\
    &  \cong \Map \Biggl(\Bigl(\prod_{\tiny \begin{array}{c} v\in T\\ v\not= v_t\end{array}}\!\!\!\! \KK(v )\Bigr)\times \KK(v_t ),\ \DeltaT_G \Biggr)\\
     &  = \Map \Biggl(\,\prod_{\tiny  v\in T } \KK(v ) ,\ \DeltaT_G \Biggr)\ .
\end{split}
\]
Here, the arrow in the first line  is the adjoint of the map $\KK(v_t )\hotimes \DeltaT_{G'} \to \DeltaT_{G}$ which is the composition of the map 
$\KK(v_t )\hotimes \DeltaT_{G'} \to \DeltaT_{e^*(G')}$ defined in view of Lemma \ref{Ldiagonalinclusion} and the inclusion  $\DeltaT_{e^*G'}\subset \DeltaT_{G}$ coming from $G\subset e^*(G')$.
The collection $\{\THOM_n\}_n$ naturally induces a functor $\THOM:\TPsi \to \SP$ with natural isomorphism $\THOM|_{\TPsi_{n+2}}\cong \THOM_n$.
\item Let $\MM$ be a model category. For a functor $X:(\TPsi)^{op}\to \MM$, we define a functor 
\[
\fcolim\,  X:(\Psi^o)^{op}\to \MM \, ,
\] called  the {\em fiberwise colimit of $X$} as follows. Abusing notations, for $T\in \Psi^o$, denote by $\omega(T)$ by the full subcategory $\{(T,G)\mid G\in \omega(T)\}$ of $\TPsi$. and by $X_T$ the restriction of $X$ to $\omega(T)$. We set 
$\fcolim\, X(T)=\underset{\omega(T)}{\colim}\,X_T$. For an edge contraction $e:T\to T'$, the map $e^*:\fcolim \,X(T')=\colim\, X_{T'}\to \colim\, X_T=\fcolim \,X(T)$ is induced by $X(T',G')\to X(T,e^*(G'))\to \colim X_T$, where the first map is induced by the map $(e,id):(T,e^*G')\to (T',G')$. This construction is natural for natural transformations and defines a functor $\fcolim:\FUN((\TPsi)^{op},\MM)\to \FUN((\Psi^o)^{op},\MM)$. We define the {\em fiberwise constant functor} $\fc:\FUN((\Psi^o)^{op},\MM) \to \FUN((\TPsi)^{op},\MM)$ by $\fc Y(T,G)=Y(T)$   

\end{itemize}
\end{defi}
$\Psi_n^o$ is equivalent to the category denoted by the same symbol in Definition 4.12 of \cite{sinha} as categories. \\

\noindent {\bf Notation} : In the rest of the paper, we omit the notation $(-)^{op}$ under $(\mathrm{ho})\colim$. For example, $\underset{\Psi^o}{\hocolim}$ denotes $\underset{(\Psi^o)^{op}}{\hocolim}$ \\

A proof of the following lemma is a standard routine work.
\begin{lem}\label{Lquillenpair} Let $\MM$ be a cofibrantly generated model category. 
\begin{enumerate}
\item The pair $(\fcolim, \fc )$ is a Quillen adjoint pair.
\item The restriction 
\[
\FUN ((\TPsi)^{op},\MM)\to \FUN(\omega(T)^{op},\MM),\quad X\mapsto X_T
\] preserves weak equivalences and cofibrations. In particular, the natural map
$\underset{\omega(T)}{\hocolim}\ X_T\to \LL \fcolim\, X(T)\in \Ho(\MM)$ is an isomorphism.
\item For any functor $X\in\FUN((\TPsi)^{op},\MM)$, there is a natural isomorphism in $\Ho(\MM)$
\[
\underset{\Psi^o}{\hocolim}\ {\LL\fcolim\ X}\cong \underset{\TPsi}{\hocolim}\  X.
\]
\end{enumerate}
\hspace{\fill}{\qedsymbol}
\end{lem}
\begin{thm}\label{Tthomcolimit}
\begin{enumerate}
\item There exists an isomorphism in $\Ho(\FUN((\Psi^o)^{op},\SP))$
\[
(\mathcal{G}\circ\mathcal{F})^*(\CC^\bullet \langle[M]\rangle)^\vee \cong \LL \fcolim\ \THOM\, .
\]
\item There exists an isomorphism in $\Ho(\CH_{\kk})$
\[
C^*(\Emb(S^1,M))\cong \underset{\TPsi}{\hocolim}\  C^S_*\circ \THOM\, .
\]
\end{enumerate}
\end{thm}
\begin{proof}
By definition, $\CCM(n)=\underset{G\in \omega(T)}{\colim}\ \DeltaT_G$. We shall show the natural map 
\[
\underset{G\in \omega(T)}{\hocolim}\ \DeltaT_G\to \underset{G\in \omega(T)}{\colim}\ \DeltaT_G =\CCM(n) \in \Ho (\SP )
\]
is an isomorphism. By abuse of notations, we denote by $P_\nu(N_0)$ the subcategory of $\omega(T)$ consisting of non-empty graphs, which is actually isomorophic to $P_\nu(N_0)$ for $N_0=\# \{(i,j)\mid i,j\in \underline{n},\ i<j\}$. Clearly the functor $P_\nu(N_0)\ni G\mapsto \DeltaT_G\in \SP$ satisfies the assumption of the part 2 of Lemma \ref{Lcheckfunctor} so the natural map $\underset{P_\nu(N_0)}{\hocolim}\,\DeltaT_G\to \underset{P_\nu (N_0)}{\colim}\,\DeltaT_G$ is an isomorphism. 
As $\CCM(n)$ is a cofiber of the natural map $\underset{P_\nu(N_0)}{\colim}\ \DeltaT_G \to \TCCM$, which is also a homotopy cofiber, we have the assertion.
We define a natural transformation $\THOM \to \fc \FPsi\CCM$ by the pushforward by the constant map $\DeltaT_G\to \{*\}\subset \CCM(v_r-1)$ for $G\not=\emptyset \in \omega(T)$, and by the quotient map $\DeltaT_{\emptyset}\to \CCM(v_r-1)$ for $G=\emptyset$. By the assertion and the part 2 of Lemma \ref{Lquillenpair}, the derived adjoint of the natural transformation $\LL\fcolim \ \THOM \to \FPsi\CCM$ is an isomorphism in $\Ho(\FUN((\Psi^o)^{op},\SP)$. It is clear that $\FPsi$ preserves weak equivalences so by Theorem \ref{TAtiyahdual}, we have isomorphisms in $\Ho(\FUN((\Psi^o)^{op},\SP))$
\[
\FPsi(\CSINHA^\vee) \cong  \FPsi\CCM\cong \LL\fcolim\ \THOM .
\]
We define a natural transformation $(\mathcal{G}\circ \mathcal{F})^*(\CC^\bullet\langle [M]\rangle^\vee) \to \FPsi(\CSINHA^\vee)$ by the inclusion  $\CC^{|v_r|-2}\langle[M]\rangle=\CSINHA(v_r-1)\subset \Map(\KK_T^{nr},\ \CSINHA(v_r-1))$ onto constant maps. This is clearly weak equivalence, so  we have proved the part 1.\\
\indent For the part 2, since the functor $C^S_*$ preserves homotopy colimits (of strongly semistable spectra), by the part 1,  Lemma \ref{Lquillenpair}, (3), and Lemma \ref{Lcochaintheory}, we have isomorphisms in $\Ho(\CH_{\kk})$ 
\[
\underset{\Psi^o}{\hocolim}\ (\mathcal{G}\circ \mathcal{F})^*C^S_*(\CC^\bullet\langle [M]\rangle^\vee)\cong \underset{\Psi^o}{\hocolim}\ \LL\fcolim\ C_*^S\circ \THOM \cong \underset{\TPsi}{\hocolim}\ C_*^S\circ \THOM .
\]
By Lemma \ref{Lcochaintheory}, Theorem \ref{Tconvergence}, and the fact that $\mathcal{G}\circ\mathcal{F}:(\Psi^o)^{op}\to \Delta^{op}$ is (homotopy) right cofinal  (see Proposition 4.15 and Theorem 6.7 of \cite{sinha}), we have  isomorphisms in $\Ho(\CH_{\kk})$
\[
\begin{split}
C^*(\Emb (S^1,M))  &    \cong \underset{\Delta}{\hocolim}\ C^*(\CC^\bullet\langle [M]\rangle) \\
 &\cong \underset{\Delta}{\hocolim}\ C^S_*(\CC^\bullet\langle [M]\rangle^\vee)\cong  \underset{\Psi^o}{\hocolim}\ (\mathcal{G}\circ \mathcal{F})^*C^S_*(\CC^\bullet\langle [M]\rangle^\vee) .
 \end{split}
\]
Thus, we have an isomorphism
 $C^*(\Emb (S^1,M))\cong \underset{\TPsi}{\hocolim}\,C_*^S\circ \THOM$.  

\end{proof}

\end{document}